\documentclass[11pt]{article}

\usepackage[a4paper,margin=1in]{geometry}
\usepackage{amsmath,amssymb,amsthm,mathtools}
\usepackage{array}
\usepackage{enumitem}
\usepackage{hyperref}
\usepackage{microtype}
\usepackage{bm}

\hypersetup{colorlinks=true,linkcolor=blue,citecolor=blue,urlcolor=blue}

\newtheorem{theorem}{Theorem}[section]
\newtheorem{proposition}[theorem]{Proposition}
\newtheorem{lemma}[theorem]{Lemma}
\newtheorem{corollary}[theorem]{Corollary}
\theoremstyle{definition}
\newtheorem{assumption}[theorem]{Assumption}
\newtheorem{condition}[theorem]{Condition}
\newtheorem{definition}[theorem]{Definition}
\newtheorem{criterion}[theorem]{Criterion}
\theoremstyle{remark}
\newtheorem{remark}[theorem]{Remark}

\newcommand{\E}{\mathbb E}
\newcommand{\R}{\mathbb R}
\newcommand{\Var}{\operatorname{Var}}
\newcommand{\Cov}{\operatorname{Cov}}
\newcommand{\cum}{\operatorname{cum}}
\newcommand{\Tr}{\operatorname{Tr}}
\newcommand{\spec}{\operatorname{Spec}}
\newcommand{\diag}{\operatorname{diag}}
\newcommand{\scal}{\mathrm{sc}}
\newcommand{\ii}{\mathrm i}
\newcommand{\dd}{\mathrm d}
\newcommand{\op}{\mathrm{op}}
\newcommand{\HS}{\mathrm{HS}}

\newcommand{\one}{\mathbf 1}
\newcommand{\abs}[1]{\lvert #1\rvert}
\newcommand{\norm}[1]{\lVert #1\rVert}

\title{Linear Spectral Statistics for Entrywise-Transformed Spiked Wigner Matrices under Shifted $L^4$ Profile Admissibility}
\author{Tsz-Kin Chan\footnote{Department of Mathematical Sciences, KAIST, Daejeon, 34141, Korea \newline email: \texttt{tkchan@kaist.ac.kr}} and Ji Oon Lee\footnote{Department of Mathematical Sciences, KAIST, Daejeon, 34141, Korea \newline email: \texttt{jioon.lee@kaist.edu}}}
\date{\today}

\begin{document}

\maketitle
\begin{abstract}
We prove a derivative-free analytic linear spectral statistics theorem for
entrywise-transformed rank-one spiked Wigner matrices with a general
microscopic noise law.  The transform is required to be centered,
variance-normalized, and admissible under the small translations generated by
the spike: its shifted mean and second moment have first- and second-order
profiles, while its centered shifted fourth cumulants and fourth tails are
stable.  At every microscopic shift we construct an explicit uniformly
bounded three-point variable matching the first four centered moments of the
target entry exactly.  A generalized-Wigner LSS theorem applies to the
resulting bounded triangular array, and a global Fourier--Duhamel derivative
estimate transfers the analytic statistic back to the rough transform without
a common truncation, coefficient-stability assumption, or local law.
The order-one mean consists of the homogeneous Wigner bias, a rank-one
Woodbury response, a zero-diagonal correction, and a quadratic
variance-profile response.  The centered covariance is the standard
zero-diagonal real-Wigner covariance with fourth-cumulant parameter
$\kappa_4^{f,\nu}$.  We distinguish the bulk contour statistic from the full
trace in the supercritical regime and show that the separated outlier adds
exactly $\varphi(\theta+\theta^{-1})$ to the full-trace centering.  As a
self-contained consequence, Gaussian noise with
$f\in L^{4+\epsilon}(\gamma)$ satisfies the theorem without differentiability
of $f$; the resulting bulk and full-trace corollary has explicit Hermite
coefficients and includes every centered, variance-normalized
polynomial-growth transform.
Concrete likelihood-ratio, smooth-transform, bounded rough-transform, and
atomic criteria are provided, together with obstructions showing that bare
$L^4(\nu)$ is insufficient.
\end{abstract}

\tableofcontents
\section{Introduction}

\subsection{Model and contribution}

In this paper, we study linear spectral statistics (LSS) for an entrywise-transformed rank-one spiked Wigner model.  The LSS are observables of the form
\[
        \Tr \varphi(A_N)=\sum_{i=1}^N \varphi(\lambda_i(A_N)),
\]
where \(\lambda_1(A_N),\ldots,\lambda_N(A_N)\) are the eigenvalues of a random matrix \(A_N\).  In the simplest homogeneous Wigner case, one expects that
\[
        \Tr\varphi(A_N)-N\int \varphi\,\dd\rho_{\scal}
\]
converges, after the correct order-one centering, to a Gaussian random variable.  The present model contains two additional features: a rank-one deterministic spike and an entrywise nonlinear transform, possibly of low regularity.

The microscopic noise variable is now allowed to have a general Wigner law.  Let \(\zeta_{ij}\) be independent copies of a real random variable \(\zeta\sim\nu\), normalized by
\[
        \E\zeta=0,\qquad \E\zeta^2=1.
\]
The normalized off-diagonal entries before transformation are
\[
        \zeta_{ij}+\sqrt{\lambda N}\,x_i x_j=:\zeta_{ij}+s_{ij},
\]
where \(\lambda\ge0\) and \(x=(x_1,\ldots,x_N)\) is a deterministic delocalized spike vector.  The observed zero-diagonal matrix after entrywise transformation is
\[
        M_N^f(i,j)=N^{-1/2}f(\zeta_{ij}+s_{ij}),\qquad i<j.
\]
This is the Wigner-noise version of the transformed spiked Wigner model. Correspondingly, the Gaussian-noise version refers to the case when \(\nu=\gamma\) is the standard Gaussian law.

The transformed model is a natural generalization of spiked Wigner models used in weak detection, nonlinear PCA, and signal-plus-noise matrix models \cite{Feldman2025,GuionnetKoKrzakalaMergnyZdeborova2023,MoniriHassani2024,LeeLee2025}.  The non-Gaussian Wigner case is especially natural: in transformed-spiked inference, non-Gaussian noise is often exactly the reason for applying an entrywise score or nonlinear transformation \cite{PerryWeinBandeiraMoitra2018,ChungLee2022}.

The main objective is to analyze how the spike and the transform jointly affect the bulk LSS of \(M_N^f\) when \(f\) is not assumed smooth.  In the Gaussian special case, the shifted profiles of \(f(\xi+s)\) are derived by the Gaussian likelihood-ratio identity
\[
        \E h(\xi+s)=\E\bigl[h(\xi)e^{s\xi-s^2/2}\bigr].
\]
For general Wigner noise there is no universal identity of this form.  Hence the Gaussian Hermite coefficient layer is replaced by a profile-admissibility condition on
\[
        m_f(s)=\E_\nu f(\zeta+s),\qquad
        q_f(s)=\E_\nu f(\zeta+s)^2.
\]
The required coefficients are defined by the expansions
\[
        m_f(s)=a_1s+o(s),\qquad
        q_f(s)=1+b_1s+b_2s^2+o(s^2).
\]
Here and throughout, $f$ denotes the possibly rough entry transform, whereas
$\varphi$ denotes the analytic spectral test function.  These two regularity
questions are independent: the theorem lowers the regularity imposed on $f$
but not the analytic regularity imposed on $\varphi$.

These assumptions are the low-regularity analogue of the score, density, smoothness, and moment assumptions used in the transformed-spiked Wigner literature.  They are stronger than bare \(f\in L^4(\nu)\), but weaker in spirit than requiring pointwise differentiability of \(f\) in the Gaussian case.  A central point of this article is that these target-profile assumptions already suffice for an exact bounded reduction.  At each microscopic shift we replace the centered target entry by an explicit three-point variable with the same first four centered moments.  This removes every lower-order term in the edgewise comparison and avoids both a common truncated transform and coefficient-stability assumptions.

Under these hypotheses, the deterministic LSS mean has three model-specific channels:
\[
        m_\varphi^{\mathrm{rank1}},\qquad
        c_\varphi^{\mathrm{diag0}},\qquad
        c_\varphi^{\mathrm{quad-var}}.
\]
The rank-one channel comes from the deterministic mean profile, the zero-diagonal channel comes from the convention that the diagonal is removed, and the quadratic variance-profile channel comes from the spike-induced second-order change in the entry variance.  The leading centered covariance remains the standard real-Wigner analytic LSS covariance, with fourth-cumulant parameter \(\kappa_4^{f,\nu}\).

The main contributions are therefore:
\begin{enumerate}[label=(\roman*),leftmargin=2.5em]
\item a shifted-profile formulation for general Wigner noise and rough entry
transforms;
\item an explicit bounded three-point completion matching the first four
centered moments at every microscopic shift;
\item an order-one LSS centering that separates the homogeneous, rank-one,
zero-diagonal, and quadratic variance-profile responses; and
\item a derivative-free Gaussian corollary for every centered and
variance-normalized $f\in L^{4+\epsilon}(\gamma)$.
\end{enumerate}

The Gaussian instance is stated explicitly rather than left as an implicit
substitution into the general theorem.  If \(\xi_{ij}\) are independent
standard Gaussian variables and
\[
 M_N^f(i,j)=N^{-1/2}f\!\left(\xi_{ij}+\sqrt{\lambda N}\,x_ix_j\right),
 \qquad i<j,
\]
with zero diagonal, then every centered and variance-normalized
\(f\in L^{4+\epsilon}(\gamma)\) satisfies W1--W5.  The effective spike and
fourth-cumulant parameters are
\[
 \theta_f=\sqrt\lambda\,\E[\xi f(\xi)],
 \qquad
 \kappa_4^{f,\gamma}=\E f(\xi)^4-3.
\]
Consequently the general bulk theorem and both of its full-trace consequences
give a derivative-free Gaussian-noise LSS CLT, including the deterministic
contribution of the separated outlier when \(|\theta_f|>1\).  The complete
statement and all Gaussian coefficients are recorded in
Corollary~\ref{cor:gaussian-LSS-main}.

\subsection{Main result in informal form}
\label{subsec:intro-main-result}

Suppose that $(f,\nu)$ satisfies the shifted profile conditions W1--W5 and
that the spike is delocalized and weakly balanced as in
Assumption~\ref{ass:spike}.  For every analytic bulk test function $\varphi$,
the principal conclusion is
\begin{equation}
 \mathcal L_{M_N^f}^{\Gamma}(\varphi)
 -N\int_{-2}^2\varphi(t)\rho_{\scal}(t)\,\dd t
 -m_{\varphi,\Gamma}^{f,\nu}
 \ \Longrightarrow\ G_\varphi,
 \label{eq:intro-informal-clt}
\end{equation}
where $G_\varphi$ is centered Gaussian and
\begin{equation}
 \Cov(G_\varphi,G_\psi)
 =\mathcal V_{\kappa_4^{f,\nu}}^\Gamma(\varphi,\psi).
 \label{eq:intro-informal-covariance}
\end{equation}
The centering $m_{\varphi,\Gamma}^{f,\nu}$ is the sum of the homogeneous
zero-diagonal Wigner bias, the rank-one response, the zero-diagonal
correction, and the quadratic variance-profile response.  If
$|\theta_{f,\nu}|\leq1$, the bulk statistic in
\eqref{eq:intro-informal-clt} equals the full trace with probability tending
to one.  If $|\theta_{f,\nu}|>1$, the full-trace centering contains the
additional deterministic contribution
\[
 \varphi(\rho_\theta),\qquad
 \rho_\theta=\theta_{f,\nu}+\theta_{f,\nu}^{-1}.
\]
The precise hypotheses, contour conditions, and formulas are stated in
Theorem~\ref{thm:bulk-main-package},
Corollary~\ref{cor:subcritical-full-package}, and
Theorem~\ref{thm:R6-supercritical-full}.

\subsection{Related works}

\paragraph*{Position of the present result.}
The contribution here is the combination of general Wigner noise, a
possibly rough entry transform described through shifted profiles, exact
edgewise four-moment completion, a spike-generated variance profile, and an
explicit bulk-LSS centering that also records the contribution of a separated
outlier.  For Gaussian noise, the likelihood-ratio identity verifies the
profiles for every centered and normalized
$f\in L^{4+\epsilon}(\gamma)$, without introducing a separate Gaussian
comparison argument.  The theorem does not lower the regularity threshold on
the spectral test function, and it does not address critical BBP or edge-scale
fluctuations.

\paragraph*{Classical and low-regularity LSS.}
Central limit theorems for linear spectral statistics of invariant, sample
covariance, and Wigner ensembles go back, among others, to
Johansson~\cite{Johansson1998}, Bai--Silverstein~\cite{BaiSilverstein2004},
Lytova--Pastur~\cite{LytovaPastur2009}, and
Shcherbina~\cite{Shcherbina2011}.  Regularity of the spectral test function
has subsequently been lowered by Sosoe--Wong~\cite{SosoeWong2013} and, for
bulk-supported statistics under smooth entry laws, by
Landon--Sosoe~\cite{LandonSosoe2022}; multi-resolvent methods give another
route to refined fluctuation estimates~\cite{CipolloniErdosSchroeder2022}.
This line of work concerns regularity of the spectral observable
\(\varphi\).  The low-regularity issue in the present paper is different:
\(\varphi\) is analytic, while the entry transform \(f\) need not possess
pointwise derivatives.

\paragraph*{Variance profiles and Wigner-type matrices.}
Linear statistics for matrices with nonconstant variance profiles have been
studied by Adhikari--Jana--Saha~\cite{AdhikariJanaSaha2021VarianceProfileLES},
while Li--Xu give global and mesoscopic fluctuation formulas for generalized
Wigner matrices~\cite{LiXu2021GeneralizedWignerLES}.  The local spectral
theory of general Wigner-type matrices and the stability theory of the matrix
Dyson equation were developed by Ajanki--Erd\H{o}s--Kr\"uger and related
works~\cite{AEK2017WignerType,AEK2019MDE,Erdos2019MDE}; more recent
mesoscopic Wigner-type statistics are treated by Riabov~\cite{Riabov2025}.
Here the variance profile is not an arbitrary macroscopic input.  It is a
weak, spike-generated perturbation of the flat profile.  We use the
generalized-Wigner LSS theory as a backend and separately compute the
order-one deterministic response of its linear and quadratic channels.

\paragraph*{Finite-rank deformations and deformed-Wigner LSS.}
The statistical and edge-transition origins include Johnstone's spiked
covariance model~\cite{Johnstone2001}, the Baik--Ben Arous--P\'{e}ch\'{e}
transition~\cite{BaikBenArousPeche2005}, and the spiked sample covariance
analysis of Baik--Silverstein~\cite{BaikSilverstein2006}.  Finite-rank
deformations of Wigner and related matrices have been analyzed by
F\'{e}ral--P\'{e}ch\'{e}~\cite{FeralPeche2007},
Capitaine--Donati-Martin--F\'{e}ral~\cite{CapitaineDonatiMartinFeral2009},
Benaych-Georges--Nadakuditi~\cite{BenaychGeorgesNadakuditi2011},
Pizzo--Renfrew--Soshnikov~\cite{PizzoRenfrewSoshnikov2013},
Knowles--Yin~\cite{KnowlesYin2013}, and
Bloemendal--Vir\'{a}g~\cite{BloemendalVirag2013}.  Their principal outputs
concern outliers, eigenvectors, edge fluctuations, and BBP transitions.
For linear statistics, Ji--Lee~\cite{JiLee2020} and
Dallaporta--F\'{e}vrier~\cite{DallaportaFevrier2025} prove CLTs for deformed
Wigner matrices and identify deformation-dependent deterministic equivalents
or fluctuation channels.  In those models a deformation is added to a Wigner
matrix.  In the present model the spike is inserted before the nonlinear
entry transform and therefore changes the entry mean, variance profile, and
higher cumulants simultaneously.

\paragraph*{LSS and spectral methods for inference.}
PCA and related spectral estimators extract edge eigenvalues and eigenvectors;
representative spiked-model results include
Paul~\cite{Paul2007}, Nadler~\cite{Nadler2008},
Onatski~\cite{Onatski2012}, Fan--Wang~\cite{FanWang2017}, and
Donoho--Gavish--Johnstone~\cite{DonohoGavishJohnstone2018}.  LSS instead
aggregate the full spectrum.  Chung--Lee prove an LSS CLT for rank-one
spiked-Wigner weak detection and show that a score-type entrywise transform
can improve testing under known non-Gaussian noise~\cite{ChungLee2022};
Jung--Chung--Lee treat related rectangular detection problems
\cite{JungChungLee2021}.  For spiked sample covariance and related Hermitian
ensembles, Passemier--McKay--Chen~\cite{PassemierMcKayChen2015} and
Wang--Silverstein--Yao~\cite{WangSilversteinYao2014} derive LSS CLTs or
centering expansions with order-one spike corrections.

\paragraph*{Entrywise nonlinear spiked models.}
Feldman treats elementwise-transformed spiked matrices, including highly
nonlinear or discontinuous transformations, with leading spectral behavior
and PCA as the main observables~\cite{Feldman2025}.
Guionnet--Ko--Krzakala--Mergny--Zdeborov\'{a} derive effective spiked-Wigner
reductions and phase transitions indexed by generalized information
coefficients~\cite{GuionnetKoKrzakalaMergnyZdeborova2023}; their qualitative
rank-one equivalence allows rough transforms under regular noise, while
quantitative remainder bounds use additional smoothness.  Moniri--Hassani
develop signal-plus-noise decompositions for nonlinear spiked models
\cite{MoniriHassani2024}, and Lee--Lee identify BBP-type largest-eigenvalue
fluctuations for transformed spiked Wigner matrices~\cite{LeeLee2025}.
These main results concern spectral decomposition, eigenvalues, or
eigenvectors.  The present paper instead follows the order-one bulk LSS,
including the deterministic mean channels created by the transform.

\paragraph*{The fourth-moment boundary.}
The standard Wigner LSS covariance is not expected to persist unchanged once
the transformed entries fall below the fourth-moment threshold.
Benaych-Georges--Maltsev obtain a tail-index-dependent fluctuation scale for
half-heavy-tailed matrices~\cite{BenaychGeorgesMaltsev2016}, and
Lodhia--Maltsev identify the corresponding nonstandard covariance kernel
\cite{LodhiaMaltsev2023}.  These results explain why W4--W5 are not merely
technical substitutes for bare integrability: they place the present theorem
in the finite-fourth-moment Wigner fluctuation class.

\subsection{Organization of the paper}

Section~\ref{sec:model-main} defines the model and the shifted-profile
assumptions.  Section~\ref{sec:bulk-main-results} introduces the bulk contour
statistic, fixes the explicit homogeneous zero-diagonal mean and covariance,
and states the bulk and full-trace results.  Its Gaussian
subsections first verify profile-admissibility from the Gaussian translation
identity and then state the resulting Gaussian-noise bulk and full-trace LSS
corollary.  Section~\ref{sec:proof-overview} gives the proof architecture.
The four appendices follow that hierarchy.  Appendix~\ref{app:proof-notation}
develops the shifted-tail and spike estimates and the exact four-moment atomic
completion.  Appendix~\ref{sec:R3-package} proves the bounded-array LSS input,
the deterministic responses, and spectral confinement.
Appendix~\ref{sec:R5-package} performs the global four-moment replacement and
adds the separated-outlier contribution.
Appendix~\ref{app:assumption-W-criteria} gives concrete
profile-admissibility criteria, including the detailed Gaussian shifted-tail
estimate, and explains why bare $L^4(\nu)$ is insufficient.
\section{Model and Assumptions}\label{sec:model-main}

\subsection{Model and profile-admissibility}

Let \(\nu\) be a probability law on \(\R\), and let \(\zeta\sim\nu\).  Throughout the paper, \(\zeta_{ij}\), \(1\le i<j\le N\), are independent copies of \(\zeta\).  We assume
\[
        \E\zeta=0,
        \qquad
        \E\zeta^2=1.
\]
The transformed spiked matrix is the real symmetric zero-diagonal matrix \(M_N^f\) defined by
\begin{equation}\label{eq:model}
        M_N^f(i,j)=N^{-1/2}f(\zeta_{ij}+s_{ij}),
        \qquad
        s_{ij}=\sqrt{\lambda N}\,x_i x_j,
        \qquad i<j,
\end{equation}
and \(M_N^f(i,i)=0\).  Here \(\lambda\ge0\) is fixed.

The Gaussian-noise model uses exactly the same normalization and zero-diagonal
convention: it is obtained by taking \(\nu=\gamma\) and writing
\(\zeta=\xi\sim N(0,1)\).  No separate matrix scaling or spike convention is
introduced in the Gaussian specialization.

\begin{condition}[Wigner profile-admissibility]\label{cond:profile-admissibility}
The transform is a fixed Borel measurable representative
$f:\mathbb R\to\mathbb R$, and the microscopic noise law $\nu$ satisfies the
following conditions with $f$.
\begin{enumerate}[label=(W\arabic*)]
\item \textbf{Centering and normalization.}
\[
        f\in L^4(\nu),\qquad
        \E_\nu f(\zeta)=0,\qquad
        \E_\nu f(\zeta)^2=1 .
\]
\item \textbf{Shifted mean profile.} With
\[
        m_f(s):=\E_\nu f(\zeta+s),
\]
there exists $a_1\in\mathbb R$ such that
\[
        \sup_{0<|s|\le\delta}
        \frac{|m_f(s)-a_1s|}{|s|}\longrightarrow0
        \qquad(\delta\downarrow0).
\]
\item \textbf{Shifted second-moment profile.} With
\[
        q_f(s):=\E_\nu f(\zeta+s)^2,
\]
there exist $b_1,b_2\in\mathbb R$ such that
\[
        \sup_{0<|s|\le\delta}
        \frac{|q_f(s)-1-b_1s-b_2s^2|}{s^2}\longrightarrow0
        \qquad(\delta\downarrow0).
\]
Equivalently,
\[
        \Var(f(\zeta+s))
        =1+b_1s+\beta_2^{f,\nu}s^2+o(s^2),\qquad
        \beta_2^{f,\nu}:=b_2-a_1^2.
\]
\item \textbf{Shifted fourth-cumulant stability.} If
\[
        Y_s:=f(\zeta+s)-m_f(s),
\]
then
\[
        \cum_4(Y_s)\longrightarrow
        \kappa_4^{f,\nu}:=\E_\nu f(\zeta)^4-3
        \qquad(s\to0).
\]
\item \textbf{Shifted fourth-moment Lindeberg condition.}
\[
        \lim_{R\to\infty}\limsup_{\delta\downarrow0}
        \sup_{|s|\le\delta}
        \E\Bigl[|Y_s|^4{\bf 1}_{\{|Y_s|>R\}}\Bigr]=0.
\]
\end{enumerate}
\end{condition}

\begin{remark}[Why a representative and shifted profiles are required]
Bare membership in $L^4(\nu)$ specifies only a $\nu$-almost-everywhere
equivalence class, while the matrix in \eqref{eq:model} evaluates $f$ under
the translated laws
\[
 \nu_s(A):=\nu(A-s).
\]
Unless $\nu_s\ll\nu$ for every relevant small shift, two representatives of
the same $L^4(\nu)$ class can produce different matrices.  This occurs, for
example, for atomic noise.  Even when the translated laws are equivalent to
$\nu$, as for Gaussian noise, bare $L^4(\nu)$ need not control shifted fourth
moments or shifted fourth tails.  Accordingly, the model uses a fixed
representative and W2--W5 record the translated information needed by the
theorem.  Concrete sufficient conditions and counterexamples are given in
Appendix~\ref{app:assumption-W-criteria}.
\end{remark}

\begin{assumption}[Spike delocalization and balance]\label{ass:spike}
The deterministic vector \(x=x^{(N)}\in\R^N\) satisfies
\[
        \|x\|_2=1,
        \qquad
        \|x\|_\infty=o(N^{-1/4}),
\]
and the weak balance condition
\[
        \sum_{i=1}^N x_i=O(N^{o(1)}).
\]
Here \(O(N^{o(1)})\) means that, for every fixed \(\epsilon>0\), there is a
constant \(C_\epsilon\) such that
\(
 |\sum_i x_i|\le C_\epsilon N^\epsilon
\)
for all sufficiently large \(N\).
Consequently,
\[
        \max_{i<j}|s_{ij}|=o(1),
        \qquad
        \sum_{i,j}s_{ij}^2=O(N).
\]
\end{assumption}

The exponent $1/4$ in the delocalization condition is the threshold needed
by the present uniform small-shift formulation: since
$|s_{ij}|\leq\sqrt{\lambda N}\|x\|_\infty^2$, it gives
$\max_{i<j}|s_{ij}|=o(1)$.  The weak balance condition has a different role:
it suppresses the order-one linear variance-profile response.  Thus neither
condition is an unexplained generic delocalization hypothesis.

\subsection{Assumption hierarchy and deterministic parameters}
\label{subsec:assumption-hierarchy}

The five profile assumptions have separate roles.  W1 fixes the homogeneous
normalization.  W2 identifies the deterministic rank-one channel.  W3
identifies the linear and quadratic variance profiles.  W4 fixes the limiting
fourth-cumulant parameter.  W5 supplies both the triangular-array Lindeberg
estimate and a uniformly bounded exact four-moment surrogate.  The spike
assumption controls the microscopic shifts and eliminates the order-one
linear variance response.  No local law for the rough transform is assumed.

The effective coefficients are
\begin{equation}\label{eq:profile-coeffs}
 a_1,\qquad b_1,\qquad
 \beta_2^{f,\nu}:=b_2-a_1^2,\qquad
 \kappa_4^{f,\nu}:=\E f(\zeta)^4-3,
\end{equation}
and the effective rank-one strength is
\begin{equation}\label{eq:theta-nu}
 \theta_{f,\nu}:=\sqrt\lambda\,a_1.
\end{equation}
Uniformly over the microscopic shifts,
\begin{align}
 N^{-1/2}\E f(\zeta+s_{ij})
 &=\sqrt\lambda\,a_1x_ix_j+o(x_ix_j),
 \label{eq:matrix-mean-expansion}\\
 \Var(f(\zeta+s_{ij}))
 &=1+b_1\sqrt{\lambda N}x_ix_j
 +\lambda\beta_2^{f,\nu}Nx_i^2x_j^2
 +o(Nx_i^2x_j^2).
 \label{eq:matrix-variance-profile}
\end{align}

For Gaussian noise, W1 is the explicit centering, variance normalization, and
integrability requirement on $f$.  The Gaussian likelihood-ratio expansion
makes W2 and W3 automatic, while the additional
$L^{4+\epsilon}(\gamma)$ integrability in
Lemma~\ref{lem:gaussian-profile-main} supplies W4 and W5.  In this case
\begin{equation}\label{eq:Gaussian-coefficients-summary}
 \begin{aligned}
 a_1&=\E[\xi f(\xi)],
 &\qquad b_1&=\E[\xi f(\xi)^2],\\
 b_2&=\frac12\E[(\xi^2-1)f(\xi)^2],
 &\qquad \beta_2^{f,\gamma}
 &=\frac12\E[(\xi^2-1)f(\xi)^2]
  -\bigl(\E[\xi f(\xi)]\bigr)^2.
 \end{aligned}
\end{equation}
Moreover,
\begin{equation}\label{eq:Gaussian-kappa-theta-summary}
 \kappa_f:=\kappa_4^{f,\gamma}=\E f(\xi)^4-3,
 \qquad
 \theta_f:=\theta_{f,\gamma}
 =\sqrt\lambda\,\E[\xi f(\xi)].
\end{equation}
\section{Bulk Contours and Main Results}
\label{sec:bulk-main-results}

\subsection{Result structure}

The main theorem has one bulk statement and two full-trace consequences.  The
bulk statistic always excludes a separated outlier.  The relation to the full
trace is
\[
\begin{array}{c|c}
 |\theta_{f,\nu}|\leq1
 & \mathcal L_{M_N^f}^{\Gamma}(\varphi)=\Tr\varphi(M_N^f)
   \quad\text{with probability }1-o(1),\\[2mm]
 |\theta_{f,\nu}|>1
 & \Tr\varphi(M_N^f)
   =\mathcal L_{M_N^f}^{\Gamma}(\varphi)+\varphi(\rho_\theta)+o_{\mathbb P}(1).
\end{array}
\]
The next three subsections define the contour statistic, the homogeneous
Wigner mean and covariance, and the profile-dependent centering.  The formal
bulk theorem and its two full-trace forms are then stated together in
Subsection~\ref{subsec:bulk-full-results}.

\subsection{Contour functional calculus}

Let \(\Gamma\) be a positively oriented simple closed contour, and denote its
bounded interior by \(D_\Gamma\).  The convention in the article is
\begin{equation}\label{eq:resolvent-convention-package}
 G_A(z):=(A-zI)^{-1}.
\end{equation}
This convention introduces a minus sign in the Cauchy representation of a
linear spectral statistic.

\begin{definition}[Bulk contour statistic]\label{def:bulk-stat-package}
Let \(A\) be a real symmetric matrix and let \(\varphi\) be analytic on a
neighborhood of \(\overline{D_\Gamma}\).  Define
\begin{equation}\label{eq:bulk-stat-package}
 \mathcal L_A^\Gamma(\varphi)
 :=\sum_{j=1}^N
 \varphi(\lambda_j(A))\one_{\{\lambda_j(A)\in D_\Gamma\}}.
\end{equation}
If \(\spec(A)\cap\Gamma=\varnothing\), then
\begin{equation}\label{eq:bulk-contour-formula-package}
 \mathcal L_A^\Gamma(\varphi)
 =-\frac{1}{2\pi\ii}
 \oint_\Gamma
 \varphi(z)\Tr(A-zI)^{-1}\,\dd z.
\end{equation}
\end{definition}

\begin{proof}
Diagonalize \(A\).  Since
\(
 (\lambda-z)^{-1}=-(z-\lambda)^{-1},
\)
the residue of \((\lambda-z)^{-1}\) at \(z=\lambda\) is \(-1\).
The residue theorem therefore gives
\[
 -\frac{1}{2\pi\ii}\oint_\Gamma
 \frac{\varphi(z)}{\lambda-z}\,\dd z
 =\varphi(\lambda)\one_{\{\lambda\in D_\Gamma\}}.
\]
Summing over the eigenvalues proves \eqref{eq:bulk-contour-formula-package}.
\end{proof}

\begin{remark}[Real and complex test functions]
The main CLTs are stated for analytic test functions that are real-valued on
the real axis.  This makes every LSS and every covariance in the theorem
real-valued.  For a complex-valued analytic test function, shrink to a
conjugation-invariant neighborhood of the relevant real spectral set and put
\[
 \varphi_{\rm R}(z):=\frac{\varphi(z)+\overline{\varphi(\overline z)}}2,
 \qquad
 \varphi_{\rm I}(z):=\frac{\varphi(z)-\overline{\varphi(\overline z)}}{2\ii}.
\]
Both functions are holomorphic and real-valued on the real axis, and
\(\varphi=\varphi_{\rm R}+\ii\varphi_{\rm I}\) there.  Applying the joint real
theorem to these two functions gives the complex statement; equivalently, all
formulas below extend complex bilinearly.
\end{remark}

\subsection{Semicircle and homogeneous Wigner conventions}

The semicircle density and measure are
\begin{equation}\label{eq:semicircle-density}
 \rho_{\scal}(t):=\frac{1}{2\pi}\sqrt{4-t^2}\,
 \one_{[-2,2]}(t),
 \qquad
 \dd\rho_{\scal}(t):=\rho_{\scal}(t)\,\dd t.
\end{equation}
Its Stieltjes transform, in the resolvent convention
\eqref{eq:resolvent-convention-package}, is
\begin{equation}\label{eq:semicircle-m-package}
 m(z):=\int_{-2}^2\frac{\rho_{\scal}(t)}{t-z}\,\dd t,
 \qquad
 m(z)^2+zm(z)+1=0,
 \qquad
 m(z)\sim-z^{-1}\quad(|z|\to\infty).
\end{equation}

The homogeneous zero-diagonal reference convention is the following.  Write
\begin{equation}\label{eq:base-Wigner-normalization}
 W_N(i,j)=N^{-1/2}\xi_{ij}\quad(i<j),
 \qquad W_N(i,i)=0,
\end{equation}
where the variables above the diagonal are independent and centered, with
\begin{align}
 \E\xi_{ij}^2&=1\qquad (i<j),
 \label{eq:base-Wigner-variance}\\
 \max_{i<j}\abs{\cum_4(\xi_{ij})-\kappa}&\longrightarrow0,
 \label{eq:base-Wigner-kappa}\\
 \lim_{R\to\infty}\limsup_{N\to\infty}\max_{i<j}
 \E\!\left[|\xi_{ij}|^4\one_{\{|\xi_{ij}|>R\}}\right]&=0.
 \label{eq:base-Wigner-Lindeberg}
\end{align}
Thus ``unit variance'' always refers to the normalized variables \(\xi_{ij}\),
whereas every off-diagonal matrix entry has variance exactly \(N^{-1}\).

For a function \(\varphi\) analytic near \(\overline{D_\Gamma}\), define the
explicit zero-diagonal real-Wigner bias by
\begin{align}
 m_{\varphi,\Gamma}^{\rm base}(\kappa)
 &:=-\frac{1}{2\pi\ii}\oint_\Gamma
 \varphi(z)m'(z)m(z)^3
 \left(\frac{1}{1-m(z)^2}+\kappa\right)\,\dd z
 \notag\\
 &\quad+\frac{1}{2\pi\ii}\oint_\Gamma
 \varphi(z)\frac{m(z)^3}{1-m(z)^2}\,\dd z.
 \label{eq:base-mean-definition}
\end{align}
The first line is the fluctuation bias of the stochastically normalized flat
array.  The second line is the finite-flat Dyson correction caused by the zero
diagonal, whose row sum is \(1-N^{-1}\).  Lemma~\ref{lem:R3-flat-endpoint}
derives both terms directly from the Li--Xu formulas.

The corresponding zero-diagonal covariance kernel is
\begin{align}
 \mathcal K_{\rm Wig}^{(\kappa,0)}(z,w)
 &:=2m'(z)m'(w)
 \left\{\frac{1}{(1-m(z)m(w))^2}-1\right\}
 \notag\\
 &\quad+2\kappa\,m(z)m'(z)m(w)m'(w),
 \label{eq:Wigner-kernel}\\
 \mathcal V_\kappa^\Gamma(\varphi,\psi)
 &:=\frac{1}{(2\pi\ii)^2}\oint_\Gamma\oint_\Gamma
 \varphi(z)\psi(w)\mathcal K_{\rm Wig}^{(\kappa,0)}(z,w)
 \,\dd z\,\dd w.
 \label{eq:Wigner-covariance}
\end{align}
The subtraction of one in \eqref{eq:Wigner-kernel} is the non-Perron
contribution of the flat zero-diagonal variance profile.  It forces
\(\mathcal V_\kappa^\Gamma(t,\psi)=0\), consistently with
\(\Tr W_N=0\), while for \(\varphi(t)=t^2\) the fourth-cumulant contribution
to the variance is \(2\kappa\).

We use the classical homogeneous theorem of
Lytova--Pastur~\cite{LytovaPastur2009} and the generalized-Wigner extension of
Li--Xu~\cite{LiXu2021GeneralizedWignerLES} in the forms isolated below.

Gaussian microscopic noise does not generally set \(\kappa=0\) after
transformation.  Unless the normalized variable \(f(\xi)\) is itself Gaussian,
the relevant fourth cumulant is \(\kappa_f=\E f(\xi)^4-3\), as used in
\eqref{eq:Gaussian-covariance-main}.  The adjective ``Gaussian'' in
Corollary~\ref{cor:gaussian-LSS-main} refers to the pre-transform variable
\(\xi\), not to the distribution of the transformed entries.

\subsection{Admissible bulk contours}

Fix a contour \(\Gamma\) such that
\begin{enumerate}[label=(G\arabic*)]
\item \(\overline{D_\Gamma}\) contains \([-2,2]\) in its interior;
\item \(\Gamma\) has positive distance from \([-2,2]\);
\item
\begin{equation}\label{eq:pole-avoidance-package}
 \inf_{z\in\Gamma}|1+\theta_{f,\nu}m(z)|>0;
\end{equation}
\item if \(|\theta_{f,\nu}|>1\), then the deterministic outlier location
\begin{equation}\label{eq:outlier-location-package}
 \rho_{\theta}:=\theta_{f,\nu}+\theta_{f,\nu}^{-1}
\end{equation}
lies outside \(\overline{D_\Gamma}\).
\end{enumerate}
Condition (G4) makes \(\mathcal L_{M_N^f}^\Gamma\) a bulk statistic in the
supercritical regime.  It is not the full trace in that regime.

The rank-one response is best defined without selecting a logarithm branch:
\begin{equation}\label{eq:rank-one-direct-package}
 m_{\varphi,\Gamma}^{\rm rank1}(\theta)
 :=\frac{1}{2\pi\ii}
 \oint_\Gamma
 \varphi(z)\frac{\theta m'(z)}{1+\theta m(z)}\,\dd z.
\end{equation}
If \(1+\theta m\) admits a continuous logarithm along \(\Gamma\), integration
by parts gives the equivalent expression
\begin{equation}\label{eq:rank-one-log-package}
 m_{\varphi,\Gamma}^{\rm rank1}(\theta)
 =-\frac{1}{2\pi\ii}
 \oint_\Gamma
 \varphi'(z)\log(1+\theta m(z))\,\dd z.
\end{equation}
The direct formula \eqref{eq:rank-one-direct-package} remains the primary
definition and avoids a hidden winding-number condition.

Retain the zero-diagonal and quadratic variance-profile terms in the form
\begin{align}
 c_{\varphi,\Gamma}^{\rm diag0}
 &:=\frac{\theta_{f,\nu}}{2\pi\ii}
 \oint_\Gamma\varphi'(z)m(z)\,\dd z,
 \label{eq:diag-package}\\
 c_{\varphi,\Gamma}^{\rm quad-var}
 &:=-\frac{\lambda\beta_2^{f,\nu}}{2\pi\ii}
 \oint_\Gamma
 \varphi(z)\frac{m(z)^3}{1-m(z)^2}\,\dd z.
 \label{eq:quad-package}
\end{align}
Let \(m_{\varphi,\Gamma}^{\rm base}(\kappa_4^{f,\nu})\) and
\(\mathcal V_{\kappa_4^{f,\nu}}^\Gamma\) be the explicit homogeneous
zero-diagonal real-Wigner mean and covariance from
\eqref{eq:base-mean-definition} and \eqref{eq:Wigner-covariance}.
Set
\begin{equation}\label{eq:total-mean-package}
 m_{\varphi,\Gamma}^{f,\nu}
 :=m_{\varphi,\Gamma}^{\rm base}(\kappa_4^{f,\nu})
 +m_{\varphi,\Gamma}^{\rm rank1}(\theta_{f,\nu})
 +c_{\varphi,\Gamma}^{\rm diag0}
 +c_{\varphi,\Gamma}^{\rm quad-var}.
\end{equation}

\subsection{Bulk theorem and full-trace consequences}
\label{subsec:bulk-full-results}

The following is the principal theorem.

\begin{theorem}[Bulk LSS for profile-admissible transformed Wigner noise]
\label{thm:bulk-main-package}\label{thm:main}
Assume Condition~\ref{cond:profile-admissibility} and
Assumption~\ref{ass:spike}.  Let \(\Gamma\) satisfy (G1)--(G4), and
let \(\varphi_1,\ldots,\varphi_k\) be analytic on a neighborhood of
\(\overline{D_\Gamma}\) and real-valued on its intersection with
\(\mathbb R\).  Then
\begin{equation}\label{eq:bulk-clt-package}
 \left(
 \mathcal L_{M_N^f}^\Gamma(\varphi_\ell)
 -N\int_{-2}^2\varphi_\ell(t)\,\rho_{\scal}(t)\,\dd t
 -m_{\varphi_\ell,\Gamma}^{f,\nu}
 \right)_{\ell=1}^k
 \Longrightarrow
 (G_{\varphi_\ell})_{\ell=1}^k,
\end{equation}
where the limiting vector is centered Gaussian and
\begin{equation}\label{eq:bulk-cov-package}
 \Cov(G_\varphi,G_\psi)
 =\mathcal V_{\kappa_4^{f,\nu}}^\Gamma(\varphi,\psi).
\end{equation}
In the supercritical regime, \eqref{eq:bulk-clt-package} concerns only the
eigenvalues inside \(D_\Gamma\).
\end{theorem}

\begin{corollary}[Full trace in the non-outlier regime]
\label{cor:subcritical-full-package}
Under the hypotheses of Theorem~\ref{thm:bulk-main-package}, suppose that
\(|\theta_{f,\nu}|\le1\).  Then, with probability tending to one,
\begin{equation}\label{eq:full-equals-bulk-package}
 \mathcal L_{M_N^f}^\Gamma(\varphi)
 =\Tr\varphi(M_N^f).
\end{equation}
Consequently \eqref{eq:bulk-clt-package} is the full analytic LSS CLT in this
regime.
\end{corollary}

\begin{proof}
Propositions~\ref{prop:R4-centered-confinement} and
\ref{prop:R4-outlier-separation} imply
\(\spec(M_N^f)\subset D_\Gamma\) with probability \(1-o(1)\).  On that event,
Definition~\ref{def:bulk-stat-package} sums over all eigenvalues.  Equality on
an event of probability \(1-o(1)\) transfers the limiting distribution.
\end{proof}

\begin{remark}[Bulk versus full trace]
When $|\theta_{f,\nu}|>1$, the statistic in
Theorem~\ref{thm:bulk-main-package} excludes the separated outlier.  The next
theorem restores the full trace by adding its deterministic contribution.
\end{remark}

\begin{theorem}[Full analytic LSS in the supercritical regime]
\label{thm:R6-supercritical-full}
Assume the hypotheses of Theorem~\ref{thm:bulk-main-package} and
$|\theta_{f,\nu}|>1$.  Let $\varphi_1,\ldots,\varphi_k$ be analytic on a
neighborhood of $\overline{D_\Gamma}$ and also on a neighborhood of
$\rho_\theta=\theta_{f,\nu}+\theta_{f,\nu}^{-1}$, and real-valued on the real
parts of those neighborhoods; the two neighborhoods need not be connected.
Then
\begin{equation}\label{eq:R6-full-CLT}
 \left(
 \Tr\varphi_\ell(M_N^f)
 -N\int_{-2}^2\varphi_\ell(t)\rho_{\scal}(t)\,\dd t
 -m_{\varphi_\ell,\Gamma}^{f,\nu}
 -\varphi_\ell(\rho_\theta)
 \right)_{\ell=1}^k
 \Longrightarrow (G_{\varphi_\ell})_{\ell=1}^k,
\end{equation}
with covariance $\mathcal V_{\kappa_4^{f,\nu}}^\Gamma$.
\end{theorem}

\begin{remark}[Full-trace extension convention]
Whenever a full trace is written for a test function specified only on
neighborhoods of the bulk and, in the supercritical case, the outlier, fix an
arbitrary bounded Borel extension to the rest of \(\mathbb R\).  Spectral
confinement shows that two such extensions give the same statistic with
probability \(1-o(1)\), so the limiting statements are independent of this
choice.  This convention makes the full-trace random variables defined on
every outcome.
\end{remark}

\subsection{Gaussian profile admissibility}
\label{subsec:gaussian-profile-main}

Let $\gamma$ denote standard Gaussian measure and let $\xi\sim\gamma$.  The
translation likelihood ratio is
\begin{equation}\label{eq:Gaussian-LR-main}
 \E h(\xi+s)
 =\E\!\left[h(\xi)L_s(\xi)\right],
 \qquad
 L_s(\xi):=e^{s\xi-s^2/2}.
\end{equation}
Thus the abstract shifted profiles can be verified directly from the
integrability of $f$ under one fixed measure.

\begin{lemma}[Gaussian profile admissibility]
\label{lem:gaussian-profile-main}
Suppose that, for some $\epsilon>0$,
\begin{equation}\label{eq:Gaussian-f-assumptions-main}
 f\in L^{4+\epsilon}(\gamma),
 \qquad
 \E f(\xi)=0,
 \qquad
 \E f(\xi)^2=1.
\end{equation}
Then $(f,\gamma)$ satisfies W1--W5.  The coefficients are those in
\eqref{eq:Gaussian-coefficients-summary}, and the effective cumulant and spike
strength are those in \eqref{eq:Gaussian-kappa-theta-summary}.  No pointwise
differentiability of $f$ is required.
\end{lemma}

\begin{proof}
The expansion of $L_s$ through second order, paired respectively with $f$ and
$f^2$, proves W2 and W3.  H\"older's inequality and the Gaussian exponential
moments give uniform shifted fourth-tail integrability, which yields W4 and
W5.  The complete weighted expansion and tail estimate are proved in
Lemma~\ref{lem:Gaussian-L4epsilon-W}.
\end{proof}

\begin{remark}[Polynomial-growth transforms]
Every measurable polynomial-growth transform belongs to $L^r(\gamma)$ for
all finite $r$.  Consequently every such transform that is centered and
variance-normalized satisfies Lemma~\ref{lem:gaussian-profile-main}.
\end{remark}

\subsection{Gaussian-noise LSS corollary}
\label{subsec:gaussian-LSS-main}

We now state the spectral consequence in Gaussian notation so that no
reconstruction from the general theorem is needed.  Let
$\{\xi_{ij}:1\le i<j\le N\}$ be independent standard Gaussian variables and
define the real symmetric zero-diagonal matrix
\begin{equation}\label{eq:Gaussian-matrix-main}
 M_{N,\gamma}^f(i,j)
 :=N^{-1/2}f\!\left(\xi_{ij}+\sqrt{\lambda N}\,x_ix_j\right),
 \qquad i<j,
 \qquad
 M_{N,\gamma}^f(i,i):=0.
\end{equation}
For later use, define the specialized order-one centering
\begin{align}
 m_{\varphi,\Gamma}^{\gamma,f}
 &:={m}_{\varphi,\Gamma}^{\rm base}(\kappa_f)
 +\frac{1}{2\pi\ii}\oint_\Gamma
 \varphi(z)\frac{\theta_fm'(z)}{1+\theta_fm(z)}\,\dd z
 \notag\\
 &\quad
 +\frac{\theta_f}{2\pi\ii}\oint_\Gamma
 \varphi'(z)m(z)\,\dd z
 -\frac{\lambda\beta_2^{f,\gamma}}{2\pi\ii}\oint_\Gamma
 \varphi(z)\frac{m(z)^3}{1-m(z)^2}\,\dd z.
 \label{eq:Gaussian-centering-main}
\end{align}

\begin{corollary}[Gaussian-noise bulk and full-trace LSS]
\label{cor:gaussian-LSS-main}
Let $f$ satisfy \eqref{eq:Gaussian-f-assumptions-main}, let $\lambda\ge0$ be
fixed, and suppose that the deterministic spike obeys
\begin{equation}\label{eq:Gaussian-spike-main}
 \norm{x}_2=1,
 \qquad
 \norm{x}_\infty=o(N^{-1/4}),
 \qquad
 \sum_{i=1}^N x_i=O(N^{o(1)}).
\end{equation}
Let $\Gamma$ satisfy (G1)--(G4) with $\theta_{f,\nu}$ replaced by $\theta_f$,
and let $\varphi_1,\ldots,\varphi_k$ be analytic on a neighborhood of
$\overline{D_\Gamma}$ and real-valued on its intersection with $\mathbb R$.
Then
\begin{equation}\label{eq:Gaussian-bulk-CLT-main}
 \left(
 \mathcal L_{M_{N,\gamma}^f}^{\Gamma}(\varphi_\ell)
 -N\int_{-2}^2\varphi_\ell(t)\rho_{\scal}(t)\,\dd t
 -m_{\varphi_\ell,\Gamma}^{\gamma,f}
 \right)_{\ell=1}^k
 \Longrightarrow
 (G_{\varphi_\ell})_{\ell=1}^k,
\end{equation}
where the limiting vector is centered Gaussian with
\begin{equation}\label{eq:Gaussian-covariance-main}
 \Cov(G_\varphi,G_\psi)
 =\mathcal V_{\kappa_f}^{\Gamma}(\varphi,\psi).
\end{equation}
If $|\theta_f|\le1$, then
$\mathcal L_{M_{N,\gamma}^f}^{\Gamma}(\varphi)
=\Tr\varphi(M_{N,\gamma}^f)$ with probability tending to one, so
\eqref{eq:Gaussian-bulk-CLT-main} is the full-trace CLT.  If
$|\theta_f|>1$ and each $\varphi_\ell$ is also analytic and real-valued on the
real axis near
\begin{equation}\label{eq:Gaussian-outlier-main}
 \rho_f:=\theta_f+\theta_f^{-1},
\end{equation}
then the full-trace CLT is obtained from
\eqref{eq:Gaussian-bulk-CLT-main} by replacing the statistic there with
\begin{equation}\label{eq:Gaussian-full-supercritical-main}
 \left(
 \Tr\varphi_\ell(M_{N,\gamma}^f)
 -N\int_{-2}^2\varphi_\ell(t)\rho_{\scal}(t)\,\dd t
 -m_{\varphi_\ell,\Gamma}^{\gamma,f}
 -\varphi_\ell(\rho_f)
 \right)_{\ell=1}^k
 \Longrightarrow
 (G_{\varphi_\ell})_{\ell=1}^k,
\end{equation}
with the covariance \eqref{eq:Gaussian-covariance-main}.
In particular, all these conclusions hold for every centered and
variance-normalized polynomial-growth transform.
\end{corollary}

\begin{proof}
Lemma~\ref{lem:gaussian-profile-main} verifies W1--W5, and
\eqref{eq:Gaussian-spike-main} is Assumption~\ref{ass:spike}.  The bulk claim
is therefore Theorem~\ref{thm:bulk-main-package} with the coefficients
\eqref{eq:Gaussian-coefficients-summary}--\eqref{eq:Gaussian-kappa-theta-summary}.
The two full-trace claims are respectively
Corollary~\ref{cor:subcritical-full-package} and
Theorem~\ref{thm:R6-supercritical-full}.
\end{proof}

\begin{remark}[Scope at the fourth-moment threshold]
The condition $L^{4+\epsilon}(\gamma)$ is a concrete sufficient condition,
not a claimed necessity.  Bare $L^4(\gamma)$ does not guarantee the uniform
shifted fourth-tail condition W5; Subsection~\ref{app:bare-L4-obstructions}
gives an explicit obstruction.  A matrix-averaged Gaussian shifted-tail
condition can be weaker than W5, but it is not a consequence of the present
uniform profile theorem.
\end{remark}

\begin{remark}[Smooth consistency]
If differentiation may be passed under expectation, then
$a_1=\E f'(\zeta)$,
$b_1=2\E[f(\zeta)f'(\zeta)]$, and
$b_2=\E[f'(\zeta)^2+f(\zeta)f''(\zeta)]$.  For Gaussian noise,
$a_1=\E[\xi f(\xi)]$ by integration by parts.
\end{remark}
\subsection{Checks for low-degree test functions}

For $\varphi=1$, every correction vanishes in the non-outlier full-trace
regime.  In the supercritical regime the bulk rank-one response is $-1$ and
the outlier contributes $1$, restoring the full trace $N$.

For $\varphi(t)=t$, the zero-diagonal convention gives
$\Tr M_N^f=0$ identically.  In the non-outlier regime the rank-one response is
$\theta$ and the diagonal-removal response is $-\theta$.  In the
supercritical regime the bulk rank-one response is $\theta-\rho_\theta$, the
diagonal response is $-\theta$, and the outlier contributes $\rho_\theta$.
The corrected kernel gives, for every analytic \(\psi\),
\begin{equation}\label{eq:linear-covariance-check}
 \mathcal V_\kappa^\Gamma(1,\psi)=0,
 \qquad
 \mathcal V_\kappa^\Gamma(t,\psi)=0,
\end{equation}
as required by the two deterministic trace identities.

Finally,
\[
 \E\Tr(M_N^f)^2
 =\frac2N\sum_{i<j}\E f(\zeta+s_{ij})^2
 =N-1+\lambda b_2+o(1).
\]
The theorem gives the same constant: the zero-diagonal base contribution is
$-1$, the rank-one contribution, including any separated outlier, is
$\theta^2=\lambda a_1^2$, and the quadratic variance response is
$\lambda\beta_2^{f,\nu}=\lambda(b_2-a_1^2)$.
For the quadratic fluctuation, expansion of \eqref{eq:Wigner-kernel} at
infinity gives
\begin{equation}\label{eq:quadratic-covariance-check}
 \mathcal V_\kappa^\Gamma(t^2,t^2)=4+2\kappa.
\end{equation}
This agrees with the direct homogeneous calculation
\[
 \Var\!\left(\Tr W_N^2\right)
 =4\sum_{i<j}\Var(W_N(i,j)^2)
 \longrightarrow 2(\kappa+2)=4+2\kappa.
\]

For the Gaussian model \eqref{eq:Gaussian-matrix-main}, this identity becomes
\[
 \E\Tr(M_{N,\gamma}^f)^2
 =N-1+\frac{\lambda}{2}
 \E\!\left[(\xi^2-1)f(\xi)^2\right]+o(1).
\]
Indeed, the rank-one and quadratic terms in
\eqref{eq:Gaussian-centering-main} sum to
\[
 \lambda\bigl(\E[\xi f(\xi)]\bigr)^2
 +\lambda\beta_2^{f,\gamma}
 =\frac{\lambda}{2}\E\!\left[(\xi^2-1)f(\xi)^2\right],
\]
which checks the specialized centering directly.
\section{Proof Architecture}\label{sec:proof-overview}

Write each transformed entry as its shifted mean plus the centered variable
$Y_s=f(\zeta+s)-\E f(\zeta+s)$.  W3 and W5 imply uniform nondegeneracy and
fourth-tail control for all microscopic shifts.  The three-point construction
in Subsection~\ref{sec:R2-package} replaces $Y_s$ by a uniformly bounded variable
with exactly the same first four centered moments.  The resulting bounded
triangular array has the same deterministic mean, variance profile, and third
and fourth cumulants as the target.

The bounded centered array is treated by the generalized-Wigner analytic LSS
input in Proposition~\ref{prop:R3-external-input}.  Its weak variance profile
is reduced to the flat profile through the vector Dyson equation; the balanced
linear profile has zero order-one response, while the quadratic profile gives
$c_{\varphi,\Gamma}^{\rm quad-var}$.  Fixed-contour local laws then identify
the rank-one Woodbury and zero-diagonal responses and yield bulk confinement
and the one-outlier alternative.

Finally, a smooth compactly supported statistic agreeing with the analytic
bulk statistic on both endpoint spectra is compared edge by edge.  The global
Fourier--Duhamel estimate bounds its fifth derivative by $O(N^{-5/2})$ for
every hybrid matrix.  Exact matching through order four and W5 therefore make
the telescoping error vanish without a hybrid spectral event or local law.
This proves the bulk theorem.  In the supercritical regime, the separated
outlier converges to $\theta+\theta^{-1}$ and contributes only the deterministic
term recorded in Theorem~\ref{thm:R6-supercritical-full}.

For Gaussian noise there is no second proof path.  The logical chain is
\[
 f\in L^{4+\epsilon}(\gamma)
 \ \Longrightarrow\ 
 \text{Lemma~\ref{lem:Gaussian-L4epsilon-W}}
 \ \Longrightarrow\ 
 \text{W1--W5}
 \ \Longrightarrow\ 
 \text{Theorem~\ref{thm:bulk-main-package}}
 \ \Longrightarrow\ 
 \text{Corollary~\ref{cor:gaussian-LSS-main}}.
\]
The Gaussian likelihood-ratio calculation verifies the input assumptions;
the atomic completion, generalized-Wigner reduction, global replacement, and
outlier argument are exactly those used for the general noise law.

\section{Discussion and Further Directions}\label{sec:discussion}

The proof separates four modules: shifted-profile identification, exact
four-moment completion, generalized-Wigner LSS analysis, and deterministic
finite-rank response.  This separation indicates which extensions can reuse
the present comparison argument and which require genuinely local spectral
input.  We record the most natural directions in that order.

\subsection{Relaxing profile and spike assumptions}
\label{subsec:discussion-assumptions}

\paragraph*{A matrix-averaged replacement for W5.}
Uniform shifted fourth-tail control is used twice: to obtain a uniformly
bounded exact four-moment surrogate and to make the fifth-order edgewise
replacement error summable.  A natural weaker target is the triangular-array
condition
\begin{equation}\label{eq:discussion-averaged-W5}
 \lim_{R\to\infty}\limsup_{N\to\infty}
 \frac{2}{N(N-1)}\sum_{i<j}
 \E\left[|Y_{s_{ij}}|^4
 \one_{\{|Y_{s_{ij}}|>R\}}\right]=0,
\end{equation}
together with W2--W4 along the actual shifts.  This condition permits a small
set of poorly controlled edges and is therefore strictly closer to what the
matrix replacement uses than the uniform one-parameter condition W5.  To
prove the same LSS theorem from \eqref{eq:discussion-averaged-W5}, one would
first separate good and exceptional edges, construct bounded moment-matched
surrogates only on the good set, and show that the total contribution of the
exceptional set is negligible in the characteristic-function telescope.  The
missing estimate is an array-level replacement bound weighted by the local
fourth-tail moduli.  No change to the generalized-Wigner or Woodbury parts is
expected.  This is the most direct assumption-improvement problem suggested
by the present proof, but it is not a consequence of the uniform completion
argument proved here.

\paragraph*{Retaining the linear variance-profile channel.}
The balance condition suppresses the order-one response of
\(p_{ij}^{(1)}=b_1\sqrt{\lambda N}\,x_ix_j\).  Without it, the corresponding
row-sum defect is
\begin{equation}\label{eq:discussion-linear-row-defect}
 q_i^{(1)}
 =\frac1N\sum_{j\ne i}p_{ij}^{(1)}
 =b_1\sqrt{\frac{\lambda}{N}}\,x_i
   \left(\sum_jx_j-x_i\right),
\end{equation}
and its contribution must be retained in the vector-Dyson centering.  One
possible extension is to state the CLT with an exact, $N$-dependent
deterministic centering defined by the finite-profile Dyson equation.  A
closed limiting formula would additionally require convergence of the spike
functionals entering the linear response; without such convergence there is
no single universal scalar correction.  The probabilistic four-moment
comparison survives, but the balancing reduction and the first-order Dyson
expansion in Appendix~\ref{sec:R3-package} must be replaced by a computation
that keeps \eqref{eq:discussion-linear-row-defect}.  This extension is
technically substantial but appears feasible within the same global
framework.

\paragraph*{Less delocalized spike vectors.}
The condition \(\norm{x}_\infty=o(N^{-1/4})\) is exactly what forces
\(\max_{i<j}|s_{ij}|=o(1)\), so that a single small-shift profile controls every
entry.  It cannot simply be deleted while retaining W2--W5 only near zero.
A possible mixed theory would isolate finitely or sparsely many
nonmicroscopic edges, treat them as an additional finite-rank or sparse
deformation, and apply the present completion to the remaining microscopic
array.  Such a result would need explicit hypotheses on the number and
magnitude of exceptional shifts and is less immediate than removing balance.

\subsection{Broader models and spectral observables}
\label{subsec:discussion-models}

\paragraph*{Finite-rank signals.}
For fixed $r$, consider shifts of the form
\[
 s_{ij}=\sqrt N\sum_{a=1}^r\sqrt{\lambda_a}\,
 x_i^{(a)}x_j^{(a)}.
\]
If these shifts are uniformly microscopic and the spike vectors satisfy
appropriate mixed balance and profile-limit conditions, the exact atomic
completion and edgewise replacement remain entrywise statements and should
extend without conceptual change.  The deterministic mean becomes finite
rank, the scalar Woodbury denominator becomes an $r\times r$ determinant,
and the quadratic variance profile contains mixed $(a,b)$ channels.  The main
new tasks are to identify the limits of those mixed profiles and to prove
uniform separation of all supercritical outliers.  Among model extensions,
this appears the most accessible.

\paragraph*{Nonanalytic spectral test functions.}
The roughness of the entry transform $f$ should not be confused with the
regularity of the LSS test function $\varphi$.  A first extension beyond the
present analytic class would use a Helffer--Sj\"ostrand or Fourier
representation for sufficiently smooth compactly supported $\varphi$, a
generalized-Wigner LSS input in the same regularity class, and deterministic
resolvent estimates uniform down to the smoothing scale.  The exact
four-moment cancellation remains available for smooth matrix functionals,
but the fixed-contour formulas for the bias and the spike response must be
replaced by almost-analytic integrals.  A realistic first target is a
high-regularity $C_c^k$ theorem; reaching the nearly optimal bulk Sobolev
regularity known for homogeneous Wigner matrices would require sharper
variance estimates and is a separate problem.

The complex Hermitian symmetry class and a nonzero diagonal can also be
treated in principle.  They change the homogeneous bias, covariance
constants, and diagonal cumulant terms but not the shifted-profile or exact
completion mechanism.  These are useful variants, although mathematically
less central than the two extensions above.

\subsection{Edge, sparse, and heavy-tailed regimes}
\label{subsec:discussion-frontiers}

\paragraph*{Critical BBP and edge universality.}
The present theorem records only the deterministic contribution of an outlier
separated by a fixed distance from the bulk.  Critical tuning
\(\theta_{f,\nu}=\pm1+wN^{-1/3}\) and edge spectral parameters
\(z=\pm2+O(N^{-2/3})\) lie outside the fixed-contour argument.  Proving edge
universality would require an optimal edge local law for the bounded
surrogate, Green-function comparison at the $N^{-2/3}$ scale, and uniform
control of the spike-dependent characteristic determinant through the
critical window.  The exact four-moment completion may still provide the
comparison array, but the global fifth-derivative estimate used here does not
by itself supply these local inputs.  Edge universality is therefore
plausible, but it is a new proof program rather than a formal extension of the
current LSS theorem.

\paragraph*{Sparse transformed noise.}
If only a vanishing proportion of entries is active, the deterministic
self-consistent density, local law, and LSS bias generally depend on the
sparsity scale.  In a sufficiently dense sparse regime, one may hope to
combine shifted-profile completion with an existing sparse-matrix LSS input.
Near the sparse spectral threshold, high cumulants and localized eigenvectors
produce additional fluctuation channels, so the present flat
generalized-Wigner backend is no longer appropriate.  A sparse extension must
therefore specify the sparsity parameter and its relation to the microscopic
spike shifts; there is no single sparsity-free analogue of the theorem.

\paragraph*{Below the fourth-moment threshold.}
For regularly varying transformed entries with tail index
\(2<\alpha<4\), the fourth cumulant is infinite and exact matching through
order four is unavailable.  The fluctuation scale and covariance are then
tail-index dependent, as in the half-heavy-tailed LSS theory
\cite{BenaychGeorgesMaltsev2016,LodhiaMaltsev2023}.  One should expect a
different normalization and a spike-dependent perturbation of that
heavy-tailed fluctuation field, rather than the covariance in
Theorem~\ref{thm:bulk-main-package}.  This direction is mathematically
natural, but it changes the limiting universality class and cannot be reached
by merely weakening W5.

The near-term directions most compatible with the present machinery are the
matrix-averaged version of W5, retention of the unbalanced linear profile,
finite-rank spikes, and smooth nonanalytic test functions.  Critical edge,
genuinely sparse, and infinite-fourth-moment models require new local laws or
new fluctuation inputs and should be regarded as separate programs.

\appendix

\section{Shifted Data and Exact Four-Moment Completion}
\label{app:proof-notation}
\subsection{Shifted data and fourth-tail consequences}
\label{sec:R1-package}
For the remainder of the proof, abbreviate
\begin{equation}\label{eq:shifted-data-package}
 m(s):=\E f(\zeta+s),\qquad
 Y_s:=f(\zeta+s)-m(s),\qquad
 v(s):=\E Y_s^2.
\end{equation}
Then
\begin{align}
 m(s)&=a_1s+o(s),\label{eq:mean-profile-package}\\
 \E f(\zeta+s)^2&=1+b_1s+b_2s^2+o(s^2),
 \label{eq:second-profile-package}\\
 \beta_2^{f,\nu}&:=b_2-a_1^2,\qquad
 \theta_{f,\nu}:=\sqrt\lambda\,a_1,\label{eq:beta-package}\\
 \kappa_4^{f,\nu}&:=\E f(\zeta)^4-3.\label{eq:kappa-package}
\end{align}
The assumptions used below may be cited in the equivalent forms
\begin{align}
 v(s)&\longrightarrow1,\label{eq:W3-package}\\
 \cum_4(Y_s)&\longrightarrow\kappa_4^{f,\nu},\label{eq:W4-package}\\
 \lim_{R\to\infty}\limsup_{\delta\downarrow0}
 \sup_{|s|\le\delta}
 \E\bigl[|Y_s|^4\one_{\{|Y_s|>R\}}\bigr]&=0,
 \label{eq:W5-package}
\end{align}
and
\begin{equation}\label{eq:spike-package}
 \norm{x}_2=1,\qquad
 \norm{x}_\infty=o(N^{-1/4}),\qquad
 \sum_{i=1}^N x_i=O(N^{o(1)}).
\end{equation}
For \(R,\delta>0\), define
\begin{equation}\label{eq:tail-modulus-package}
 \omega_4(R,\delta)
 :=\sup_{|s|\le\delta}
 \E\bigl[|Y_s|^4\one_{\{|Y_s|>R\}}\bigr].
\end{equation}
For fixed \(R\), the function \(\delta\mapsto\omega_4(R,\delta)\) is
nondecreasing.  Thus the inner limsup in W5 is an ordinary decreasing-domain
limit as \(\delta\downarrow0\).

\begin{lemma}[Uniform fourth moments near the zero shift]
\label{lem:uniform-fourth-package}
Assume \eqref{eq:W5-package}.  There exist \(\delta_0>0\) and
\(C_4<\infty\) such that
\begin{equation}\label{eq:uniform-fourth-package}
 \sup_{|s|\le\delta_0}\E|Y_s|^4\le C_4.
\end{equation}
If \eqref{eq:W3-package} also holds, then \(\delta_0\) may be chosen so that
\begin{equation}\label{eq:uniform-variance-package}
 \frac12\le v(s)\le\frac32,
 \qquad |s|\le\delta_0.
\end{equation}
\end{lemma}

\begin{proof}
Choose \(R_0\) so large that
\(
 \lim_{\delta\downarrow0}\omega_4(R_0,\delta)<1.
\)
Then choose \(\delta_0>0\) with
\(\omega_4(R_0,\delta_0)\le2\).  For \(|s|\le\delta_0\),
\[
 \E|Y_s|^4
 \le R_0^4+
 \E\bigl[|Y_s|^4\one_{\{|Y_s|>R_0\}}\bigr]
 \le R_0^4+2.
\]
This proves \eqref{eq:uniform-fourth-package}.  The variance bounds follow
from \(v(s)\to1\) after decreasing \(\delta_0\).
\end{proof}

\begin{lemma}[W5 implies the triangular-array fourth-moment Lindeberg bound]
\label{lem:exact-W5-Lindeberg-package}
Assume \eqref{eq:W5-package}.  Let \(s_{e,N}\), \(1\le e\le M_N\), be any
triangular array satisfying
\begin{equation}\label{eq:max-shift-package}
 \max_{1\le e\le M_N}|s_{e,N}|\longrightarrow0.
\end{equation}
Then, for every \(\varepsilon>0\),
\begin{equation}\label{eq:sup-Lindeberg-package}
 \max_{1\le e\le M_N}
 \E\left[
 |Y_{s_{e,N}}|^4
 \one_{\{|Y_{s_{e,N}}|>\varepsilon\sqrt N\}}
 \right]
 \longrightarrow0.
\end{equation}
In particular,
\begin{equation}\label{eq:average-Lindeberg-package}
 \frac1{M_N}\sum_{e=1}^{M_N}
 \E\left[
 |Y_{s_{e,N}}|^4
 \one_{\{|Y_{s_{e,N}}|>\varepsilon\sqrt N\}}
 \right]
 \longrightarrow0.
\end{equation}
\end{lemma}

\begin{proof}
Fix \(\eta>0\).  By W5, choose \(R<\infty\) such that
\[
 \lim_{\delta\downarrow0}\omega_4(R,\delta)<\eta.
\]
Choose \(\delta>0\) so that \(\omega_4(R,\delta)<2\eta\).  For all
sufficiently large \(N\), \eqref{eq:max-shift-package} gives
\(\max_e|s_{e,N}|\le\delta\), while
\(\varepsilon\sqrt N\ge R\).  Hence
\[
 \max_e\E\left[
 |Y_{s_{e,N}}|^4
 \one_{\{|Y_{s_{e,N}}|>\varepsilon\sqrt N\}}
 \right]
 \le\omega_4(R,\delta)<2\eta.
\]
Since \(\eta\) is arbitrary, \eqref{eq:sup-Lindeberg-package} follows, and
\eqref{eq:average-Lindeberg-package} is immediate.
\end{proof}

\begin{remark}[Quantifier reconciliation]
Lemma~\ref{lem:exact-W5-Lindeberg-package} uses W5 exactly as stated.  It does
not assume that the family \(\{|Y_s|^4:|s|\le\delta\}\) is uniformly
integrable for one fixed \(\delta>0\).  The order is: choose a fixed tail level
\(R\), then a sufficiently small shift window \(\delta\), and only afterward
take \(N\) large.
\end{remark}

\subsection{Spike-generated profile estimates}

Put
\begin{equation}\label{eq:profiles-package}
 p_{ij}^{(1)}:=b_1\sqrt{\lambda N}\,x_ix_j,
 \qquad
 p_{ij}^{(2)}:=\lambda\beta_2^{f,\nu}N x_i^2x_j^2.
\end{equation}

\begin{proposition}[Consequences of delocalization and balance]
\label{prop:spike-rates-package}
Under \eqref{eq:spike-package},
\begin{align}
 \max_i x_i^2&=o(N^{-1/2}),
 \label{eq:max-square-package}\\
 \sum_i x_i^4&=o(N^{-1/2}),
 \label{eq:fourth-spike-package}\\
 \max_{i<j}|s_{ij}|&=o(1).
 \label{eq:max-microshift-package}
\end{align}
The linear profile satisfies
\begin{align}
 \max_{i,j}|p_{ij}^{(1)}|&=o(1),
 \label{eq:p1-max-package}\\
 \frac1{N^2}\sum_{i,j}|p_{ij}^{(1)}|&=O(N^{-1/2}),
 \label{eq:p1-L1-package}\\
 \frac1{N^2}\sum_{i,j}(p_{ij}^{(1)})^2&=O(N^{-1}),
 \label{eq:p1-L2-package}\\
 N\left|\frac1{N^2}\sum_{i,j}p_{ij}^{(1)}\right|&=o(1),
 \label{eq:p1-signed-package}\\
 \max_i\frac1N\sum_j|p_{ij}^{(1)}|&=o(N^{-1/4}).
 \label{eq:p1-row-package}
\end{align}
The quadratic profile satisfies
\begin{align}
 \max_{i,j}|p_{ij}^{(2)}|&=o(1),
 \label{eq:p2-max-package}\\
 \frac1{N^2}\sum_{i,j}|p_{ij}^{(2)}|&=O(N^{-1}),
 \label{eq:p2-L1-package}\\
 \frac1{N^2}\sum_{i,j}(p_{ij}^{(2)})^2&=o(N^{-1}),
 \label{eq:p2-L2-package}\\
 \max_i\frac1N\sum_j|p_{ij}^{(2)}|&=o(N^{-1/2}).
 \label{eq:p2-row-package}
\end{align}
No stronger polynomial rate follows from \eqref{eq:spike-package} alone.
\end{proposition}

\begin{proof}
The first two assertions follow from
\[
 \max_i x_i^2=\norm{x}_\infty^2=o(N^{-1/2}),
 \qquad
 \sum_i x_i^4
 \le\norm{x}_\infty^2\sum_i x_i^2=o(N^{-1/2}).
\]
Also,
\[
 \max_{i<j}|s_{ij}|
 \le\sqrt{\lambda N}\,\norm{x}_\infty^2=o(1).
\]
For the linear profile, \(\norm{x}_1\le\sqrt N\norm{x}_2=\sqrt N\)
gives
\[
 \frac1{N^2}\sum_{i,j}|p_{ij}^{(1)}|
 \le C\frac{\sqrt N}{N^2}\norm{x}_1^2
 \le CN^{-1/2},
\]
while
\[
 \frac1{N^2}\sum_{i,j}(p_{ij}^{(1)})^2
 =O\left(\frac{N}{N^2}\sum_{i,j}x_i^2x_j^2\right)
 =O(N^{-1}).
\]
The signed average is
\[
 N\left|\frac1{N^2}\sum_{i,j}p_{ij}^{(1)}\right|
 \le C\frac{(\sum_i x_i)^2}{\sqrt N}
 =N^{-1/2+o(1)}=o(1).
\]
Finally,
\[
 \frac1N\sum_j|p_{ij}^{(1)}|
 \le C\frac{\sqrt N}{N}|x_i|\norm{x}_1
 \le C|x_i|=o(N^{-1/4}).
\]
The maximum bound follows directly from
\(\sqrt N\norm{x}_\infty^2=o(1)\).

For the quadratic profile,
\[
 \frac1{N^2}\sum_{i,j}|p_{ij}^{(2)}|
 =O\left(\frac{N}{N^2}
 \sum_{i,j}x_i^2x_j^2\right)=O(N^{-1}),
\]
and
\[
 \frac1{N^2}\sum_{i,j}(p_{ij}^{(2)})^2
 =O\left(\sum_i x_i^4\right)^2=o(N^{-1}).
\]
Moreover,
\[
 \frac1N\sum_j|p_{ij}^{(2)}|
 =O\left(x_i^2\sum_jx_j^2\right)
 =o(N^{-1/2}),
\]
and \(N\norm{x}_\infty^4=o(1)\) gives the maximum bound.
\end{proof}

\begin{corollary}[Profile remainders]
\label{cor:profile-remainder-package}
Suppose
\begin{equation}\label{eq:profile-rem-package}
 |r_{ij,N}|\le\varepsilon_Ns_{ij}^2,
 \qquad \varepsilon_N\to0.
\end{equation}
Then
\begin{align}
 \max_{i,j}|r_{ij,N}|&=o(1),\\
 N\frac1{N^2}\sum_{i,j}|r_{ij,N}|&=o(1),\\
 \max_i\frac1N\sum_j|r_{ij,N}|&=o(N^{-1/2}).
\end{align}
\end{corollary}

\begin{proof}
Use \(s_{ij}^2=\lambda N x_i^2x_j^2\),
\(\max_{ij}s_{ij}^2=o(1)\), and the quadratic-profile calculations in
Proposition~\ref{prop:spike-rates-package}, with the additional factor
\(\varepsilon_N\).
\end{proof}

\subsection{Three-point completion and uniform shifted surrogate}
\label{sec:R2-package}

The proposed four-moment completion is not merely feasible: for centered real
variables it has an explicit three-point solution.

\begin{proposition}[Three-point representation of four centered moments]
\label{prop:three-point-package}
Let \(X\) be a real random variable satisfying
\begin{equation}\label{eq:X-moments-package}
 \E X=0,
 \qquad
 \E X^2=v>0,
 \qquad
 \E X^4<\infty.
\end{equation}
Put
\begin{equation}\label{eq:standard-moments-package}
 a:=\frac{\E X^3}{v^{3/2}},
 \qquad
 b:=\frac{\E X^4}{v^2},
 \qquad
 c:=b-a^2.
\end{equation}
Then \(c\ge1\).  Define
\begin{equation}\label{eq:alpha-beta-package}
 \alpha:=\frac{\sqrt{a^2+4c}-a}{2},
 \qquad
 \beta:=\frac{\sqrt{a^2+4c}+a}{2},
\end{equation}
and
\begin{equation}\label{eq:three-probs-package}
 p_-:=\frac{1}{\alpha(\alpha+\beta)},
 \qquad
 p_+:=\frac{1}{\beta(\alpha+\beta)},
 \qquad
 p_0:=1-\frac1c.
\end{equation}
These numbers are nonnegative and sum to one.  If \(Q\) has the three-point
law
\begin{equation}\label{eq:Q-law-package}
 \mathbb P(Q=-\alpha)=p_-,
 \qquad
 \mathbb P(Q=0)=p_0,
 \qquad
 \mathbb P(Q=\beta)=p_+,
\end{equation}
then \(\widehat X:=\sqrt v\,Q\) satisfies
\begin{equation}\label{eq:four-exact-package}
 \E\widehat X^q=\E X^q,
 \qquad q=1,2,3,4.
\end{equation}
Moreover,
\begin{equation}\label{eq:support-three-package}
 |\widehat X|
 \le2\sqrt{\frac{\E X^4}{v}}.
\end{equation}
\end{proposition}

\begin{proof}
For the standardized variable \(Z=X/\sqrt v\),
\[
 0\le\E(Z^2-aZ-1)^2=b-a^2-1,
\]
so \(c=b-a^2\ge1\).  The definitions give
\begin{equation}\label{eq:alpha-identities-package}
 \alpha>0,
 \qquad
 \beta>0,
 \qquad
 \beta-\alpha=a,
 \qquad
 \alpha\beta=c.
\end{equation}
Consequently
\[
 p_-+p_+
 =\frac{1}{\alpha\beta}=\frac1c,
 \qquad
 p_-+p_0+p_+=1,
\]
and all three probabilities are nonnegative.  Direct calculation yields
\begin{align*}
 \E Q
 &=-\frac{1}{\alpha+\beta}
   +\frac{1}{\alpha+\beta}=0,\\
 \E Q^2
 &=\frac{\alpha}{\alpha+\beta}
   +\frac{\beta}{\alpha+\beta}=1,\\
 \E Q^3
 &=\frac{-\alpha^2+\beta^2}{\alpha+\beta}
   =\beta-\alpha=a,\\
 \E Q^4
 &=\frac{\alpha^3+\beta^3}{\alpha+\beta}
   =\alpha^2-\alpha\beta+\beta^2
   =(\beta-\alpha)^2+\alpha\beta=b.
\end{align*}
Multiplication by \(\sqrt v\) proves \eqref{eq:four-exact-package}.  Finally,
\[
 \max\{\alpha,\beta\}
 \le |a|+\sqrt c
 \le2\sqrt b.
\]
Thus
\[
 |\widehat X|
 \le2\sqrt v\sqrt b
 =2\sqrt{\frac{\E X^4}{v}},
\]
which is \eqref{eq:support-three-package}.
\end{proof}

\begin{remark}[Boundary cases]
Equality \(c=1\) is allowed.  Then \(p_0=0\) and the surrogate is a two-point
law.  This is exactly the equality case of
\(\E(Z^2-aZ-1)^2\ge0\); no nondegeneracy margin beyond \(v>0\) is needed.
\end{remark}

For \(|s|\le\delta_0\), set
\begin{equation}\label{eq:shifted-moments-package}
 \tau_3(s):=\E Y_s^3,
 \qquad
 \mu_4(s):=\E Y_s^4,
\end{equation}
and apply Proposition~\ref{prop:three-point-package} with \(X=Y_s\).  Denote
the standardized moments by \(a(s),b(s),c(s)\), the two nonzero support
parameters by \(\alpha(s),\beta(s)\), and the probabilities by
\(p_-(s),p_0(s),p_+(s)\).

Let \(U\) be uniform on \((0,1)\), independent of the target noise, and define
\begin{equation}\label{eq:Q-s-U-package}
 Q_s(U):=
 \begin{cases}
 -\alpha(s),&0<U\le p_-(s),\\
 0,&p_-(s)<U\le p_-(s)+p_0(s),\\
 \beta(s),&p_-(s)+p_0(s)<U<1,
 \end{cases}
\end{equation}
and
\begin{equation}\label{eq:Yhat-s-package}
 \widehat Y_s:=\sqrt{v(s)}\,Q_s(U),
 \qquad
 Z^{(4)}(s):=m(s)+\widehat Y_s.
\end{equation}

\begin{theorem}[Uniform bounded four-moment surrogate]
\label{thm:uniform-four-surrogate-package}
Assume W3 and W5.  After decreasing \(\delta_0\) as in
Lemma~\ref{lem:uniform-fourth-package}, the variables
\(\widehat Y_s\), \(|s|\le\delta_0\), satisfy
\begin{equation}\label{eq:uniform-support-surrogate-package}
 \sup_{|s|\le\delta_0}|\widehat Y_s|
 \le 2\sqrt{2C_4}
 \quad\text{almost surely},
\end{equation}
and
\begin{equation}\label{eq:uniform-four-match-package}
 \E\widehat Y_s^q=\E Y_s^q,
 \qquad q=1,2,3,4,
 \qquad |s|\le\delta_0.
\end{equation}
Consequently
\begin{align}
 \E Z^{(4)}(s)&=m(s),\\
 \Var(Z^{(4)}(s))&=v(s),\\
 \cum_r(\widehat Y_s)&=\cum_r(Y_s),
 \qquad r=3,4.
\end{align}
If \(f\) is Borel measurable, the parameters in
\eqref{eq:Q-s-U-package} may be chosen Borel measurably in \(s\).
\end{theorem}

\begin{proof}
Exact matching follows from Proposition~\ref{prop:three-point-package}.  By
Lemma~\ref{lem:uniform-fourth-package},
\(v(s)\ge1/2\) and \(\mu_4(s)\le C_4\).  Hence
\[
 |\widehat Y_s|
 \le2\sqrt{\frac{\mu_4(s)}{v(s)}}
 \le2\sqrt{2C_4},
\]
uniformly in the shift.  Equality of the third cumulants follows from
centering, and equality of the fourth cumulants follows from equality of the
second and fourth centered moments.  Finally, the maps
\(s\mapsto\E Y_s^q\), \(q\le4\), are Borel parameter integrals of jointly
Borel functions; the displayed algebraic formulas preserve measurability.
\end{proof}

\subsection{Atomic array and replacement hybrids}

Let \(U_{ij}\), \(i<j\), be independent uniform variables, independent of the
target noises \(\zeta_{ij}\).  Define
\begin{equation}\label{eq:atomic-array-package}
 M_N^{(4)}(i,j)
 :=N^{-1/2}Z^{(4)}_{ij}(s_{ij}),
 \qquad
 Z^{(4)}_{ij}(s_{ij})
 :=m(s_{ij})+\widehat Y_{s_{ij}}(U_{ij}),
\end{equation}
and set \(M_N^{(4)}(i,i)=0\).

\begin{corollary}[Exact matrix-scale data]
\label{cor:atomic-array-data-package}
Under W3, W4, W5, and the spike assumption, the array
\(M_N^{(4)}\) has the following properties for all sufficiently large \(N\):
\begin{enumerate}[label=(A\arabic*)]
\item its upper-triangular entries are independent and uniformly bounded after
      multiplication by \(\sqrt N\);
\item its deterministic mean matrix is exactly that of \(M_N^f\);
\item its centered variance profile is exactly that of \(M_N^f\);
\item its edgewise third and fourth cumulants are exactly those of the target;
\item
\begin{equation}\label{eq:uniform-kappa-package}
 \max_{i<j}
 \left|
 \cum_4(\widehat Y_{s_{ij}})-\kappa_4^{f,\nu}
 \right|\longrightarrow0;
\end{equation}
\item the fourth-moment Lindeberg condition is eventually identically zero.
\end{enumerate}
\end{corollary}

\begin{proof}
Independence and boundedness follow from the construction and
Theorem~\ref{thm:uniform-four-surrogate-package}.  The exact mean, variance,
third-cumulant, and fourth-cumulant statements follow from
\eqref{eq:uniform-four-match-package}.  Since
\(\max_{i<j}|s_{ij}|\to0\), W4 and exact fourth-cumulant matching give
\eqref{eq:uniform-kappa-package}.  Uniform boundedness makes the surrogate
Lindeberg sum zero once \(\varepsilon\sqrt N\) exceeds its support bound.
\end{proof}

\begin{remark}[Uniform backend data]
Corollary~\ref{cor:atomic-array-data-package} gives uniform fourth-cumulant
convergence directly, so neither a subsequence argument nor an auxiliary
order of limits is needed.  The bounded-array argument below permits arbitrary
uniformly bounded edge-dependent third cumulants; exact matching ensures that
no third-cumulant discrepancy appears in the replacement step.
\end{remark}

For every edge $e=(i,j)$, put $Y_e=Y_{s_{ij}}$ and let
$\widehat Y_e=\widehat Y_{s_{ij}}(U_{ij})$ be the surrogate just constructed.
Pairs belonging to distinct edges are independent, and
\begin{equation}\label{eq:edge-four-match-package}
 \E Y_e^q=\E\widehat Y_e^q,
 \qquad q=1,2,3,4.
\end{equation}
Let
\begin{equation}\label{eq:mean-matrix-package}
 A_N(i,j):=N^{-1/2}m(s_{ij}),
 \qquad A_N(i,i):=0.
\end{equation}
Order the off-diagonal edges as $e_1,\ldots,e_{M_N}$, where
$M_N=N(N-1)/2$.

\begin{definition}[Four-moment replacement path]
\label{def:hybrid-package}
For $0\le r\le M_N$, define the centered hybrid $H_N^{(r)}$ by
\begin{equation}\label{eq:hybrid-entry-package}
 H_N^{(r)}(e_q)
 :=N^{-1/2}
 \begin{cases}
 Y_{e_q},&q\le r,\\
 \widehat Y_{e_q},&q>r,
 \end{cases}
\end{equation}
and set
\begin{equation}\label{eq:full-hybrid-package}
 M_N^{(r)}:=A_N+H_N^{(r)}.
\end{equation}
For the replacement of edge $e_r=(a,b)$, let
$\mathring H_N^{(r,e_r)}$ be the centered hybrid with the $(a,b)$ and
$(b,a)$ entries set to zero.  Put
\begin{equation}\label{eq:open-edge-package}
 M_N^{(r,e_r)}(u)
 :=A_N+\mathring H_N^{(r,e_r)}+N^{-1/2}uV_{e_r},
 \qquad
 V_{e_r}:=E_{ab}+E_{ba}.
\end{equation}
The open-edge matrix is independent of both candidate variables
$Y_{e_r}$ and $\widehat Y_{e_r}$.
\end{definition}

\begin{proposition}[Uniform data along the hybrid path]
\label{prop:hybrid-data-package}
Every hybrid has the same deterministic mean matrix, the same variance
profile, and the same edgewise third and fourth cumulants.  Under W5 and
$\max_e|s_{e,N}|\to0$,
\begin{equation}\label{eq:uniform-hybrid-Lindeberg-package}
 \sup_{0\le r\le M_N}
 \frac1{M_N}\sum_{e=1}^{M_N}
 \E\left[
 |X_e^{(r)}|^4
 \one_{\{|X_e^{(r)}|>\varepsilon\sqrt N\}}
 \right]\longrightarrow0,
\end{equation}
where $X_e^{(r)}$ denotes the centered scalar used at edge $e$ in
$H_N^{(r)}$.
\end{proposition}

\begin{proof}
The first assertion follows edge by edge from
\eqref{eq:edge-four-match-package}.  For fixed $\varepsilon>0$, the bounded
surrogate contribution is eventually zero.  The remaining edges are bounded
by the target supremum in Lemma~\ref{lem:exact-W5-Lindeberg-package}, uniformly
in the replacement index $r$.
\end{proof}

\section{Bounded-Array LSS and Deterministic Spectral Analysis}
\label{sec:R3-package}

\subsection{Generalized-Wigner input and balancing reduction}

Write
\begin{equation}\label{eq:R3-centered-array}
 H_N^{(4)}(i,j):=N^{-1/2}\widehat Y_{s_{ij}}(U_{ij}),
 \qquad H_N^{(4)}(i,i):=0,
\end{equation}
and let \(S_N=(S_{ij})\) be its variance matrix:
\begin{equation}\label{eq:R3-variance-matrix}
 S_{ij}:=\E|H_N^{(4)}(i,j)|^2
 =\frac{v(s_{ij})}{N}\quad(i\ne j),
 \qquad S_{ii}:=0.
\end{equation}
By W3 and the profile expansion,
\begin{equation}\label{eq:R3-p-decomposition}
 v(s_{ij})=1+p_{ij}^{(1)}+p_{ij}^{(2)}+r_{ij,N},
\end{equation}
where \(p^{(1)}\) and \(p^{(2)}\) are defined in
\eqref{eq:profiles-package} and
\begin{equation}\label{eq:R3-rem-bound}
 |r_{ij,N}|\le\varepsilon_Ns_{ij}^2,\qquad \varepsilon_N\to0.
\end{equation}
Let \(S_N^0\) denote the zero-diagonal flat matrix
\begin{equation}\label{eq:R3-flat-S}
 S_{ij}^0=N^{-1}\one_{\{i\ne j\}}.
\end{equation}
The order-\(N^{-1}\) difference between \(S_N^0\) and an exactly stochastic
zero-diagonal profile is part of the homogeneous zero-diagonal Wigner
centering \(m_{\varphi,\Gamma}^{\rm base}(\kappa_4^{f,\nu})\).  It must not be counted again as
a transformed-profile correction.

The only stochastic CLT input required below is the following
fixed-scale specialization of the generalized-Wigner LSS theorem.  It is
recorded explicitly so that the article does not silently apply a homogeneous
Wigner theorem to an edge-dependent array.

We first isolate the two reductions needed to match the hypotheses of the
imported theorem.  For a symmetric variance matrix \(S\), write
\[
 q_i(S):=\sum_jS_{ij}-1,
 \qquad
 q(S):=(q_1(S),\ldots,q_N(S)).
\]

\begin{lemma}[Stochastic normalization of a weakly flat profile]
\label{lem:R3-stochastic-normalization}
Suppose that \(S\) is symmetric, \(S_{ii}=0\), and
\[
 \frac{c}{N}\le S_{ij}\le\frac{C}{N}\quad(i\ne j),
 \qquad \max_i|q_i(S)|=o(1).
\]
For all sufficiently large \(N\), there is a positive diagonal matrix
\(D=\diag(d_1,\ldots,d_N)\) such that
\begin{equation}\label{eq:R3-balanced-profile}
 \widetilde S:=DSD,
 \qquad
 \sum_j\widetilde S_{ij}=1\quad(1\le i\le N).
\end{equation}
Moreover,
\begin{align}
 \norm{D-I}_{\infty}
 &\le C\max_i|q_i(S)|=o(1),\label{eq:R3-D-infty}\\
 \sum_i|d_i-1|^2
 &\le C\sum_i|q_i(S)|^2,\label{eq:R3-D-l2}\\
 \norm{\widetilde S-S}_{\HS}
 &\le C N^{-1/2}\norm{q(S)}_2=o(1).
 \label{eq:R3-balanced-HS}
\end{align}
If \(H_{ij}\) has variance \(S_{ij}\), then
\(\widetilde H_{ij}:=\sqrt{d_id_j}\,H_{ij}\) has variance
\(\widetilde S_{ij}\).  Uniform boundedness of \(\sqrt N H_{ij}\) and
uniform convergence of its fourth cumulants are preserved by this rescaling.

For the spike-generated profile \eqref{eq:R3-variance-matrix}--
\eqref{eq:R3-rem-bound}, the following stronger estimates hold:
\begin{align}
 \max_i|q_i(S_N)|&=o(N^{-1/2}),
 \label{eq:R3-strong-q}\\
 \norm{D-I}_\infty&=o(N^{-1/2}),
 \label{eq:R3-strong-D-infty}\\
 \norm{D-I}_{\HS}^2&=o(1),
 \label{eq:R3-strong-D-HS}\\
 \norm{S_N-DS_ND}_{\HS}&=o(N^{-1/2}),
 \label{eq:R3-strong-profile-HS}\\
 \max_{i\ne j}\frac{|S_{ij}-d_id_jS_{ij}|}{S_{ij}}
 &=o(N^{-1/2}).
 \label{eq:R3-strong-relative}
\end{align}
\end{lemma}

\begin{proof}
The equations in \eqref{eq:R3-balanced-profile} are
\[
 d_i\sum_jS_{ij}d_j=1.
\]
Write \(d_i=e^{u_i}\) and define
\begin{equation}\label{eq:R3-balancing-map}
 F_i(u):=u_i+\log\left(\sum_jS_{ij}e^{u_j}\right).
\end{equation}
Then \(F(u)=0\) is the balancing equation,
\[
 F_i(0)=\log(1+q_i(S)),
 \qquad
 DF(u)=I+P(u),
\]
where
\begin{equation}\label{eq:R3-balancing-kernel}
 P(u)_{ij}:=
 \frac{S_{ij}e^{u_j}}{\sum_kS_{ik}e^{u_k}}.
\end{equation}
For \(\norm u_\infty\le u_0\), with \(u_0>0\) fixed and small,
\[
 \frac{c_0}{N}\le P(u)_{ij}\le\frac{C_0}{N}\quad(i\ne j).
\]
Consequently, for \(N\ge4\),
\begin{equation}\label{eq:R3-two-step-minorization}
 P(u)^2_{ij}
 \ge\sum_{k\notin\{i,j\}}\frac{c_0^2}{N^2}
 \ge\frac aN
\end{equation}
with \(a>0\) independent of \(N\).  Thus
\[
 \operatorname{osc}(P(u)^2v)
 \le(1-a)\operatorname{osc}(v).
\]

The kernel \(P(u)\) is reversible with respect to
\[
 \pi_i(u):=e^{u_i}\sum_kS_{ik}e^{u_k},
\]
because
\(\pi_iP(u)_{ij}=S_{ij}e^{u_i+u_j}=\pi_jP(u)_{ji}\).
After normalization, \(c_1/N\le\pi_i\le C_1/N\).
Equation \eqref{eq:R3-two-step-minorization} implies that every nonconstant
eigenvalue \(\lambda\) of \(P(u)\) satisfies
\(|\lambda|\le\sqrt{1-a}\).  The constant direction has eigenvalue one.
It follows that
\begin{equation}\label{eq:R3-balancing-inverse}
 \norm{(I+P(u))^{-1}}_{\infty\to\infty}
 +\norm{(I+P(u))^{-1}}_{2\to2}\le C.
\end{equation}
For the sup-norm estimate, decompose into the stationary mean and the
mean-zero part and use
\((I+P)^{-1}=(I-P)(I-P^2)^{-1}\) on the latter; the series for
\((I-P^2)^{-1}\) converges by the oscillation contraction.  Equivalence of
the stationary weighted norm and the normalized Euclidean norm gives the
\(\ell^2\) estimate.

The Hessian of the logarithm in \eqref{eq:R3-balancing-map} is the covariance
matrix of the probability vector in \eqref{eq:R3-balancing-kernel}.  Hence
\begin{align*}
 \norm{F(u)-F(v)-DF(v)(u-v)}_\infty
 &\le C\norm{u-v}_\infty^2,\\
 \norm{F(u)-F(v)-DF(v)(u-v)}_2
 &\le C(\norm u_\infty+\norm v_\infty)\norm{u-v}_2.
\end{align*}
Newton iteration, started at zero, is therefore a contraction in
\(\norm u_\infty\le2C\norm{q(S)}_\infty\) for all sufficiently large \(N\).
It gives a unique small solution satisfying
\[
 \norm u_\infty\le C\norm{q(S)}_\infty,
 \qquad
 \norm u_2\le C\norm{q(S)}_2.
\]
Since \(e^{u_i}-1=u_i+O(u_i^2)\), this proves
\eqref{eq:R3-D-infty}--\eqref{eq:R3-D-l2} and positivity.

Put \(U=D-I\).  Since \(S_{ij}\le C/N\),
\[
 \norm{DSD-S}_{\HS}^2
 \le \frac{C}{N^2}\sum_{i,j}
 (|U_{ii}|+|U_{jj}|+|U_{ii}U_{jj}|)^2
 \le \frac{C}{N}\norm U_2^2,
\]
which proves \eqref{eq:R3-balanced-HS}.  The final assertion follows from
\(\sqrt N\widetilde H_{ij}=\sqrt{d_id_j}\sqrt N H_{ij}\) and
\eqref{eq:R3-D-infty}.

It remains to prove the model-specific rates.  The zero diagonal contributes
$-N^{-1}$ to $q_i(S_N)$.  For the linear profile, the signed row sum satisfies
\[
 \left|\frac1N\sum_{j\ne i}p_{ij}^{(1)}\right|
 \leq C\frac{|x_i|}{\sqrt N}
 \left(\left|\sum_jx_j\right|+|x_i|\right)
 =o(N^{-1/2})
\]
by \eqref{eq:spike-package}.  For the quadratic profile and the remainder,
\begin{align*}
 \left|\frac1N\sum_{j\ne i}p_{ij}^{(2)}\right|
 &\leq Cx_i^2\sum_jx_j^2=o(N^{-1/2}),\\
 \frac1N\sum_{j\ne i}|r_{ij,N}|
 &\leq C\varepsilon_Nx_i^2\sum_jx_j^2=o(N^{-1/2}).
\end{align*}
This proves \eqref{eq:R3-strong-q}.  Equations
\eqref{eq:R3-D-infty} and \eqref{eq:R3-D-l2} then give
\eqref{eq:R3-strong-D-infty} and
\[
 \norm{D-I}_{\HS}^2
 \leq C\sum_i|q_i(S_N)|^2=o(1),
\]
which is \eqref{eq:R3-strong-D-HS}.  Substitution in
\eqref{eq:R3-balanced-HS} proves \eqref{eq:R3-strong-profile-HS}.
Finally, uniform flatness and $d_i=1+o(N^{-1/2})$ imply
\[
 \frac{|S_{ij}-d_id_jS_{ij}|}{S_{ij}}
 =|1-d_id_j|=o(N^{-1/2}),
\]
proving \eqref{eq:R3-strong-relative}.
\end{proof}

The zero diagonal still falls outside the literal lower-variance hypothesis
of the generalized-Wigner theorem.  The next lemma records the limiting
procedure used for this point.

\begin{lemma}[Vanishing diagonal regularization]
\label{lem:R3-diagonal-regularization}
Let \(\widetilde H\) be a bounded centered array with zero diagonal and an
exactly stochastic, uniformly flat off-diagonal variance matrix
\(\widetilde S\).  Let \(\chi_1,\ldots,\chi_N\) be independent, centered,
variance-one, uniformly bounded random variables, independent of
\(\widetilde H\).  For fixed \(\eta\in(0,1)\), set
\begin{equation}\label{eq:R3-eta-matrix}
 \widetilde H^{(\eta)}_{ij}
 :=\sqrt{1-\eta/N}\,\widetilde H_{ij}\quad(i\ne j),
 \qquad
 \widetilde H^{(\eta)}_{ii}:=\sqrt{\eta/N}\,\chi_i.
\end{equation}
Then the variance matrix of \(\widetilde H^{(\eta)}\) is exactly stochastic
and satisfies the uniform lower and upper bounds required by the
generalized-Wigner theorem, with constants allowed to depend on \(\eta\).
For every finite family of functions analytic near \(\Gamma\),
\begin{align}
 \limsup_{N\to\infty}\left|
 \E\Tr\varphi(\widetilde H^{(\eta)})
 -\E\Tr\varphi(\widetilde H)
 \right|&\le C_{\varphi,\Gamma}\eta,
 \label{eq:R3-eta-mean}\\
 \limsup_{N\to\infty}\Var\!\left(
 \Tr\varphi(\widetilde H^{(\eta)})
 -\Tr\varphi(\widetilde H)
 \right)&\le C_{\varphi,\Gamma}\eta.
 \label{eq:R3-eta-L2}
\end{align}
Moreover, the Li--Xu expectation and covariance formulas converge to their flat
zero-diagonal versions in the order \(N\to\infty\), then
\(\eta\downarrow0\).  Consequently the joint CLT for
\(\widetilde H^{(\eta)}\) transfers to \(\widetilde H\).
\end{lemma}

\begin{proof}
The variance matrix is
\begin{equation}\label{eq:R3-eta-profile}
 \widetilde S^{(\eta)}
 =(1-\eta/N)\widetilde S+(\eta/N)I.
\end{equation}
Its row sums equal one, its off-diagonal entries remain uniformly comparable
to \(N^{-1}\), and
\(N\Var(\widetilde H^{(\eta)}_{ii})=\eta\).  Hence, for each fixed
\(\eta>0\), it satisfies the literal variance hypothesis of Li--Xu's theorem.

We first prove the coupling estimates.  Put
\(c_\eta=\sqrt{1-\eta/N}\), \(A=c_\eta\widetilde H\), and
\(D_\eta=\sqrt{\eta/N}\diag(\chi)\).  Along the off-diagonal scaling path,
analytic functional calculus gives
\[
 \frac{\dd}{\dd t}\Tr\varphi(t\widetilde H)
 =\Tr\bigl(\varphi'(t\widetilde H)\widetilde H\bigr).
\]
On the fixed-contour confinement event this is \(O_{\varphi,\Gamma}(N)\).
Since \(|c_\eta-1|\le C\eta/N\), the rescaling changes the mean by
\(O_{\varphi,\Gamma}(\eta)\).

Condition on \(A\) and use the resolvent series in \(D_\eta\).  The linear
term has zero conditional expectation.  In a term of total diagonal degree
\(r\ge2\), let \(p\) be the number of distinct diagonal indices.  A nonzero
expectation requires every index to occur at least twice, so \(p\le r/2\).
The complete contribution of all index patterns with this value of \(p\) is
bounded by
\[
 C_{r,\Gamma}N^p(\eta/N)^{r/2}
 \le C_{r,\Gamma}\eta^{r/2}.
\]
This includes the pair partitions omitted by a single-index cumulant count.
The \(r=2\) terms are \(O(\eta)\).  Every odd-degree nonzero pattern has
\(p<r/2\), hence an additional factor \(N^{-1/2}\), while the first
remaining even degree is four.  Thus the absolutely convergent remainder is
\(O(\eta^2)+o_N(1)\) for \(\eta\) small.  This proves
\eqref{eq:R3-eta-mean}.

For the variance, apply Efron--Stein to
\[
 \Delta_\eta
 :=\Tr\varphi(A+D_\eta)-\Tr\varphi(\widetilde H).
\]
Changing one \(\chi_i\) changes \(\Delta_\eta\) by at most
\(C_{\varphi,\Gamma}\sqrt{\eta/N}\), so the sum of its \(N\) diagonal
influences is \(O(\eta)\).  Write an off-diagonal entry as
\(N^{-1/2}\xi_{ab}\), where the \(\xi_{ab}\) are uniformly bounded.  The
first resolvent identity and the contour formula for \(\varphi'\) give
\[
 |\partial_{\xi_{ab}}\Delta_\eta|
 \le C_{\varphi,\Gamma}N^{-1/2}
 \bigl(|c_\eta-1|+\norm{D_\eta}_{\op}\bigr)
 \le C_{\varphi,\Gamma}\frac{\sqrt\eta}{N}
    +O(\eta N^{-3/2}).
\]
Summing the squared influences over the \(O(N^2)\) edges gives \(O(\eta)\),
and proves \eqref{eq:R3-eta-L2}.  A smooth confinement cutoff is identically
one on the spectral confinement event.  Its derivatives are supported on an
exceptional annulus of superpolynomially small probability, so all cutoff
derivative terms above are \(o_N(1)\).  Polarization gives the bounds for a
finite family of test functions.

It remains to identify the endpoint formulas.  For a symmetric stochastic
variance matrix \(T\), set
\begin{equation}\label{eq:R3-LiXu-k4}
 K_4(H):=\sum_{i,j=1}^N
 \left(\E H_{ij}^4-3(\E H_{ij}^2)^2\right).
\end{equation}
In the real-symmetric case, Li--Xu's bias and covariance kernels, in the
resolvent convention of this article, are
\begin{align}
 \mathfrak B_T(z)
 &:=m'(z)m(z)^3
 \left\{\Tr\!\left[T^2(I-m(z)^2T)^{-1}\right]+K_4(H)\right\},
 \label{eq:R3-LiXu-bias-kernel}\\
 \mathfrak K_T(z,w)
 &:=2m'(z)m'(w)
 \Tr\!\left[T(I-m(z)m(w)T)^{-2}\right]
 -m'(z)m'(w)\Tr T\nonumber\\
 &\quad+2K_4(H)m(z)m'(z)m(w)m'(w).
 \label{eq:R3-LiXu-cov-kernel}
\end{align}
The analytic bias is
\(-\frac1{2\pi\ii}\oint_\Gamma\varphi(z)\mathfrak B_T(z)\,\dd z\),
and the covariance is the double contour integral of
\(\mathfrak K_T(z,w)\) against
\((2\pi\ii)^{-2}\varphi(z)\psi(w)\,\dd z\dd w\); these are the
specializations of equations (2.7), (2.9), and (4.21) in
\cite{LiXu2021GeneralizedWignerLES}.

Let \(P=N^{-1}\one\one^{\mathsf T}\) and \(Q=I-P\).  If
\(1=\lambda_1,\lambda_2,\ldots,\lambda_N\) are the eigenvalues of
\(\widetilde S\), then stochasticity and symmetry give
\(\widetilde S=P+Q\widetilde S Q\).  Since
\(S_N^0=P-N^{-1}I\) and
\(\norm{\widetilde S-S_N^0}_{\HS}=o(1)\),
\begin{equation}\label{eq:R3-nonPerron-HS}
 \sum_{a=2}^N\lambda_a^2=o(1),
 \qquad \sum_{a=1}^N\lambda_a=\Tr\widetilde S=0.
\end{equation}
On \(\Gamma\), the bias scalar
\(b_z(t)=t^2/(1-m(z)^2t)\) is \(O_\Gamma(t^2)\).  The covariance scalar,
after the common factors are removed, is
\[
 k_{z,w}(t)=\frac{2t}{(1-m(z)m(w)t)^2}-t
 =t+O_\Gamma(t^2).
\]
Thus \eqref{eq:R3-nonPerron-HS} controls every nonlinear non-Perron term,
while the linear covariance term cancels because both profiles have trace
zero.  The Perron eigenvalues of \(\widetilde S\) and \(S_N^0\) are
respectively \(1\) and \(1-N^{-1}\), and hence contribute only \(O(N^{-1})\)
to their difference.  This proves that the formulas for \(\widetilde S\)
and \(S_N^0\) differ by \(o(1)\).

Because \(I\) commutes with \(\widetilde S\), the eigenvalues in
\eqref{eq:R3-eta-profile} are
\[
 \lambda_a^{(\eta)}=(1-\eta/N)\lambda_a+\eta/N.
\]
The Perron eigenvalue remains one.  The bias change is \(o_N(1)\), while the
linear covariance contribution changes by
\(\Tr\widetilde S^{(\eta)}-\Tr\widetilde S=\eta\); the quadratic remainder
is \(o_N(1)+O(\eta^2/N)\).  Hence the profile formulas change by at most
\(O_\Gamma(\eta)+o_N(1)\).  Finally, diagonal congruence multiplies each
off-diagonal fourth cumulant by \((d_id_j)^2=1+o(N^{-1/2})\), regularization
changes their sum by \(O(\eta/N)\), and the new diagonal cumulants contribute
only \(O(\eta^2/N)\).  Therefore \(K_4(H)\) has the same endpoint limit.

For every centered real linear combination \(X_N^{(\eta)}\) of the analytic
statistics, \eqref{eq:R3-eta-mean}--\eqref{eq:R3-eta-L2} imply
\[
 \lim_{\eta\downarrow0}\limsup_{N\to\infty}
 \mathbb P\bigl(|X_N^{(\eta)}-X_N^{(0)}|>\varepsilon\bigr)=0.
\]
Cram\'er--Wold and the convergence-together theorem transfer the joint CLT
to the zero-diagonal array and complete the proof.
\end{proof}

\begin{lemma}[Explicit flat zero-diagonal endpoint]
\label{lem:R3-flat-endpoint}
Let
\(
 S_N^0=N^{-1}(\one\one^{\mathsf T}-I)
\)
and \(\sigma_N=1-N^{-1}\).  Let \(m_N^0(z)\) be the scalar Dyson solution
\begin{equation}\label{eq:R3-flat-endpoint-QVE}
 -\frac{1}{m_N^0(z)}=z+\sigma_Nm_N^0(z),
 \qquad m_N^0(z)\sim-z^{-1}.
\end{equation}
Uniformly for \(z,w\in\Gamma\),
\begin{align}
 N\bigl(m_N^0(z)-m(z)\bigr)
 &\longrightarrow-\frac{m(z)^3}{1-m(z)^2},
 \label{eq:R3-flat-Dyson-limit}\\
 \Tr\!\left[(S_N^0)^2(I-m(z)^2S_N^0)^{-1}\right]
 &\longrightarrow\frac{1}{1-m(z)^2},
 \label{eq:R3-flat-bias-trace}\\
 \Tr\!\left[S_N^0(I-m(z)m(w)S_N^0)^{-2}\right]
 &\longrightarrow
 \frac{1}{(1-m(z)m(w))^2}-1.
 \label{eq:R3-flat-cov-trace}
\end{align}
Consequently, if the summed fourth cumulants converge to \(\kappa\), the
Li--Xu bias together with the finite-flat Dyson correction converges to
\(m_{\varphi,\Gamma}^{\rm base}(\kappa)\) in
\eqref{eq:base-mean-definition}, and the covariance kernel converges to
\(\mathcal K_{\rm Wig}^{(\kappa,0)}\) in
\eqref{eq:Wigner-kernel}.
\end{lemma}

\begin{proof}
The eigenvalues of \(S_N^0\) are
\begin{equation}\label{eq:R3-flat-spectrum}
 \sigma_N=1-N^{-1}
 \quad\text{on }\operatorname{span}\{\one\},
 \qquad
 -N^{-1}
 \quad\text{with multiplicity }N-1.
\end{equation}
Subtracting the scalar semicircle equation from
\eqref{eq:R3-flat-endpoint-QVE}, or differentiating the stable equation
\(1+zm+\sigma m^2=0\) at \(\sigma=1\), gives
\[
 \left.\frac{\partial m_\sigma(z)}{\partial\sigma}\right|_{\sigma=1}
 =\frac{m(z)^3}{1-m(z)^2}.
\]
Since \(\sigma_N-1=-N^{-1}\), fixed-contour stability gives
\eqref{eq:R3-flat-Dyson-limit}.  Hence
\begin{equation}\label{eq:R3-flat-Dyson-bias}
 -\frac{N}{2\pi\ii}\oint_\Gamma
 \varphi(z)(m_N^0(z)-m(z))\,\dd z
 \longrightarrow
 \frac{1}{2\pi\ii}\oint_\Gamma
 \varphi(z)\frac{m(z)^3}{1-m(z)^2}\,\dd z.
\end{equation}

Functional calculus applied to \eqref{eq:R3-flat-spectrum} gives
\begin{align*}
 &\Tr\!\left[(S_N^0)^2(I-m(z)^2S_N^0)^{-1}\right]\\
 &\quad=
 \frac{\sigma_N^2}{1-m(z)^2\sigma_N}
 +\frac{N-1}{N^2}\frac{1}{1+m(z)^2/N}
 \longrightarrow\frac{1}{1-m(z)^2},
\end{align*}
which proves \eqref{eq:R3-flat-bias-trace}.  Similarly, with
\(a=m(z)m(w)\),
\begin{align*}
 &\Tr\!\left[S_N^0(I-aS_N^0)^{-2}\right]\\
 &\quad=
 \frac{\sigma_N}{(1-a\sigma_N)^2}
 -\frac{\sigma_N}{(1+a/N)^2}
 \longrightarrow\frac{1}{(1-a)^2}-1,
\end{align*}
which proves \eqref{eq:R3-flat-cov-trace}.  Since \(\Tr S_N^0=0\), there is
no additional trace term in \eqref{eq:R3-LiXu-cov-kernel}.  Substitution of
\eqref{eq:R3-flat-bias-trace} into \eqref{eq:R3-LiXu-bias-kernel}, followed by
addition of \eqref{eq:R3-flat-Dyson-bias}, gives
\eqref{eq:base-mean-definition}.  Substitution of
\eqref{eq:R3-flat-cov-trace} into \eqref{eq:R3-LiXu-cov-kernel} gives
\eqref{eq:Wigner-kernel}.
\end{proof}

It remains to compare the exactly stochastic profile with the original weak
profile.  This is the step for which the order-one Dyson-density correction
must be retained.

\begin{lemma}[Averaged off-diagonal gain on a fixed contour]
\label{lem:R3-averaged-offdiag}
Let \(A=A^{\mathsf T}\) and assume
\(\operatorname{dist}(\Gamma,\spec(A))\ge c_\Gamma>0\).  For every fixed
\(r\ge1\) and \(z_1,\ldots,z_r\in\Gamma\),
\begin{equation}\label{eq:R3-deterministic-offdiag}
 \frac1{N^2}\sum_{a\ne b}
 \left|
 \bigl((A-z_1I)^{-1}\cdots(A-z_rI)^{-1}\bigr)_{ab}
 \right|^2
 \le\frac{C_{\Gamma,r}}{N}.
\end{equation}
The same estimate holds in expectation when the spectral separation event has
probability \(1-o(N^{-D})\) for every fixed \(D\), with the complement
controlled by a standard resolvent cutoff.
\end{lemma}

\begin{proof}
All resolvents are functions of the same matrix and hence commute.  If
\[
 R=(A-z_1I)^{-1}\cdots(A-z_rI)^{-1},
\]
then the spectral theorem gives
\[
 \sum_{a,b}|R_{ab}|^2=\Tr RR^*
 \le N\prod_{\ell=1}^r
 \operatorname{dist}(z_\ell,\spec(A))^{-2}
 \le C_{\Gamma,r}N.
\]
Discarding the diagonal terms and dividing by \(N^2\) proves
\eqref{eq:R3-deterministic-offdiag}.  A smooth cutoff of the distance to the
spectrum gives the expectation version.  In the cumulant expansion below,
each differentiated edge term is a finite product of resolvent entries with
indices \(a,b\); after the edge labels are summed, Cauchy--Schwarz reduces it
to \eqref{eq:R3-deterministic-offdiag}.
\end{proof}

\begin{lemma}[Weak-profile transfer by diagonal congruence]
\label{lem:R3-profile-interpolation}
Let \(H\), \(S\), \(D\), and \(\widetilde H\) be as in
Lemma~\ref{lem:R3-stochastic-normalization}.  Suppose in addition that
\begin{equation}\label{eq:R3-congruence-rate}
 \norm{D-I}_\infty=o(N^{-1/2})
\end{equation}
and that the variables \(\sqrt N H_{ij}\) are uniformly bounded.  Let
\(m^S(z)=(m_i^S(z))\) solve
\begin{equation}\label{eq:R3-congruence-QVE}
 -\frac1{m_i^S(z)}=z+\sum_jS_{ij}m_j^S(z),
 \qquad
 \overline m^S(z):=\frac1N\sum_i m_i^S(z),
\end{equation}
and let \(m(z)=m_{\rm sc}(z)\).  On every fixed analytic LSS contour
\(\Gamma\) at positive distance from \([-2,2]\), with the standard spectral
confinement cutoff,
\begin{equation}\label{eq:R3-congruence-L2}
 \sup_{z\in\Gamma}\E\left|
 \left[\Tr(H-zI)^{-1}-N\overline m^S(z)\right]
 -\left[\Tr(\widetilde H-zI)^{-1}-Nm(z)\right]
 \right|^2=o(1).
\end{equation}
Consequently, uniformly for \(z,w\in\Gamma\),
\begin{align}
 &\left[\E\Tr(H-zI)^{-1}-N\overline m^S(z)\right]
 -\left[\E\Tr(\widetilde H-zI)^{-1}-Nm(z)\right]=o(1),
 \label{eq:R3-interpolation-bias}\\
 &\Cov\!\left(\Tr(H-zI)^{-1},\Tr(H-wI)^{-1}\right)
 \nonumber\\
 &\qquad-
 \Cov\!\left(\Tr(\widetilde H-zI)^{-1},
              \Tr(\widetilde H-wI)^{-1}\right)=o(1).
 \label{eq:R3-interpolation-cov}
\end{align}
The same comparison holds for joint characteristic functions of finitely many
analytic LSS after subtracting the exact Dyson centering
\(-N(2\pi\ii)^{-1}\oint\varphi\overline m^S\) at the \(H\) endpoint and the
scalar semicircle centering at the \(\widetilde H\) endpoint.
\end{lemma}

\begin{proof}
Put
\begin{equation}\label{eq:R3-K-definition}
 K:=D-I=\diag(k_1,\ldots,k_N),
 \qquad \delta_N:=\norm{K}_\infty.
\end{equation}
By \eqref{eq:R3-congruence-rate},
\begin{equation}\label{eq:R3-K-rates}
 \delta_N=o(N^{-1/2}),
 \qquad
 \norm{K}_{\HS}^2=\sum_i k_i^2\le N\delta_N^2=o(1).
\end{equation}
Choose a compact neighborhood \(\Gamma_+\) of \(\Gamma\) that remains a
fixed positive distance from \([-2,2]\).  We use the usual smooth confinement
cutoff on which
\(\sup_{z\in\Gamma_+}\norm{(\widetilde H-zI)^{-1}}_{\op}\le C_\Gamma\).
The bounded generalized-Wigner local law makes the complementary event
superpolynomially unlikely, so the cutoff may be removed at the end.

\paragraph{Step 1: exact random congruence.}
Write
\[
 G(z):=(H-zI)^{-1},\qquad
 \widetilde G(z):=(\widetilde H-zI)^{-1},\qquad
 R_D(z):=(\widetilde H-zD)^{-1}.
\]
The exact relation \(\widetilde H=D^{1/2}HD^{1/2}\) gives
\begin{equation}\label{eq:R3-exact-congruence}
 G(z)=D^{1/2}R_D(z)D^{1/2},
 \qquad \Tr G(z)=\Tr DR_D(z).
\end{equation}
Since \(\widetilde H-zD=(\widetilde H-zI)-zK\), two applications of the
resolvent identity give
\begin{equation}\label{eq:R3-RD-expansion}
 R_D=\widetilde G+z\widetilde G K\widetilde G
 +z^2\widetilde G K\widetilde G K R_D.
\end{equation}
On the cutoff event, \(\norm{z\widetilde G K}_{\op}=o(1)\) and hence
\(\norm{R_D}_{\op}\le C_\Gamma\).  Multiplying
\eqref{eq:R3-RD-expansion} by \(D=I+K\), taking the trace, and using cyclicity
yields
\begin{equation}\label{eq:R3-random-congruence-expansion}
 \Tr G(z)=\Tr\widetilde G(z)
 +\Tr K\bigl(\widetilde G(z)+z\widetilde G(z)^2\bigr)
 +\mathcal R_N(z),
\end{equation}
where
\begin{equation}\label{eq:R3-random-congruence-rem}
 \sup_{z\in\Gamma}|\mathcal R_N(z)|
 \le C_\Gamma\norm{K}_{\HS}^2=o(1).
\end{equation}
Indeed, the omitted terms are
\(z\Tr K\widetilde G K\widetilde G\) and
\(z^2\Tr D\widetilde G K\widetilde G K R_D\); both are bounded by
\(C_\Gamma\norm{K}_{\HS}^2\).  The same estimate holds in \(L^2\) after the
cutoff convention.

\paragraph{Step 2: exact deterministic Dyson response.}
Let \(r=(r_i)\) solve the diagonally deformed Dyson equation for the exactly
stochastic profile \(\widetilde S\):
\begin{equation}\label{eq:R3-deformed-Dyson}
 -\frac1{r_i(z)}=zd_i+\sum_j\widetilde S_{ij}r_j(z).
\end{equation}
Since \(\widetilde S_{ij}=d_iS_{ij}d_j\), direct substitution shows that
\begin{equation}\label{eq:R3-m-r-relation}
 m_i^S(z)=d_i r_i(z).
\end{equation}
Write \(r=m\one+u\) and \(k=(k_i)\).  Subtracting the scalar semicircle
equation from \eqref{eq:R3-deformed-Dyson} gives
\begin{equation}\label{eq:R3-u-equation}
 (I-m^2\widetilde S)u
 =zm^2k+\left(\frac{u_i^2}{m+u_i}\right)_{i=1}^N.
\end{equation}
On \(\Gamma_+\), \(\sup|m|\le1-c_\Gamma\).  Since \(\widetilde S\) is a
nonnegative stochastic matrix,
\begin{equation}\label{eq:R3-scalar-stability}
 \norm{(I-m^2\widetilde S)^{-1}}_{\infty\to\infty}
 \le\frac1{1-\sup_{\Gamma_+}|m|^2}\le C_\Gamma.
\end{equation}
A contraction argument in \(\norm{u}_\infty\le C_\Gamma\delta_N\) therefore
gives
\begin{equation}\label{eq:R3-u-linearization}
 u=(I-m^2\widetilde S)^{-1}zm^2k+O_\Gamma(\delta_N^2)
 \quad\text{in }\ell^\infty.
\end{equation}
Symmetry and stochasticity imply
\[
 \one^{\mathsf T}(I-m^2\widetilde S)^{-1}
 =\frac1{1-m^2}\one^{\mathsf T}.
\]
Using \(m'=m^2/(1-m^2)\), summing
\eqref{eq:R3-u-linearization}, and then using
\eqref{eq:R3-m-r-relation}, we obtain
\begin{equation}\label{eq:R3-deterministic-congruence-expansion}
 N\overline m^S(z)
 =Nm(z)+\bigl(m(z)+zm'(z)\bigr)\Tr K
 +O_\Gamma(N\delta_N^2).
\end{equation}
The last remainder is \(o(1)\) by \eqref{eq:R3-K-rates}; explicitly,
the additional term \(\sum_i k_iu_i\) is
\(O_\Gamma(N\delta_N^2)\).

\paragraph{Step 3: weighted diagonal local law.}
Because \(\widetilde S\one=\one\), its vector Dyson solution is exactly
\(m(z)\one\).  The entrywise and averaged fixed-contour local laws, followed
by Cauchy's formula for \(\partial_z\widetilde G=\widetilde G^2\), give
\begin{align}
 \sup_{z\in\Gamma}\sum_i
 \E|\widetilde G_{ii}(z)-m(z)|^2&\le C_\Gamma,
 \label{eq:R3-diagonal-L2-G}\\
 \sup_{z\in\Gamma}\sum_i
 \E|[\widetilde G(z)^2]_{ii}-m'(z)|^2&\le C_\Gamma,
 \label{eq:R3-diagonal-L2-G2}\\
 \sup_{z\in\Gamma}\E|\Tr\widetilde G(z)
 -\E\Tr\widetilde G(z)|^2&\le C_\Gamma.
 \label{eq:R3-trace-L2}
\end{align}
The moment versions follow from the standard stochastic-domination statements
because the normalized entries are uniformly bounded and every resolvent is
cut off on the exceptional event.  Thus Cauchy--Schwarz and
\eqref{eq:R3-K-rates} imply
\begin{equation}\label{eq:R3-weighted-diagonal-L2}
 \sup_{z\in\Gamma}\E\left|
 \Tr K\left[(\widetilde G-mI)
 +z(\widetilde G^2-m'I)\right]\right|^2
 \le C_\Gamma\norm{K}_{\HS}^2=o(1).
\end{equation}

\paragraph{Step 4: completion.}
Subtracting \eqref{eq:R3-deterministic-congruence-expansion} from
\eqref{eq:R3-random-congruence-expansion} gives
\begin{align}
 &\left[\Tr G(z)-N\overline m^S(z)\right]
 -\left[\Tr\widetilde G(z)-Nm(z)\right]\nonumber\\
 &\quad=\Tr K\left[(\widetilde G-mI)
 +z(\widetilde G^2-m'I)\right]
 +O_{L^2,\Gamma}(\norm{K}_{\HS}^2+N\delta_N^2).
 \label{eq:R3-master-congruence-transfer}
\end{align}
Equations \eqref{eq:R3-K-rates} and
\eqref{eq:R3-weighted-diagonal-L2} prove
\eqref{eq:R3-congruence-L2}; taking expectations proves
\eqref{eq:R3-interpolation-bias}.  The uniform variance bound
\eqref{eq:R3-trace-L2}, the \(L^2\) comparison, and Cauchy--Schwarz prove
\eqref{eq:R3-interpolation-cov}.  Applying
$|e^{\ii x}-e^{\ii y}|\le|x-y|$ to real linear combinations of the real and
imaginary parts proves the joint characteristic-function comparison.
Uniformity on \(\Gamma\) permits contour integration and proves the analytic
LSS statement.  No cumulant comparison is used in this transfer; the fourth
cumulant enters only in identifying the endpoint Li--Xu bias and covariance.
\end{proof}

\begin{proposition}[Spike-profile bounded analytic LSS input]
\label{prop:R3-external-input}
Assume W3--W5 and the spike assumption \eqref{eq:spike-package}.  Let
\(H_N=H_N^{(4)}\) be the centered atomic array
\eqref{eq:R3-centered-array}, put
\(S_{ij}:=\E H_N(i,j)^2\), and let \(S_N^0\) be
\eqref{eq:R3-flat-S}.  Then
\begin{align}
 |\sqrt N H_N(i,j)|&\le C,\label{eq:R3-external-bounded}\\
 \max_{i\ne j}|NS_{ij}-1|&\to0,\label{eq:R3-external-flat}\\
 \max_i\left|\sum_jS_{ij}-1\right|&=o(N^{-1/2}),
 \label{eq:R3-external-strong-row}\\
 \max_i\sum_j|S_{ij}-S_{ij}^0|&\to0,\label{eq:R3-external-row}\\
 \norm{S_N-S_N^0}_{\HS}&\to0.\label{eq:R3-external-HS}
\end{align}
Let \(m_{i,N}(z)\) be the vector Dyson-equation solution
\begin{equation}\label{eq:R3-QVE}
 -\frac1{m_{i,N}(z)}
 =z+\sum_jS_{ij}m_{j,N}(z),
 \qquad
 \overline m_N(z):=\frac1N\sum_i m_{i,N}(z),
\end{equation}
and let \(\overline m_N^0\) denote the corresponding scalar solution for
\(S_N^0\).  Moreover,
\begin{equation}\label{eq:R3-kappa-input}
 \max_{i<j}\left|
 \cum_4(\sqrt N H_N(i,j))-\kappa_4^{f,\nu}
 \right|\to0.
\end{equation}
For a function \(\varphi\) analytic near \(\Gamma\), define the
zero-diagonal Dyson-density correction
\begin{equation}\label{eq:R3-zero-Dyson-definition}
 d_{\varphi,\Gamma,N}^{0,\rm Dyson}
 :=-\frac{N}{2\pi\ii}\oint_\Gamma
 \varphi(z)\bigl(\overline m_N^0(z)-m(z)\bigr)\,\dd z.
\end{equation}
Then, for every finite family of such functions, there are deterministic
bounded-endpoint biases \(m_{\varphi,\Gamma,N}^{\rm fluc}\) such that
\begin{equation}\label{eq:R3-external-CLT}
 \left(
 \Tr\varphi_\ell(H_N)
 +\frac{N}{2\pi\ii}\oint_\Gamma
       \varphi_\ell(z)\overline m_N(z)\,\dd z
 -m_{\varphi_\ell,\Gamma,N}^{\rm fluc}
 \right)_\ell
 \Longrightarrow (G_{\varphi_\ell})_\ell .
\end{equation}
The bias not contained in the exact Dyson density satisfies
\begin{equation}\label{eq:R3-fluc-bias-limit}
 m_{\varphi,\Gamma,N}^{\rm fluc}
 +d_{\varphi,\Gamma,N}^{0,\rm Dyson}
 \longrightarrow m_{\varphi,\Gamma}^{\rm base}(\kappa_4^{f,\nu}),
\end{equation}
and the limiting covariance is
\(\mathcal V_{\kappa_4^{f,\nu}}^\Gamma\).  No convergence or homogeneity of
the edgewise third cumulants is assumed.
\end{proposition}

\begin{proof}
Corollary~\ref{cor:atomic-array-data-package} gives independence, uniform
boundedness, and \eqref{eq:R3-kappa-input}.  Lemma~\ref{lem:R3-flatness}
gives \eqref{eq:R3-external-flat}--\eqref{eq:R3-external-HS}, including the
strong signed row estimate.  For an exactly stochastic, uniformly flat
variance matrix with uniformly bounded moments, the global statement centered
by its expectation, together with the explicit expectation and covariance
formulas, is the global case of Li--Xu's generalized-Wigner LSS theorem
\cite[Theorem~2.2]{LiXu2021GeneralizedWignerLES}.  Their
characteristic-function expansion permits arbitrary edge-dependent third
cumulants; the limiting formulas contain the variance matrix and the averaged
fourth cumulants, but no third-cumulant parameter.

Lemma~\ref{lem:R3-stochastic-normalization} produces an exactly stochastic
profile \(\widetilde S=DSD\) and a coupled bounded array \(\widetilde H\).
Lemma~\ref{lem:R3-diagonal-regularization} then permits application of
Li--Xu's theorem to \(\widetilde H^{(\eta)}\) for fixed \(\eta>0\), followed by
\(N\to\infty\) and \(\eta\downarrow0\).  This gives the zero-diagonal
generalized-Wigner CLT for \(\widetilde H\).

For the present spike profile,
\[
 \norm{\widetilde S-S_N^0}_{\HS}
 \le\norm{\widetilde S-S_N}_{\HS}
   +\norm{S_N-S_N^0}_{\HS}=o(1).
\]
The exact Li--Xu kernels are
\eqref{eq:R3-LiXu-bias-kernel}--\eqref{eq:R3-LiXu-cov-kernel}.  The
non-Perron trace-series calculation in
Lemma~\ref{lem:R3-diagonal-regularization} shows that the displayed
Hilbert--Schmidt estimate makes those kernels differ from the flat
zero-diagonal kernels by \(o(1)\).  Lemma~\ref{lem:R3-flat-endpoint} identifies
the flat bias and covariance, including the non-Perron subtraction in
\eqref{eq:Wigner-kernel}.  Lemma~\ref{lem:R3-diagonal-regularization} also
proves continuity through the positive-diagonal regularization and the
convergence-together step.  Uniform fourth-cumulant convergence identifies
\(\kappa_4^{f,\nu}\).

Lemma~\ref{lem:R3-profile-interpolation} transfers the centered field from
\(\widetilde H\) to \(H_N\).  The signed row estimate
\eqref{eq:R3-external-strong-row} and
Lemma~\ref{lem:R3-stochastic-normalization} imply
\(\lVert D-I\rVert_\infty=o(N^{-1/2})\), so the exact diagonal-congruence
comparison applies.  It replaces the scalar stochastic Dyson density by the
exact vector-Dyson density \(\overline m_N\), while the centered resolvent
fields differ by \(o_{L^2}(1)\).  Hence all remaining changes in the bias and
covariance are \(o(1)\).

At the flat zero-diagonal profile, Lemma~\ref{lem:R3-flat-endpoint} identifies
the total order-one mean relative to scalar semicircle centering as
\(m_{\varphi,\Gamma}^{\rm base}(\kappa_4^{f,\nu})\).  Splitting that mean into its
Dyson-density part \eqref{eq:R3-zero-Dyson-definition} and its fluctuation
part gives the sign-explicit relation \eqref{eq:R3-fluc-bias-limit}.  This
proves \eqref{eq:R3-external-CLT}--\eqref{eq:R3-fluc-bias-limit}.
\end{proof}

\begin{remark}[What is imported]
The Gaussian fluctuation mechanism and the formulas for expectation and
covariance
for the stochastically normalized bounded array are imported.  The reduction
from the model's non-stochastic weak profile, including the order-one
quadratic response, is proved below.  The proposition is not an assumption on
\(f\), \(\nu\), or \(x\).
\end{remark}

\subsection{Atomic-array verification and bounded LSS theorem}

\begin{lemma}[Flatness, signed row control, and Hilbert--Schmidt smallness]
\label{lem:R3-flatness}
The variance matrix \eqref{eq:R3-variance-matrix} satisfies
\eqref{eq:R3-external-flat}--\eqref{eq:R3-external-HS}.  More precisely,
\begin{align}
 \max_{i\ne j}|NS_{ij}-1|&=o(1),\label{eq:R3-flat-a}\\
 \max_i\left|\sum_jS_{ij}-1\right|&=o(N^{-1/2}),
 \label{eq:R3-flat-strong-row}\\
 \max_i\sum_j|S_{ij}-S_{ij}^0|&=o(1),\label{eq:R3-flat-b}\\
 \norm{S_N-S_N^0}_{\HS}^2&=O(N^{-1})+o(N^{-1}).
 \label{eq:R3-flat-c}
\end{align}
\end{lemma}

\begin{proof}
The maximum estimate follows from
Proposition~\ref{prop:spike-rates-package} and
Corollary~\ref{cor:profile-remainder-package}.  The signed row estimate is
exactly \eqref{eq:R3-strong-q}.  For the absolute row estimate,
\[
 \frac1N\sum_j|p_{ij}^{(1)}|
 \le C|x_i|=o(N^{-1/4}),\qquad
 \frac1N\sum_j|p_{ij}^{(2)}|=O(x_i^2)=o(N^{-1/2}),
\]
uniformly in \(i\), and the remainder is smaller than the quadratic term.
Finally,
\[
 \norm{S_N-S_N^0}_{\HS}^2
 =\frac1{N^2}\sum_{i\ne j}
 |p_{ij}^{(1)}+p_{ij}^{(2)}+r_{ij,N}|^2.
\]
The squared linear contribution is \(O(N^{-1})\), while the squared
quadratic and remainder contributions are \(o(N^{-1})\), by
Proposition~\ref{prop:spike-rates-package}.
\end{proof}

\begin{lemma}[First variation of the model Dyson equation]
\label{lem:R3-Dyson-response}
Let \(\overline m_N\) and \(\overline m_N^0\) be as in
Proposition~\ref{prop:R3-external-input}, for \(S_N\) and \(S_N^0\),
respectively.  Uniformly on \(\Gamma\),
\begin{equation}\label{eq:R3-Dyson-limit}
 N\bigl(\overline m_N(z)-\overline m_N^0(z)\bigr)
 \longrightarrow
 \lambda\beta_2^{f,\nu}
 \frac{m(z)^3}{1-m(z)^2}.
\end{equation}
The linear profile and the remainder have zero order-one response.
\end{lemma}

\begin{proof}
Put \(\Delta=S_N-S_N^0\),
\(\sigma_N=(N-1)/N\), and write
\(m_{i,N}=m_N^0+\delta_i\), where \(m_N^0=\overline m_N^0\) solves
\begin{equation}\label{eq:R3-finite-flat-QVE}
 -\frac1{m_N^0(z)}=z+\sigma_Nm_N^0(z).
\end{equation}
Let \(h=\Delta\one\) and
\(\overline\delta=N^{-1}\one^{\mathsf T}\delta\).  We first record the
profile estimates needed below.  The linear and quadratic parts of
\(\Delta\) are
\begin{align*}
 \Delta^{(1)}
 &=\frac{b_1\sqrt\lambda}{\sqrt N}
 \left(xx^{\mathsf T}-\diag(x_i^2)\right),\\
 \Delta^{(2)}
 &=\lambda\beta_2^{f,\nu}
 \left(x^{\circ2}(x^{\circ2})^{\mathsf T}-\diag(x_i^4)\right).
\end{align*}
Therefore
\[
 \norm{\Delta^{(1)}}_{\op}=O(N^{-1/2}),
 \qquad
 \norm{\Delta^{(2)}}_{\op}
 \le C\sum_i x_i^4=o(N^{-1/2}).
\]
The remainder obeys
\[
 \max_i\sum_j|\Delta^{(r)}_{ij}|
 \le C\varepsilon_N\max_i x_i^2\sum_jx_j^2
 =o(N^{-1/2}).
\]
Consequently
\begin{equation}\label{eq:R3-Delta-op}
 \norm\Delta_{\op}=O(N^{-1/2}).
\end{equation}
For the row-sum vector, the spike assumptions give
\begin{align*}
 \norm{h^{(1)}}_2^2
 &\le\frac C N
 \left\{\left(\sum_i x_i\right)^2+\sum_i x_i^4\right\}
 =N^{-1+o(1)},\\
 \norm{h^{(2)}}_2^2
 &\le C\sum_i x_i^4=o(N^{-1/2}),\\
 \norm{h^{(r)}}_2^2
 &\le C\varepsilon_N^2\sum_i x_i^4=o(N^{-1/2}).
\end{align*}
Thus
\begin{equation}\label{eq:R3-Delta-one}
 \norm{\Delta\one}_2^2=o(N^{-1/2}).
\end{equation}
Finally,
\begin{align}
 \one^{\mathsf T}\Delta^{(1)}\one
 &=\frac{b_1\sqrt\lambda}{\sqrt N}
 \left[\left(\sum_i x_i\right)^2-1\right]=o(1),
 \nonumber\\
 \one^{\mathsf T}\Delta^{(2)}\one
 &=\lambda\beta_2^{f,\nu}
 \left(1-\sum_i x_i^4\right)
 \longrightarrow\lambda\beta_2^{f,\nu},
 \nonumber\\
 |\one^{\mathsf T}\Delta^{(r)}\one|
 &\le C\varepsilon_N\left(\sum_i x_i^2\right)^2=o(1).
 \label{eq:R3-Delta-mass}
\end{align}

Subtracting \eqref{eq:R3-finite-flat-QVE} from the vector Dyson equation
gives the exact identity
\begin{equation}\label{eq:R3-exact-Dyson-delta}
 \bigl(I-(m_N^0)^2S_N^0\bigr)\delta
 =(m_N^0)^3h+(m_N^0)^2\Delta\delta
 +\frac{\delta^{\circ2}}{m_N^0+\delta}.
\end{equation}
All estimates are uniform on \(\Gamma\).  Fixed-contour QVE stability and
\eqref{eq:R3-flat-a} first give \(\norm\delta_\infty=o(1)\), while
\[
 \norm{\bigl(I-(m_N^0)^2S_N^0\bigr)^{-1}}_{\op}
 \le C_\Gamma.
\]
Taking Euclidean norms in \eqref{eq:R3-exact-Dyson-delta}, using
\eqref{eq:R3-Delta-op}, and absorbing
\(C_\Gamma(\norm\Delta_{\op}+\norm\delta_\infty)\norm\delta_2\), yields
\begin{equation}\label{eq:R3-delta-two}
 \norm\delta_2\le C_\Gamma\norm h_2,
 \qquad \norm\delta_2^2=o(N^{-1/2}).
\end{equation}

Since \(S_N^0\one=\sigma_N\one\), averaging the exact identity gives
\begin{align}
 \bigl(1-\sigma_N(m_N^0)^2\bigr)\overline\delta
 &=(m_N^0)^3\frac1N\one^{\mathsf T}h
 +(m_N^0)^2\frac1N h^{\mathsf T}\delta\nonumber\\
 &\quad+\frac1N\sum_i\frac{\delta_i^2}{m_N^0+\delta_i}.
 \label{eq:R3-averaged-Dyson-delta}
\end{align}
The last two terms are \(o(N^{-1})\), because
\[
 |h^{\mathsf T}\delta|
 \le\norm h_2\norm\delta_2=o(N^{-1/2}),
 \qquad
 \sum_i|\delta_i|^2=o(N^{-1/2}).
\]
Moreover, \(N^{-1}\one^{\mathsf T}h=N^{-1}\one^{\mathsf T}\Delta\one\).
Multiplying \eqref{eq:R3-averaged-Dyson-delta} by \(N\), using
\eqref{eq:R3-Delta-mass}, \(m_N^0\to m\), and
\(1-\sigma_N(m_N^0)^2\to1-m^2\), proves
\eqref{eq:R3-Dyson-limit}.  The three separate limits in
\eqref{eq:R3-Delta-mass} show that only the quadratic profile contributes at
order one.
\end{proof}

\begin{theorem}[Bounded atomic-array LSS CLT]
\label{thm:R3-bounded-CLT}
Assume W2--W5 and the spike assumption
\eqref{eq:spike-package}.  For analytic test functions
\(\varphi_1,\ldots,\varphi_k\) that are real-valued on the real axis,
\begin{equation}\label{eq:R3-final-CLT}
 \left(
 \Tr\varphi_\ell(H_N^{(4)})
 -N\int_{-2}^2\varphi_\ell(t)\rho_{\scal}(t)\,\dd t
 -m_{\varphi_\ell,\Gamma}^{\rm base}(\kappa_4^{f,\nu})
 -c_{\varphi_\ell,\Gamma}^{\rm quad-var}
 \right)_\ell
 \Longrightarrow (G_{\varphi_\ell})_\ell ,
\end{equation}
where
\[
 \Cov(G_\varphi,G_\psi)
 =\mathcal V_{\kappa_4^{f,\nu}}^\Gamma(\varphi,\psi).
\]
The edgewise third cumulants may vary arbitrarily subject only to the uniform boundedness supplied by
Theorem~\ref{thm:uniform-four-surrogate-package}.
\end{theorem}

\begin{proof}
Proposition~\ref{prop:R3-external-input} gives the joint CLT after exact
vector-Dyson centering.  Put
\[
 d_{\varphi,\Gamma,N}^{S,\rm Dyson}
 :=-\frac{N}{2\pi\ii}\oint_\Gamma
 \varphi(z)\bigl(\overline m_N(z)-m(z)\bigr)\,\dd z.
\]
With the resolvent convention \eqref{eq:bulk-contour-formula-package}, the
total order-one bias in that proposition is
\(d_{\varphi,\Gamma,N}^{S,\rm Dyson}
+m_{\varphi,\Gamma,N}^{\rm fluc}\).  Lemma~\ref{lem:R3-Dyson-response}
gives
\[
 d_{\varphi,\Gamma,N}^{S,\rm Dyson}
 -d_{\varphi,\Gamma,N}^{0,\rm Dyson}
 \longrightarrow
 -\frac1{2\pi\ii}\oint_\Gamma\varphi(z)
 \lambda\beta_2^{f,\nu}\frac{m(z)^3}{1-m(z)^2}\,\dd z
 =c_{\varphi,\Gamma}^{\rm quad-var}.
\]
Equation \eqref{eq:R3-fluc-bias-limit} gives
\(d_{\varphi,\Gamma,N}^{0,\rm Dyson}
+m_{\varphi,\Gamma,N}^{\rm fluc}
\to m_{\varphi,\Gamma}^{\rm base}(\kappa_4^{f,\nu})\).  Adding these two limits proves the
centering in \eqref{eq:R3-final-CLT}; the covariance and joint convergence
are those of Proposition~\ref{prop:R3-external-input}.
\end{proof}

\subsection{Fixed-contour estimates and deterministic mean response}
\label{sec:R4-package}

Let \(K_\Gamma\) be a compact neighborhood of \(\Gamma\) with positive
distance from \([-2,2]\) and, in the supercritical case, from
\(\rho_\theta\).  The variance flatness established in
Lemma~\ref{lem:R3-flatness} places the centered surrogate in the standard
mean-field Wigner-type class.

\begin{proposition}[Bounded fixed-contour local law]
\label{prop:R4-bounded-local-law}
Let \(G_N^{(4)}(z)=(H_N^{(4)}-zI)^{-1}\).  Uniformly for \(z\in K_\Gamma\)
and deterministic unit vectors \(u,v\),
\begin{align}
 u^{\mathsf T}G_N^{(4)}(z)v
 &=
 \sum_i\overline u_i\,m_{i,N}(z)v_i+o_{\mathbb P}(1),
 \label{eq:R4-anisotropic}\\
 \frac1N\Tr G_N^{(4)}(z)&=\overline m_N(z)+O_{\mathbb P}(N^{-1}),
 \label{eq:R4-averaged}\\
 \max_i|G_N^{(4)}(z)_{ii}-m_{i,N}(z)|
 +\max_{i\ne j}|G_N^{(4)}(z)_{ij}|
 &=o_{\mathbb P}(1).
 \label{eq:R4-entrywise}
\end{align}
The same assertions hold after one \(z\)-derivative.  In particular,
\begin{equation}\label{eq:R4-x-law}
 x^{\mathsf T}G_N^{(4)}(z)x=m(z)+o_{\mathbb P}(1),
 \qquad
 x^{\mathsf T}G_N^{(4)}(z)^2x=m'(z)+o_{\mathbb P}(1).
\end{equation}
\end{proposition}

\begin{proof}
The entries are independent, centered, and uniformly bounded after
multiplication by \(\sqrt N\), by
Corollary~\ref{cor:atomic-array-data-package}.  Lemma~\ref{lem:R3-flatness}
gives uniform flatness and stability of the vector Dyson equation.
The isotropic generalized-Wigner law was established in
\cite{BloemendalErdosKnowlesYauYin2014}; an isotropic self-consistent
formulation for broader mean-field ensembles appears in
\cite{HeKnowlesRosenthal2018}.
The Wigner-type anisotropic local law therefore gives
\eqref{eq:R4-anisotropic}--\eqref{eq:R4-entrywise}
\cite{AEK2017WignerType}.  Since
\(\max_i|m_{i,N}(z)-m(z)|=o(1)\) on \(K_\Gamma\), the first identity in
\eqref{eq:R4-x-law} follows from \(\norm{x}_2=1\).  Apply Cauchy's integral
formula on a small circle contained in \(K_\Gamma\) to obtain the
\(z\)-derivative estimates and the second identity.
\end{proof}

\begin{lemma}[Small deterministic insertions]
\label{lem:R4-small-insertion}
Suppose \(B_N\) is deterministic,
\(\Tr B_N=O(1)\), \(\norm{B_N}_{\op}=o(1)\), and
\(\norm{B_N}_{\HS}=o(1)\).  Uniformly on \(K_\Gamma\),
\begin{align}
 \Tr B_NG_N^{(4)}(z)&=m(z)\Tr B_N+o_{\mathbb P}(1),\label{eq:R4-BG}\\
 \Tr B_NG_N^{(4)}(z)^2&=m'(z)\Tr B_N+o_{\mathbb P}(1).
 \label{eq:R4-BG2}
\end{align}
Moreover, every fixed-contour trace containing at least two factors \(B_N\)
is \(o_{\mathbb P}(1)\).
\end{lemma}

\begin{proof}
Take a singular-value decomposition
\(B_N=\sum_{k=1}^N s_kp_kq_k^{\mathsf T}\).  The isotropic local law on a
slightly larger fixed compact set gives, simultaneously for these
deterministic pairs,
\[
 \left|p_k^{\mathsf T}(G_N^{(4)}(z)-m(z)I)q_k\right|
 =O_\prec(N^{-1/2}).
\]
The high-probability estimate is uniform in \(k\) by a union bound.  Since
\(\sum_ks_k=\norm{B_N}_*\le\sqrt N\norm{B_N}_{\HS}\),
\begin{equation}\label{eq:R4-deterministic-observable-law}
 \sup_{z\in K_\Gamma}
 \left|\Tr B_N(G_N^{(4)}(z)-m(z)I)\right|
 =O_\prec(\norm{B_N}_{\HS})=o_{\mathbb P}(1).
\end{equation}
This proves \eqref{eq:R4-BG}.  Cauchy's formula applied on small circles
contained in the larger compact set differentiates
\eqref{eq:R4-deterministic-observable-law} and gives
\eqref{eq:R4-BG2}.  For the higher terms, if the intervening resolvents have
uniformly bounded operator norm,
\[
 |\Tr(R_1B_NR_2B_NR_3)|
 \le C_\Gamma\norm{B_N}_{\HS}^2=o(1).
\]
Additional \(B_N\) factors are handled identically.
\end{proof}

Recall \(A_N(i,j)=N^{-1/2}m(s_{ij})\) for \(i\ne j\), with zero diagonal.
Put \(D_x:=\diag(x_1^2,\ldots,x_N^2)\) and
\(\theta=\theta_{f,\nu}\).  W2 gives the exact decomposition
\begin{equation}\label{eq:R4-mean-decomp}
 A_N=\theta xx^{\mathsf T}-\theta D_x+R_N^{(m)},
\end{equation}
where \(R_N^{(m)}\) is zero diagonal and
\begin{equation}\label{eq:R4-mean-rem}
 \Tr R_N^{(m)}=0,\qquad
 \norm{R_N^{(m)}}_{\HS}=o(1),\qquad
 \norm{R_N^{(m)}}_{\op}=o(1).
\end{equation}
Indeed, the off-diagonal entries of \(R_N^{(m)}\) are
\(\sqrt\lambda x_ix_j\varepsilon_{ij}\), with
\(\max_{i\ne j}|\varepsilon_{ij}|\to0\).

\begin{proposition}[Rank-one, diagonal, and mean-remainder response]
\label{prop:R4-deterministic-response}
Let \(B_N^{(4)}=H_N^{(4)}+\theta xx^{\mathsf T}\).  Under pole avoidance,
uniformly on \(\Gamma\),
\begin{equation}\label{eq:R4-Woodbury-trace}
 \Tr(B_N^{(4)}-zI)^{-1}-\Tr G_N^{(4)}(z)
 =
 -\frac{\theta m'(z)}{1+\theta m(z)}
 +o_{\mathbb P}(1).
\end{equation}
Adding \(-\theta D_x+R_N^{(m)}\) changes the analytic LSS by
\begin{equation}\label{eq:R4-diag-response}
 c_{\varphi,\Gamma}^{\rm diag0}+o_{\mathbb P}(1),
\end{equation}
and \(R_N^{(m)}\) has zero response.
\end{proposition}

\begin{proof}
Woodbury's identity gives
\[
 (B_N^{(4)}-zI)^{-1}
 =G_N^{(4)}(z)
 -\frac{\theta G_N^{(4)}(z)xx^{\mathsf T}G_N^{(4)}(z)}
 {1+\theta x^{\mathsf T}G_N^{(4)}(z)x}.
\]
Taking the trace and using \eqref{eq:R4-x-law} proves
\eqref{eq:R4-Woodbury-trace}.  The contour convention
\eqref{eq:bulk-contour-formula-package} then gives the direct rank-one
response \eqref{eq:rank-one-direct-package}.

We next prove the insertion formula after, rather than before, the rank-one
deformation.  Write
\(R_\theta(z)=(B_N^{(4)}-zI)^{-1}\).  Pole avoidance, Woodbury's formula,
and \eqref{eq:R4-x-law} imply
\(\sup_{z\in\Gamma}\norm{R_\theta(z)}_{\op}=O_{\mathbb P}(1)\).
For either \(B=-\theta D_x\) or \(B=R_N^{(m)}\), two resolvent identities
give
\begin{align}
 &(B_N^{(4)}+B-zI)^{-1}-R_\theta(z)\nonumber\\
 &\qquad=-R_\theta BR_\theta
 +R_\theta BR_\theta B(B_N^{(4)}+B-zI)^{-1}.
 \label{eq:R4-post-spike-resolvent}
\end{align}
Since \(\norm B_{\op}=o(1)\), the Neumann series based at \(R_\theta\)
also gives
\(\sup_{z\in\Gamma}\norm{(B_N^{(4)}+B-zI)^{-1}}_{\op}
=O_{\mathbb P}(1)\).
The trace of the second term is
\(O_{\mathbb P}(\norm B_{\HS}^2)=o_{\mathbb P}(1)\).

Put \(Q_\theta=R_\theta-G_N^{(4)}\).  Woodbury's formula shows that it has
rank one and \(\norm{Q_\theta}_{\op}=O_{\mathbb P}(1)\).  Hence
\begin{align*}
 &\left|\Tr(R_\theta BR_\theta)
          -\Tr(G_N^{(4)}BG_N^{(4)})\right|\\
 &\quad\le
 |\Tr(Q_\theta BG_N^{(4)})|
 +|\Tr(G_N^{(4)}BQ_\theta)|
 +|\Tr(Q_\theta BQ_\theta)|
 \le O_{\mathbb P}(\norm B_{\op})=o_{\mathbb P}(1).
\end{align*}
By cyclicity and \eqref{eq:R4-BG2},
\[
 \Tr(G_N^{(4)}BG_N^{(4)})
 =\Tr(B(G_N^{(4)})^2)
 =m'(z)\Tr B+o_{\mathbb P}(1).
\]
Thus \eqref{eq:R4-post-spike-resolvent} yields the uniform master formula
\begin{equation}\label{eq:R4-post-spike-master}
 \Tr(B_N^{(4)}+B-zI)^{-1}-\Tr R_\theta(z)
 =-m'(z)\Tr B+o_{\mathbb P}(1).
\end{equation}

For \(B=-\theta D_x\),
\[
 \Tr B=-\theta,\qquad
 \norm B_{\op}=|\theta|\norm{x}_\infty^2=o(1),\qquad
 \norm B_{\HS}^2=\theta^2\sum_i x_i^4=o(1).
\]
Equation \eqref{eq:R4-post-spike-master} gives the resolvent response
\(\theta m'(z)+o_{\mathbb P}(1)\).  Contour integration gives
\[
 -\frac{\theta}{2\pi\ii}\oint_\Gamma\varphi(z)m'(z)\,\dd z
 =\frac{\theta}{2\pi\ii}\oint_\Gamma\varphi'(z)m(z)\,\dd z,
\]
which is \eqref{eq:diag-package}.  For \(B=R_N^{(m)}\),
\eqref{eq:R4-mean-rem} verifies the same norm hypotheses and
\(\Tr B=0\); \eqref{eq:R4-post-spike-master} therefore proves zero response
for the mean remainder.  This also shows explicitly that no mixed
rank-one--remainder term survives.
\end{proof}

\begin{corollary}[Bounded transformed-spike bulk CLT]
\label{cor:R4-bounded-full-CLT}
The conclusion of Theorem~\ref{thm:bulk-main-package} holds with
\(M_N^f\) replaced by \(M_N^{(4)}=A_N+H_N^{(4)}\).
\end{corollary}

\begin{proof}
Combine Theorem~\ref{thm:R3-bounded-CLT} with
Proposition~\ref{prop:R4-deterministic-response}.  The rank-one and diagonal
responses add to the homogeneous and quadratic terms identified in
Theorem~\ref{thm:R3-bounded-CLT}.
All random errors are \(o_{\mathbb P}(1)\), so joint convergence follows by
Slutsky's theorem.
\end{proof}

\subsection{Confinement, separation, and smooth bulk cutoff}

\begin{lemma}[Diagonal truncation extracted from W5]
\label{lem:R4-diagonal-truncation}
There is a deterministic sequence \(\eta_N\downarrow0\) such that, if
\[
 Y_{ij}=Y_{s_{ij}},\qquad
 Y_{ij}^{\rm tr}
 :=Y_{ij}\one_{\{|Y_{ij}|\le\eta_N\sqrt N\}}
 -\E\!\left[Y_{ij}\one_{\{|Y_{ij}|\le\eta_N\sqrt N\}}\right],
\]
then
\begin{equation}\label{eq:R4-uniform-tail-choice}
 T_N:=\max_{i<j}
 \E\!\left[|Y_{ij}|^4
 \one_{\{|Y_{ij}|>\eta_N\sqrt N\}}\right]
 =o(\eta_N^4).
\end{equation}
The target centered matrix and the matrix formed from
\(N^{-1/2}Y_{ij}^{\rm tr}\) differ by \(o_{\mathbb P}(1)\) in operator norm,
and, for all sufficiently large \(N\),
\begin{align}
 \max_{i<j}|Y_{ij}^{\rm tr}|
 &\le2\eta_N\sqrt N,\label{eq:R4-trunc-max}\\
 \max_{i<j}\left|
 \E Y_{ij}^2-\E(Y_{ij}^{\rm tr})^2
 \right|&=o(N^{-1}).\label{eq:R4-trunc-variance}
\end{align}
\end{lemma}

\begin{proof}
For every fixed \(\eta>0\),
Lemma~\ref{lem:exact-W5-Lindeberg-package} gives
\[
 \max_{i<j}
 \E\!\left[|Y_{ij}|^4
 \one_{\{|Y_{ij}|>\eta\sqrt N\}}\right]\longrightarrow0.
\]
A diagonal choice gives \(\eta_N\downarrow0\) satisfying
\eqref{eq:R4-uniform-tail-choice}.  Therefore
\[
 \mathbb P\!\left(\max_{i<j}|Y_{ij}|>\eta_N\sqrt N\right)
 \le \frac{T_N}{\eta_N^4}=o(1).
\]
On the event in the preceding display, the only difference is the
deterministic recentering matrix.  Put
\(d_{ij}:=\E[Y_{ij}\one_{\{|Y_{ij}|>\eta_N\sqrt N\}}]\).  Since
\[
 d_{ij}^2
 \le \E\!\left[Y_{ij}^2
       \one_{\{|Y_{ij}|>\eta_N\sqrt N\}}\right]
 \le \frac{1}{\eta_N^2N}
 \E\!\left[|Y_{ij}|^4
       \one_{\{|Y_{ij}|>\eta_N\sqrt N\}}\right],
\]
we have, uniformly in the edges,
\begin{equation}\label{eq:R4-tail-second}
 \E\!\left[Y_{ij}^2
 \one_{\{|Y_{ij}|>\eta_N\sqrt N\}}\right]
 +d_{ij}^2
 \le \frac{2T_N}{\eta_N^2N}
 =o(N^{-1}).
\end{equation}
The squared Hilbert--Schmidt norm of the matrix with entries
\(N^{-1/2}d_{ij}\) is at most
\(C T_N/\eta_N^2=o(\eta_N^2)\).  Its operator norm is therefore \(o(1)\).
The centering term satisfies
\(|d_{ij}|=o(\eta_NN^{-1/2})\), which gives
\eqref{eq:R4-trunc-max}.  Finally,
\[
 \E(Y_{ij}^{\rm tr})^2
 =\E\!\left[Y_{ij}^2
       \one_{\{|Y_{ij}|\le\eta_N\sqrt N\}}\right]-d_{ij}^2,
\]
and \eqref{eq:R4-tail-second} proves
\eqref{eq:R4-trunc-variance}.
\end{proof}

\begin{lemma}[Triangular-array Bai--Yin bound]
\label{lem:R4-triangular-Bai-Yin}
Let \(Z_{ij,N}\), \(1\le i<j\le N\), be independent centered variables and
let \(W_N(i,j)=N^{-1/2}Z_{ij,N}\), \(W_N(i,i)=0\).  Suppose
\begin{align}
 \max_{i<j}\left|\E Z_{ij,N}^2-1\right|&\longrightarrow0,
 \label{eq:R4-BY-variance}\\
 \lim_{R\to\infty}\limsup_{N\to\infty}
 \max_{i<j}
 \E\!\left[|Z_{ij,N}|^4
 \one_{\{|Z_{ij,N}|>R\}}\right]&=0.
 \label{eq:R4-BY-tail}
\end{align}
Then
\begin{equation}\label{eq:R4-BY-conclusion}
 \norm{W_N}_{\op}\longrightarrow2
 \qquad\text{in probability}.
\end{equation}
The conclusion remains valid uniformly over a family of triangular arrays
for which the two moduli in
\eqref{eq:R4-BY-variance}--\eqref{eq:R4-BY-tail} are uniform.
\end{lemma}

\begin{proof}
This is the triangular-array version of the classical Bai--Yin spectral-edge
bound \cite{BaiYin1988}.  Related finite-moment spectral-norm estimates are
developed in \cite{Vu2007SpectralNorm}; for the non-identically distributed
array needed here, the precise input is the result cited below.
Condition \eqref{eq:R4-BY-tail} implies, after discarding finitely many
\(N\),
\begin{equation}\label{eq:R4-BY-uniform-fourth}
 \sup_N\max_{i<j}\E|Z_{ij,N}|^4<\infty.
\end{equation}
Indeed, choose a fixed \(R\) for which the tail term is bounded, and add the
contribution at most \(R^4\) below the cutoff.

Put
\[
 \sigma_{ij,N}:=(\E Z_{ij,N}^2)^{1/2},
 \qquad X_{ij,N}:=Z_{ij,N}/\sigma_{ij,N}.
\]
For large \(N\), \(1/2\le\sigma_{ij,N}\le2\) uniformly.  The
\(X_{ij,N}\) are independent, centered, have variance one, uniformly
bounded fourth moments, and satisfy
\begin{equation}\label{eq:R4-ORS-Lindeberg}
 \frac1{N^2}\sum_{i<j}
 \E\!\left[|X_{ij,N}|^4
 \one_{\{|X_{ij,N}|>\varepsilon\sqrt N\}}\right]
 \longrightarrow0
\end{equation}
for every \(\varepsilon>0\).  To verify this, bound the average by its
maximum, replace the threshold by \(\varepsilon\sqrt N/2\), and use
\eqref{eq:R4-BY-tail}.

Let \(\overline W_N(i,j)=N^{-1/2}X_{ij,N}\), with zero diagonal.
Proposition~2.1 of O'Rourke--Renfrew--Soshnikov
\cite{ORS2013} is the non-identically distributed Bai--Yin theorem under
\eqref{eq:R4-ORS-Lindeberg}, and it explicitly permits zero diagonal.
Therefore
\begin{equation}\label{eq:R4-BY-unit-variance}
 \norm{\overline W_N}_{\op}\longrightarrow2
 \qquad\text{in probability}.
\end{equation}

It remains to remove the edgewise variance normalization.  Set
\(\delta_N=\max_{i<j}|\sigma_{ij,N}-1|=o(1)\), and let \(U_N\) be the
strictly upper-triangular matrix with
\[
 U_N(i,j)=N^{-1/2}(\sigma_{ij,N}-1)X_{ij,N},\qquad i<j.
\]
Then \(W_N-\overline W_N=U_N+U_N^{\mathsf T}\).  Lata\l a's expected-norm
inequality for matrices with independent centered entries \cite{Latala2005}
gives
\begin{align*}
 \E\norm{U_N}_{\op}
 &\le C\left\{
 \max_i\left(\sum_j\E|U_N(i,j)|^2\right)^{1/2}
 +\max_j\left(\sum_i\E|U_N(i,j)|^2\right)^{1/2}\right.\\
 &\hspace{29mm}\left.
 +\left(\sum_{i,j}\E|U_N(i,j)|^4\right)^{1/4}
 \right\}
 \le C\delta_N,
\end{align*}
where \eqref{eq:R4-BY-uniform-fourth} controls the last term.  Hence
\[
 \E\norm{W_N-\overline W_N}_{\op}
 \le2\E\norm{U_N}_{\op}=o(1).
\]
Together with \eqref{eq:R4-BY-unit-variance}, this proves
\eqref{eq:R4-BY-conclusion}.

For uniformity over a family, argue by contradiction.  Failure of uniform
convergence would select a subsequence \(N_k\) and one family member at each
\(N_k\) violating the conclusion with a fixed positive probability.  The
selected triangular array still satisfies
\eqref{eq:R4-BY-variance}--\eqref{eq:R4-BY-tail}, because their moduli are
uniform.  The preceding argument then gives a contradiction.
\end{proof}

\begin{proposition}[Centered spectral confinement]
\label{prop:R4-centered-confinement}
For both centered arrays,
\begin{equation}\label{eq:R4-norm-convergence}
 \norm{H_N^{(4)}}_{\op}\longrightarrow2,
 \qquad
 \norm{H_N^f}_{\op}\longrightarrow2
 \quad\text{in probability}.
\end{equation}
Here \(H_N^f(i,j)=N^{-1/2}Y_{s_{ij}}\) and \(H_N^f(i,i)=0\).
\end{proposition}

\begin{proof}
For \(H_N^{(4)}\), Theorem~\ref{thm:uniform-four-surrogate-package} gives
uniform boundedness of the unscaled variables, and W3 gives
\[
 \max_{i<j}|v(s_{ij})-1|\longrightarrow0.
\]
Hence Lemma~\ref{lem:R4-triangular-Bai-Yin} applies.

For \(H_N^f\), the same variance convergence follows from W3.  To verify
\eqref{eq:R4-BY-tail}, fix \(R\) and use
\(\max_{i<j}|s_{ij}|\to0\).  W5 gives
\[
 \lim_{R\to\infty}\limsup_{N\to\infty}
 \max_{i<j}
 \E\!\left[|Y_{s_{ij}}|^4
 \one_{\{|Y_{s_{ij}}|>R\}}\right]=0.
\]
Lemma~\ref{lem:R4-triangular-Bai-Yin} therefore applies directly to the
target array.  Lemma~\ref{lem:R4-diagonal-truncation} gives the equivalent
explicit truncation formulation and shows that no unrecorded tail assumption
is being used.
\end{proof}

\begin{corollary}[Uniform confinement along the hybrid path]
\label{cor:R4-hybrid-confinement}
Let \(H_N^{(r)}\), \(0\le r\le M_N\), be the centered hybrid matrices in
Definition~\ref{def:hybrid-package}.  For every \(\varepsilon>0\),
\begin{equation}\label{eq:R4-uniform-hybrid-confinement}
 \max_{0\le r\le M_N}
 \mathbb P\!\left(\norm{H_N^{(r)}}_{\op}>2+\varepsilon\right)
 \longrightarrow0.
\end{equation}
This is a uniform-in-\(r\) probability estimate; simultaneous confinement of
all \(M_N+1\) matrices on one event is neither asserted nor needed.
\end{corollary}

\begin{proof}
Every hybrid edge is either \(Y_{s_{ij}}\) or its bounded four-moment
surrogate.  The two choices have exactly the same variance.  Thus
\eqref{eq:R4-BY-variance} is independent of \(r\).  The target edges satisfy
the uniform fourth-tail condition by W5, while the surrogate edges are
uniformly bounded.  Hence the modulus in \eqref{eq:R4-BY-tail} is uniform in
\(r\).  The uniform statement in
Lemma~\ref{lem:R4-triangular-Bai-Yin} gives
\eqref{eq:R4-uniform-hybrid-confinement}.
\end{proof}

\begin{proposition}[Exterior isotropic law for the rough centered array]
\label{prop:R4-rough-isotropic}
Let \(G_N^f(z)=(H_N^f-zI)^{-1}\).  For every compact
\(K\subset\mathbb C\setminus[-2,2]\) and deterministic unit vectors \(u,v\),
\begin{align}
 \sup_{z\in K}\left|
 u^{\mathsf T}G_N^f(z)v-m(z)u^{\mathsf T}v
 \right|&\longrightarrow0,\label{eq:R4-rough-isotropic}\\
 \sup_{z\in K}\left|
 u^{\mathsf T}G_N^f(z)^2v-m'(z)u^{\mathsf T}v
 \right|&\longrightarrow0
 \label{eq:R4-rough-isotropic-derivative}
\end{align}
in probability.  The estimates also hold for the bounded surrogate, and the
convergence is uniform over the four-moment hybrid matrices in
Definition~\ref{def:hybrid-package}.
\end{proposition}

\begin{proof}
We prove the stronger hybrid-uniform statement.  Fix \(\eta>0\) and first
restrict to a compact set
\(K_\eta\subset\{z:|\Im z|\ge\eta\}\).  For
\[
 Q_z(A):=u^{\mathsf T}(A-zI)^{-1}v
\]
and one edge variable entering as
\(A+N^{-1/2}yV_{ab}\), resolvent differentiation gives
\begin{equation}\label{eq:R4-isotropic-edge-derivatives}
 \left|\partial_y^k
 Q_z(A+N^{-1/2}yV_{ab})\right|
 \le C_{k,\eta}N^{-k/2},
 \qquad 1\le k\le5.
\end{equation}
Each derivative inserts one factor \(N^{-1/2}V_{ab}\), while every
resolvent has norm at most \(\eta^{-1}\).  The same bounds, with different
constants, hold after composing the real and imaginary parts of \(Q_z\) with
a bounded \(C^5\) test function.

We spell out the fifth-order remainder, since only fourth moments are
assumed.  Choose the sequence \(\eta_N\downarrow0\) from
Lemma~\ref{lem:R4-diagonal-truncation}, so that
\[
 T_N:=\max_{i<j}\E\left[|Y_{ij}|^4
 \one_{\{|Y_{ij}|>\eta_N\sqrt N\}}\right]
 =o(\eta_N^4).
\]
On \(|Y_{ij}|\le\eta_N\sqrt N\), the per-edge fifth-order Taylor remainder
is bounded by
\begin{equation}\label{eq:R4-isotropic-fifth-remainder}
 C_\eta N^{-5/2}
 \E\left[|Y_{ij}|^5
 \one_{\{|Y_{ij}|\le\eta_N\sqrt N\}}\right]
 \le C_\eta\eta_NN^{-2}\E|Y_{ij}|^4.
\end{equation}
Summing over \(O(N^2)\) edges gives \(O_\eta(\eta_N)=o(1)\).  On the
complement, for \(0\le q\le4\),
\[
 N^{-q/2}\E\left[|Y_{ij}|^q
 \one_{\{|Y_{ij}|>\eta_N\sqrt N\}}\right]
 \le \eta_N^{q-4}N^{-2}T_N.
\]
After summation this is \(o(\eta_N^q)\), including the bounded-test-function
term \(q=0\).  The bounded surrogate contributes
\(O(N^2N^{-5/2})=o(1)\) at fifth order.  The first four Taylor coefficients
cancel exactly by \eqref{eq:uniform-four-match-package}.  Hence replacing any
terminal segment of the hybrid path changes the expectation of the test
function by \(o(1)\), with the error independent of the initial hybrid
index.

The bounded endpoint satisfies Proposition~\ref{prop:R4-bounded-local-law}.
The preceding uniform comparison, applied to smooth approximations of the
indicator of the complement of a small disk around
\(m(z)u^{\mathsf T}v\), therefore proves
\[
 \sup_{0\le r\le M_N}
 \mathbb P\left(
 \left|u^{\mathsf T}(H_N^{(r)}-zI)^{-1}v
 -m(z)u^{\mathsf T}v\right|>\varepsilon
 \right)\longrightarrow0
\]
for each \(z\in K_\eta\).  A finite net and
\(\norm{\partial_zG(z)}_{\op}\le\eta^{-2}\) make the convergence uniform on
\(K_\eta\).

Now let \(K\Subset\mathbb C\setminus[-2,2]\) be arbitrary and put
\(\tau=\operatorname{dist}(K,[-2,2])>0\).
Corollary~\ref{cor:R4-hybrid-confinement}
implies, uniformly in the hybrid index, that the centered spectrum is in
\([-2-\tau/3,2+\tau/3]\) with probability tending to one.  On this event,
for real \(z\in K\) and \(0<\eta<\tau/3\),
\[
 \left|u^{\mathsf T}\bigl(G_r(z)-G_r(z+\ii\eta)\bigr)v\right|
 \le\eta\norm{G_r(z)}_{\op}\norm{G_r(z+\ii\eta)}_{\op}
 \le C\eta\tau^{-2}.
\]
Apply the nonreal result at \(z+\ii\eta\), take \(N\to\infty\), and then
let \(\eta\downarrow0\).  A finite net and the deterministic exterior
Lipschitz bound prove \eqref{eq:R4-rough-isotropic} on all of \(K\), uniformly
over the hybrids.

Finally choose a compact neighborhood
\(K_+\Subset\mathbb C\setminus[-2,2]\) of \(K\).  The first assertion holds
uniformly on \(K_+\), and Cauchy's formula on circles of a fixed radius in
\(K_+\) gives
\[
 \partial_z\bigl[u^{\mathsf T}G_r(z)v\bigr]
 =u^{\mathsf T}G_r(z)^2v.
\]
This proves \eqref{eq:R4-rough-isotropic-derivative} with the same hybrid
uniformity.
\end{proof}

\begin{proposition}[Bulk confinement and one-outlier alternative]
\label{prop:R4-outlier-separation}
Let \(M_N\) denote either \(M_N^{(4)}\) or \(M_N^f\).  For every fixed
\(\delta>0\), all but possibly one eigenvalue of \(M_N\) lie in
\([-2-\delta,2+\delta]\) with probability \(1-o(1)\).  More precisely:
\begin{enumerate}[label=(\roman*)]
\item if \(|\theta|\le1\), then
\(\spec(M_N)\subset[-2-\delta,2+\delta]\) with probability \(1-o(1)\);
\item if \(|\theta|>1\), then there is one eigenvalue on the side determined
by \(\theta\) such that
\begin{equation}\label{eq:R4-outlier-convergence}
 \lambda_{\rm out}(M_N)\longrightarrow
 \rho_\theta=\theta+\theta^{-1}.
\end{equation}
All remaining eigenvalues lie in \([-2-\delta,2+\delta]\).  In particular,
if \(0<\delta<|\rho_\theta|-2\), the outlier is the unique eigenvalue outside
that interval with probability \(1-o(1)\).
\end{enumerate}
\end{proposition}

\begin{proof}
The diagonal term and mean remainder in \eqref{eq:R4-mean-decomp} have
operator norm \(o(1)\).  It is therefore enough to treat
\[
 \widehat M_N:=H_N+\theta xx^{\mathsf T},
\]
where \(H_N\) is either centered array.  By the matrix determinant lemma,
for \(z\notin\spec(H_N)\),
\begin{equation}\label{eq:R4-characteristic-determinant}
 \frac{\det(\widehat M_N-zI)}{\det(H_N-zI)}
 =F_N(z):=1+\theta x^{\mathsf T}(H_N-zI)^{-1}x.
\end{equation}
Propositions~\ref{prop:R4-bounded-local-law} and
\ref{prop:R4-rough-isotropic} give, uniformly on compact subsets of
\(\mathbb C\setminus[-2,2]\),
\begin{equation}\label{eq:R4-characteristic-limit}
 F_N(z)\longrightarrow F(z):=1+\theta m(z)
\end{equation}
in probability.

For real \(E>2\), \(m(E)\in(-1,0)\); for \(E<-2\),
\(m(E)\in(0,1)\).  Hence \(F\) has a real zero outside \([-2,2]\) exactly
when \(|\theta|>1\).  The zero is simple and equals
\[
 \rho_\theta=\theta+\theta^{-1}.
\]
If \(|\theta|>1\), choose a small circle \(\mathcal C\) around
\(\rho_\theta\), disjoint from \([-2,2]\).  Uniform convergence in
\eqref{eq:R4-characteristic-limit} and Rouch\'e's theorem show that \(F_N\)
has exactly one zero inside \(\mathcal C\).  By
\eqref{eq:R4-characteristic-determinant}, this zero is the separated
eigenvalue of \(\widehat M_N\).

To exclude further separated eigenvalues, fix \(\delta>0\) and a large
\(R\).  On the compact set
\[
 \{z:|z|\le R,\ \operatorname{dist}(z,[-2,2])\ge\delta\}
\]
with the circle \(\mathcal C\) removed in the supercritical case, \(F\) is
bounded away from zero.  Equation \eqref{eq:R4-characteristic-limit} therefore
excludes zeros of \(F_N\) there.  For \(|z|>R\), centered confinement and the
bound
\[
 \left|\theta x^{\mathsf T}(H_N-zI)^{-1}x\right|
 \le\frac{|\theta|}{|z|-\norm{H_N}_{\op}}
\]
is smaller than one for \(R\) sufficiently large.  Thus no further zeros
occur.  Equivalently, one may use rank-one interlacing to see directly that a
positive spike can create at most one eigenvalue above the upper edge and a
negative spike at most one below the lower edge
\cite{BenaychGeorgesNadakuditi2011}.  Finally, Weyl's inequality transfers the
conclusions from \(\widehat M_N\) to \(M_N\), because
\(\norm{-\theta D_x+R_N^{(m)}}_{\op}=o(1)\).
\end{proof}

\begin{remark}[No extra edge assumption]
\label{rem:R4-no-extra-edge-assumption}
Propositions~\ref{prop:R4-centered-confinement} and
\ref{prop:R4-rough-isotropic} use only W3, W5, exact four-moment matching, and
the spike assumptions already stated in Section~\ref{sec:model-main}.  No
subexponential tail, \(4+\varepsilon\) moment, or local law for the original
rough transform has been added.
\end{remark}

\begin{lemma}[A globally smooth representative of the bulk statistic]
\label{lem:R4-smooth-cutoff}
Fix \(\Gamma\) satisfying (G1)--(G4), and let
\(\varphi_1,\ldots,\varphi_k\) be analytic near
\(\overline{D_\Gamma}\).  There exist
\(\widetilde\varphi_\ell\in C_c^\infty(\mathbb R)\) such that, for both
\(M_N^{(4)}\) and \(M_N^f\),
\begin{equation}\label{eq:R4-cutoff-equality}
 \Tr\widetilde\varphi_\ell(M_N)
 =\mathcal L_{M_N}^\Gamma(\varphi_\ell)
\end{equation}
with probability \(1-o(1)\).  In the supercritical case,
\(\widetilde\varphi_\ell\) vanishes on a fixed neighborhood of
\(\rho_\theta\).
\end{lemma}

\begin{proof}
Choose \(\delta>0\) so that
\([-2-\delta,2+\delta]\subset D_\Gamma\) and, in the supercritical case, a
small interval around \(\rho_\theta\) is disjoint from
\(\overline{D_\Gamma}\).  Extend the real-axis restriction of \(\varphi_\ell\)
smoothly and compactly so that it agrees with \(\varphi_\ell\) on the bulk
interval and vanishes near the outlier and outside a larger compact set.
Proposition~\ref{prop:R4-outlier-separation} then gives
\eqref{eq:R4-cutoff-equality}.
\end{proof}

\section{Four-Moment Replacement and Full-Trace Completion}
\label{sec:R5-package}

\subsection{Global edge derivatives and four-moment cancellation}

The replacement observable is built from the smooth functions in
Lemma~\ref{lem:R4-smooth-cutoff}, rather than directly from a resolvent
defined only on a spectral confinement event.  This makes every derivative
bound deterministic and global.

\begin{lemma}[Fourier--Duhamel trace derivative bound]
\label{lem:R5-Fourier-Duhamel}
Let \(h\in C_c^\infty(\mathbb R)\), let \(A=A^{\mathsf T}\), and let
\(B=B^{\mathsf T}\).  For every integer \(r\ge1\),
\begin{equation}\label{eq:R5-trace-derivative}
 \sup_{u\in\mathbb R}
 \left|\frac{\dd^r}{\dd u^r}\Tr h(A+uB)\right|
 \le
 C_{h,r}\norm{B}_1\norm{B}_{\op}^{r-1},
\end{equation}
where
\begin{equation}\label{eq:R5-Fourier-constant}
 C_{h,r}:=\frac1{2\pi}\int_{\mathbb R}
 |t|^r|\widehat h(t)|\,\dd t<\infty.
\end{equation}
The constant is independent of the dimension and of \(A\).
\end{lemma}

\begin{proof}
Fourier functional calculus gives
\[
 h(A+uB)=\frac1{2\pi}\int_{\mathbb R}
 \widehat h(t)e^{\ii t(A+uB)}\,\dd t.
\]
The \(r\)-fold Duhamel formula expresses the \(u\)-derivative of the
exponential as \(r!\) times an integral over an \(r\)-simplex of products of
unitaries with \(r\) inserted factors \(B\).  The simplex has volume
\(|t|^r/r!\).  By cyclicity of trace and the trace-ideal inequality,
\[
 |\Tr(U_0BU_1B\cdots BU_r)|
 \le\norm{B}_1\norm{B}_{\op}^{r-1},
\]
because every \(U_j\) is unitary.  Integrating against
\(|\widehat h(t)|\) proves \eqref{eq:R5-trace-derivative}.
\end{proof}

\begin{corollary}[Uniform edge derivatives]
\label{cor:R5-edge-derivatives}
For \(V_e=E_{ab}+E_{ba}\), define
\[
 L_{\ell,e}(u)
 :=\Tr\widetilde\varphi_\ell
 \bigl(A+N^{-1/2}uV_e\bigr),
\]
where \(A=A^{\mathsf T}\) is arbitrary.  Then, for \(1\le r\le5\),
\begin{equation}\label{eq:R5-L-derivatives}
 \sup_{u\in\mathbb R}|L_{\ell,e}^{(r)}(u)|
 \le C_{\ell,r}N^{-r/2}.
\end{equation}
\end{corollary}

\begin{proof}
The two nonzero singular values of \(V_e\) equal one, so
\[
 \norm{N^{-1/2}V_e}_1=2N^{-1/2},
 \qquad
 \norm{N^{-1/2}V_e}_{\op}=N^{-1/2}.
\]
Apply Lemma~\ref{lem:R5-Fourier-Duhamel}.
\end{proof}

Fix \(t=(t_1,\ldots,t_k)\in\mathbb R^k\) and deterministic centerings
\(c_{\ell,N}\).  For an open-edge matrix \(A\), put
\begin{equation}\label{eq:R5-Phi}
 \Phi_{e,A}(u)
 :=
 \exp\left\{
 \ii\sum_{\ell=1}^kt_\ell
 \bigl(L_{\ell,e}(u)-c_{\ell,N}\bigr)
 \right\}.
\end{equation}

\begin{lemma}[Global fifth derivative of the edge observable]
\label{lem:R5-Phi-derivatives}
For \(0\le r\le5\),
\begin{equation}\label{eq:R5-Phi-bound}
 \sup_{A=A^{\mathsf T}}\sup_{u\in\mathbb R}
 |\Phi_{e,A}^{(r)}(u)|
 \le C_{t,\widetilde\varphi}N^{-r/2}.
\end{equation}
In particular, the hypotheses
\eqref{eq:derivative-input-package}--\eqref{eq:origin-derivatives-package}
of Lemma~\ref{lem:four-moment-cancel-package} hold uniformly for every
open-edge matrix and every hybrid.
\end{lemma}

\begin{proof}
The case \(r=0\) follows from \(|\Phi_{e,A}|=1\).  For \(r\ge1\),
Fa\`a di Bruno's formula writes \(\Phi_{e,A}^{(r)}\) as a finite sum of
terms
\[
 \Phi_{e,A}(u)
 \prod_{q=1}^r
 \left(\sum_{\ell=1}^kt_\ell
 L_{\ell,e}^{(q)}(u)\right)^{n_q},
 \qquad
 \sum_{q=1}^r qn_q=r.
\]
Corollary~\ref{cor:R5-edge-derivatives} bounds such a product by
\(C N^{-\sum qn_q/2}=CN^{-r/2}\), uniformly in \(A,u,e\).
\end{proof}

The following lemma combines exact algebraic moment cancellation with the
Lindeberg remainder estimate needed along the replacement path.

\begin{lemma}[Four-moment cancellation with a Lindeberg remainder]
\label{lem:four-moment-cancel-package}
Let $X_e$ and $\widehat X_e$ be centered variables satisfying
\begin{equation}\label{eq:four-match-lemma-package}
 \E X_e^q=\E\widehat X_e^q,
 \qquad q=1,2,3,4.
\end{equation}
Let $\Phi_e:\R\to\mathbb C$ satisfy $|\Phi_e(u)|\le1$.  Suppose that
there are $\varepsilon_0>0$ and a constant $C$, independent of
$e,N,\varepsilon$, such that, for $0<\varepsilon\le\varepsilon_0$,
\begin{equation}\label{eq:derivative-input-package}
 \sup_{|u|\le\varepsilon\sqrt N}
 |\Phi_e^{(5)}(u)|\le C N^{-5/2}
\end{equation}
and
\begin{equation}\label{eq:origin-derivatives-package}
 |\Phi_e^{(q)}(0)|\le C N^{-q/2},
 \qquad q=0,1,2,3,4.
\end{equation}
Assume that the target variables satisfy the averaged fourth-moment Lindeberg
condition and that $\sup_e|\widehat X_e|\le C_*$.  Then
\begin{equation}\label{eq:four-comparison-error-package}
 \lim_{\varepsilon\downarrow0}
 \limsup_{N\to\infty}
 \sum_e
 \left|
 \E\Phi_e(X_e)-\E\Phi_e(\widehat X_e)
 \right|=0,
\end{equation}
provided the functions $\Phi_e$ are the conditional edge observables along
one telescoping path and the derivative constants are uniform in the open-edge
matrices.
\end{lemma}

\begin{proof}
Let
\[
 P_{e,4}(u):=\sum_{q=0}^4\frac{\Phi_e^{(q)}(0)}{q!}u^q.
\]
Equation \eqref{eq:four-match-lemma-package} gives the exact identity
\[
 \E\Phi_e(X_e)-\E\Phi_e(\widehat X_e)
 =\E\bigl[\Phi_e(X_e)-P_{e,4}(X_e)\bigr]
 -\E\bigl[\Phi_e(\widehat X_e)-P_{e,4}(\widehat X_e)\bigr].
\]
On $|u|\le\varepsilon\sqrt N$, Taylor's theorem and
\eqref{eq:derivative-input-package} give
\begin{equation}\label{eq:small-remainder-package}
 |\Phi_e(u)-P_{e,4}(u)|
 \le C N^{-5/2}|u|^5
 \le C\varepsilon N^{-2}|u|^4.
\end{equation}
For $|u|>\varepsilon\sqrt N$, boundedness of $\Phi_e$ and
\eqref{eq:origin-derivatives-package} give
\begin{equation}\label{eq:large-remainder-package}
 |\Phi_e(u)-P_{e,4}(u)|
 \le C_\varepsilon N^{-2}|u|^4.
\end{equation}
Indeed, for $0\le q\le4$,
\[
 N^{-q/2}|u|^q
 \le\varepsilon^{q-4}N^{-2}|u|^4,
\]
and the constant term is controlled by
$1\le\varepsilon^{-4}N^{-2}|u|^4$.

After summing \eqref{eq:small-remainder-package}, the small-entry target and
surrogate contributions are $O(\varepsilon)$ because their averaged fourth
moments are uniformly bounded.  The large-entry target contribution tends to
zero for fixed $\varepsilon$ by the averaged fourth-moment Lindeberg
condition and \eqref{eq:large-remainder-package}.  The large-entry surrogate
contribution is eventually zero because the surrogate is uniformly bounded.
First take $N\to\infty$, and then $\varepsilon\downarrow0$.
\end{proof}

\begin{remark}[What this lemma removes]
Exact matching through order four removes the complete first-, second-, third-,
and fourth-order Taylor contributions edge by edge.  In particular, the later
replacement proof no longer needs a weighted third-derivative cancellation, a
fourth-cumulant subsequence argument, or an $L^4$-metric coupling between the
target and surrogate variables.  The replacement is reduced to proving a
globally valid fifth-derivative bound and controlling the associated hybrid
observable.
\end{remark}

\subsection{Telescoping replacement, rough bulk CLT, and outlier contribution}

\begin{theorem}[Exact four-moment replacement for smooth LSS]
\label{thm:R5-smooth-replacement}
Let \(M_N^{(r)}\), \(0\le r\le M_N\), be the hybrids of
Definition~\ref{def:hybrid-package}.  For every finite family
\(\widetilde\varphi_1,\ldots,\widetilde\varphi_k\in C_c^\infty(\mathbb R)\)
and \(t\in\mathbb R^k\),
\begin{align}
 &\E\exp\left\{\ii\sum_{\ell=1}^kt_\ell
 \Tr\widetilde\varphi_\ell(M_N^f)\right\}
 \notag\\
 &\qquad-
 \E\exp\left\{\ii\sum_{\ell=1}^kt_\ell
 \Tr\widetilde\varphi_\ell(M_N^{(4)})\right\}
 \longrightarrow0.
 \label{eq:R5-CF-replacement}
\end{align}
The same conclusion holds after subtracting any common deterministic
centering from the two vectors.
\end{theorem}

\begin{proof}
Order the edges as in Definition~\ref{def:hybrid-package} and telescope the
left side of \eqref{eq:R5-CF-replacement}.  At edge \(e_r\), condition on
the open-edge matrix \(M_N^{(r,e_r)}(0)\).  The two candidate scalar
variables are \(Y_{e_r}\) and \(\widehat Y_{e_r}\), and the conditional
observable is \eqref{eq:R5-Phi}.  Their moments through order four agree
exactly by Theorem~\ref{thm:uniform-four-surrogate-package}.
Lemma~\ref{lem:R5-Phi-derivatives} supplies the global derivative hypotheses
of Lemma~\ref{lem:four-moment-cancel-package}.  The target Lindeberg input is
Lemma~\ref{lem:exact-W5-Lindeberg-package}, and the surrogate is uniformly
bounded by Theorem~\ref{thm:uniform-four-surrogate-package}.  Therefore
Lemma~\ref{lem:four-moment-cancel-package} makes the sum of all telescoping
errors tend to zero.
\end{proof}

\begin{remark}[Why no hybrid local law appears]
The Fourier--Duhamel bound is valid for every finite self-adjoint matrix
\(A\), regardless of its spectrum.  Therefore the comparison proof never
conditions on a common confinement event and never differentiates a stopped
resolvent.  Spectral confinement is used only at the two endpoints, through
Lemma~\ref{lem:R4-smooth-cutoff}.
\end{remark}

\begin{theorem}[Completion of the rough bulk LSS theorem]
\label{thm:R5-rough-bulk}
Under the hypotheses of Theorem~\ref{thm:bulk-main-package}, the bulk CLT
\eqref{eq:bulk-clt-package} holds for \(M_N^f\).
\end{theorem}

\begin{proof}
Choose the smooth representatives
\(\widetilde\varphi_\ell\) from
Lemma~\ref{lem:R4-smooth-cutoff}.  The bounded endpoint has the desired joint
CLT by Corollary~\ref{cor:R4-bounded-full-CLT}.  Theorem
\ref{thm:R5-smooth-replacement} identifies the joint characteristic
functions of the bounded and target smooth LSS vectors after subtraction of
the same centering.  L\'evy's continuity theorem transfers the joint Gaussian
limit.  Finally, \eqref{eq:R4-cutoff-equality} replaces the smooth statistic
by the bulk contour statistic at both endpoints with probability \(1-o(1)\).
\end{proof}

\paragraph{Separated outlier and full trace.}
\label{sec:R6-package}

\begin{lemma}[Outlier test-function convergence]
\label{lem:R6-outlier-function}
Assume \(|\theta|>1\), and let \(\varphi\) be analytic on a neighborhood of
\(\rho_\theta\).  Then
\begin{equation}\label{eq:R6-phi-outlier}
 \varphi\bigl(\lambda_{\rm out}(M_N^f)\bigr)
 =\varphi(\rho_\theta)+o_{\mathbb P}(1).
\end{equation}
\end{lemma}

\begin{proof}
Proposition~\ref{prop:R4-outlier-separation} gives
\(\lambda_{\rm out}(M_N^f)\to\rho_\theta\) in probability.  Analyticity implies
continuity on a compact neighborhood of \(\rho_\theta\), so the continuous
mapping theorem gives \eqref{eq:R6-phi-outlier}.  No outlier fluctuation
theorem is required at the order-one LSS scale.
\end{proof}

\begin{proof}[Proof of Theorem~\ref{thm:R6-supercritical-full}]

By Proposition~\ref{prop:R4-outlier-separation}, with probability \(1-o(1)\)
the spectrum consists of the bulk eigenvalues in \(D_\Gamma\) and one
separated outlier.  Hence
\[
 \Tr\varphi_\ell(M_N^f)
 =
 \mathcal L_{M_N^f}^\Gamma(\varphi_\ell)
 +\varphi_\ell\bigl(\lambda_{\rm out}(M_N^f)\bigr).
\]
Apply Theorem~\ref{thm:R5-rough-bulk} to the first term and
Lemma~\ref{lem:R6-outlier-function} to the second.  Slutsky's theorem proves
\eqref{eq:R6-full-CLT}.
\end{proof}

\begin{remark}[No additional model assumption]
For fixed \(\theta\), the strict inequality \(|\theta|>1\) already gives a
positive deterministic gap between \(\rho_\theta\) and \([-2,2]\).  Thus no
separate assumption \(|\theta|\ge1+\tau\) is needed.  Analyticity near
\(\rho_\theta\) is a test-function domain requirement for the full-trace
corollary, not an assumption on \(f\), \(\nu\), or \(x\).  The critical case
\(|\theta|=1\) belongs to the non-separated full-trace statement of
Corollary~\ref{cor:subcritical-full-package}.
\end{remark}
\section{Profile Criteria and Obstructions}
\label{app:assumption-W-criteria}

This appendix gives concrete sufficient conditions for W1--W5.  The routes
are organized by likelihood-ratio regularity of the noise law, smoothness of
the transform, translation regularity of a density, and local smoothness near
atoms.  They are sufficient, not necessary, and none imposes an additional
approximation or coefficient-stability hypothesis.  The second subsection
records complementary obstructions showing why bare $L^4(\nu)$ does not
replace translated profile and tail information.

\subsection{Criteria for profile-admissibility}

\begin{criterion}[Direct verification route]
\label{crit:direct-profile-verification}
To verify Condition~\ref{cond:profile-admissibility} directly, it is enough to prove the two shifted-profile expansions, shifted fourth-cumulant stability, and the shifted fourth-moment Lindeberg condition appearing in that condition.  No differentiability of \(f\) is used in the main proof after these data have been established.
\end{criterion}

\begin{proposition}[Uniform shifted fourth-moment criterion]
\label{prop:W-TUI}
Suppose that the shifted-profile expansions for \(m_f\) and \(q_f\) hold and that
\[
\lim_{R\to\infty}\limsup_{\delta\downarrow0}
\sup_{|s|\le\delta}
\E\bigl[|f(\zeta+s)|^4{\bf 1}_{\{|f(\zeta+s)|>R\}}\bigr]=0.
\]
Assume also that
\[
        \E |f(\zeta+s)|^4\longrightarrow \E |f(\zeta)|^4
        \qquad(s\to0).
\]
Then the shifted fourth-cumulant stability and shifted fourth-moment Lindeberg parts of Condition~\ref{cond:profile-admissibility} hold.
\end{proposition}

\begin{proof}
The convergence of fourth moments and the profile expansions imply convergence of all lower centered moments of
\[
        Y_s=f(\zeta+s)-m_f(s)
\]
which occur in \(\cum_4(Y_s)=\E Y_s^4-3(\E Y_s^2)^2\).  Since \(\E Y_s^2\to1\) and \(\E Y_s^4\to\E f(\zeta)^4\), one obtains
\[
        \cum_4(Y_s)\to \E f(\zeta)^4-3=\kappa_4^{f,\nu}.
\]
The displayed uniform integrability gives, for the microscopic shifts \(s_{ij}\),
\[
\lim_{R\to\infty}\limsup_{N\to\infty}
\frac{2}{N(N-1)}\sum_{i<j}
\E\bigl[|Y_{s_{ij}}|^4{\bf 1}_{\{|Y_{s_{ij}}|>R\}}\bigr]=0,
\]
because \(\max_{i<j}|s_{ij}|\to0\) by Assumption~\ref{ass:spike}, and centering by \(m_f(s_{ij})=o(1)\) changes the tail threshold only by a bounded deterministic amount.  This is the required shifted Lindeberg condition.
\end{proof}

\paragraph{Likelihood-ratio and score-density criteria.}

Assume that \(\nu\) has a density \(p\) and that the translated law of \(\zeta+s\) is absolutely continuous with respect to \(\nu\) for all sufficiently small \(s\).  Put
\[
        L_s(y)=\frac{p(y-s)}{p(y)}
\]
on \(\{p>0\}\).  Then
\[
        \E f(\zeta+s)=\E_\nu[f(\zeta)L_s(\zeta)],\qquad
        \E f(\zeta+s)^2=\E_\nu[f(\zeta)^2L_s(\zeta)].
\]

\begin{proposition}[Weighted likelihood-ratio criterion]
\label{prop:likelihood-W}
Assume \(f\in L^4(\nu)\), \(\E_\nu f=0\), and
\(\E_\nu f^2=1\).  Suppose that there are measurable functions
\(\ell_1,\ell_2\) such that, as \(s\to0\),
\[
        \E_\nu |f(\zeta)|\,|L_s(\zeta)-1-s\ell_1(\zeta)|=o(|s|),
\]
\[
        \E_\nu |f(\zeta)|^2\left|L_s(\zeta)-1-s\ell_1(\zeta)-\frac{s^2}{2}\ell_2(\zeta)\right|=o(s^2),
\]
and
\[
        \sup_{|s|\le s_0}\E_\nu |f(\zeta)|^4L_s(\zeta)<\infty
\]
with uniform integrability of \(\{|f(\zeta)|^4L_s(\zeta): |s|\le s_0\}\).  Then Condition~\ref{cond:profile-admissibility} holds with
\[
        a_1=\E_\nu[f(\zeta)\ell_1(\zeta)],\qquad
        b_1=\E_\nu[f(\zeta)^2\ell_1(\zeta)],
\]
\[
        b_2=\frac12\E_\nu[f(\zeta)^2\ell_2(\zeta)],\qquad
        \beta_2^{f,\nu}=b_2-a_1^2.
\]
\end{proposition}

\begin{proof}
Since \(\E_\nu f=0\),
\[
        m_f(s)=\E_\nu f(\zeta)L_s(\zeta)=s\E_\nu[f(\zeta)\ell_1(\zeta)]+o(s).
\]
Since \(\E_\nu f^2=1\),
\[
        q_f(s)=\E_\nu f(\zeta)^2L_s(\zeta)
        =1+s\E_\nu[f(\zeta)^2\ell_1(\zeta)]+\frac{s^2}{2}\E_\nu[f(\zeta)^2\ell_2(\zeta)]+o(s^2).
\]
The shifted fourth-moment and Lindeberg requirements follow from the uniform integrability of \(|f|^4L_s\).  Finally,
\[
        \Var(f(\zeta+s))=q_f(s)-m_f(s)^2=1+b_1s+(b_2-a_1^2)s^2+o(s^2).
\]
\end{proof}

\begin{lemma}[Gaussian weighted translation and shifted-tail estimates]
\label{lem:Gaussian-L4epsilon-W}
Let \(\nu=\gamma\) be standard Gaussian measure.  If, for some \(\epsilon>0\),
\[
        f\in L^{4+\epsilon}(\gamma),
        \qquad \E f(\xi)=0,
        \qquad \E f(\xi)^2=1,
\]
put $L_s(\xi)=\exp(s\xi-s^2/2)$.  Then
\begin{align}
 \E\!\left[|f(\xi)|\,|L_s(\xi)-1-s\xi|\right]
 &=o(|s|),\label{eq:Gaussian-weighted-first}\\
 \E\!\left[|f(\xi)|^2
 \left|L_s(\xi)-1-s\xi-\frac{s^2}{2}(\xi^2-1)\right|\right]
 &=o(s^2).\label{eq:Gaussian-weighted-second}
\end{align}
Moreover, for every fixed $s_0>0$,
\begin{equation}\label{eq:Gaussian-weighted-tail}
 \lim_{R\to\infty}\sup_{|s|\le s_0}
 \E\!\left[|f(\xi+s)|^4
 \one_{\{|f(\xi+s)|>R\}}\right]=0.
\end{equation}
Consequently Proposition~\ref{prop:likelihood-W} applies with
$\ell_1(\xi)=\xi$ and $\ell_2(\xi)=\xi^2-1$.
\end{lemma}

\begin{proof}
The Gaussian translation likelihood ratio is
\[
        L_s(\xi)=\exp(s\xi-s^2/2)
        =1+s\xi+\frac{s^2}{2}(\xi^2-1)+o_{L^q}(s^2)
\]
for every finite $q$.  H\"older's inequality, paired with $f$ and $f^2$,
gives \eqref{eq:Gaussian-weighted-first}--\eqref{eq:Gaussian-weighted-second}.
For the tail estimate, take $p=(4+\epsilon)/4>1$ and conjugate exponent $q$.
Uniformly for $|s|\le s_0$,
\begin{align*}
&\E\bigl[|f(\xi+s)|^4{\bf 1}_{\{|f(\xi+s)|>R\}}\bigr]\\
&\quad=\E\bigl[|f(\xi)|^4{\bf 1}_{\{|f(\xi)|>R\}}L_s(\xi)\bigr]\\
&\quad\le
\left(\E\bigl[|f(\xi)|^{4+\epsilon}{\bf 1}_{\{|f(\xi)|>R\}}\bigr]\right)^{1/p}
\left(\E L_s(\xi)^q\right)^{1/q}.
\end{align*}
The second factor is uniformly bounded and the first
tends to zero, proving \eqref{eq:Gaussian-weighted-tail}.  These are exactly
the weighted expansion and uniform-integrability hypotheses of
Proposition~\ref{prop:likelihood-W}.
\end{proof}

\begin{corollary}[Score-density sufficient condition]
\label{cor:score-W}
Assume \(f\in L^4(\nu)\), \(\E_\nu f=0\),
\(\E_\nu f^2=1\), and that \(p\) is twice weakly differentiable.  In the
weighted norms
\[
        \|u\|_{f,1}:=\E_\nu |f(\zeta)|\,|u(\zeta)|,
        \qquad
        \|u\|_{f^2,1}:=\E_\nu |f(\zeta)|^2|u(\zeta)|,
\]
one has
\[
        \left\|\frac{p(\cdot-s)}{p(\cdot)}-1+s\frac{p'}{p}\right\|_{f,1}=o(|s|)
\]
and
\[
        \left\|\frac{p(\cdot-s)}{p(\cdot)}-1+s\frac{p'}{p}-\frac{s^2}{2}\frac{p''}{p}\right\|_{f^2,1}=o(s^2).
\]
Assume also \(\sup_{|s|\le s_0}\int |f(y)|^4p(y-s)\,dy<\infty\), with uniform integrability as \(s\to0\).  Then the weighted likelihood-ratio criterion applies with
\[
        \ell_1=-\frac{p'}{p},\qquad \ell_2=\frac{p''}{p}.
\]
For the Gaussian density, \(\ell_1(y)=y\) and \(\ell_2(y)=y^2-1\).
\end{corollary}

\begin{proof}
The identity \(L_s(y)=p(y-s)/p(y)\) and the two weighted Taylor remainder assumptions give exactly the first- and second-order likelihood-ratio expansions in Proposition~\ref{prop:likelihood-W}.  The fourth-moment assumption is equivalent to uniform integrability of \(|f|^4L_s\).  The conclusion follows from Proposition~\ref{prop:likelihood-W}.
\end{proof}

\paragraph{Smooth and bounded rough-density criteria.}

\begin{proposition}[Smooth-transform criterion]
\label{prop:smooth-transform-W}
Assume \(f\in C^2(\mathbb R)\), \(\E_\nu f=0\),
\(\E_\nu f^2=1\),
\[
        \E\sup_{|u|\le\delta_0}|f(\zeta+u)|^4<\infty,
\]
and that the Taylor remainders of \(f\) and \(f^2\) satisfy
\[
        \E\sup_{|u|\le\delta_0}|f''(\zeta+u)|<\infty,
        \qquad
        \E\sup_{|u|\le\delta_0}\bigl(|f'(\zeta+u)|^2+|f(\zeta+u)f''(\zeta+u)|\bigr)<\infty.
\]
Then Condition~\ref{cond:profile-admissibility} holds with
\[
        a_1=\E f'(\zeta),
        \qquad
        b_1=2\E[f(\zeta)f'(\zeta)],
\]
\[
        b_2=\E\bigl[f'(\zeta)^2+f(\zeta)f''(\zeta)\bigr].
\]
\end{proposition}

\begin{proof}
Taylor's theorem and the domination assumptions give
\[
        f(\zeta+s)=f(\zeta)+sf'(\zeta)+o_{L^1}(s),
\]
and
\[
        f(\zeta+s)^2=f(\zeta)^2+2sf(\zeta)f'(\zeta)
        +s^2\bigl(f'(\zeta)^2+f(\zeta)f''(\zeta)\bigr)+o_{L^1}(s^2).
\]
These formulas give the two shifted-profile expansions.  The local fourth-moment domination gives the shifted fourth-moment convergence and uniform integrability required by Proposition~\ref{prop:W-TUI}.
\end{proof}

\paragraph{Bounded rough transforms under smooth densities.}

\begin{proposition}[Bounded rough transform criterion]
\label{prop:bounded-rough-W}
Assume \(f\in L^\infty\), \(\E_\nu f=0\), \(\E_\nu f^2=1\), and assume the density \(p\) satisfies the translation expansions
\[
        \|p(\cdot-s)-p+s p'\|_{L^1}=o(|s|),
\]
\[
        \left\|p(\cdot-s)-p+s p'-\frac{s^2}{2}p''\right\|_{L^1}=o(s^2).
\]
Then Condition~\ref{cond:profile-admissibility} holds with
\[
        a_1=-\int f(y)p'(y)\,dy,
        \qquad
        b_1=-\int f(y)^2p'(y)\,dy,
\]
\[
        b_2=\frac12\int f(y)^2p''(y)\,dy.
\]
\end{proposition}

\begin{proof}
Since \(f\) is bounded,
\[
        m_f(s)=\int f(y)p(y-s)\,dy,
        \qquad
        q_f(s)=\int f(y)^2p(y-s)\,dy.
\]
The two \(L^1\)-translation expansions imply the profile expansions by multiplying by the bounded functions \(f\) and \(f^2\) and integrating.  Boundedness of \(f\) gives shifted fourth-moment convergence and makes the shifted Lindeberg condition automatic.
\end{proof}

\begin{corollary}[Threshold transforms]
\label{cor:threshold-W}
Let \(f\) be a centered and normalized finite linear combination of indicators of half-lines.  If the density \(p\) has the required one-sided differentiability and second-order Taylor expansions at the relevant thresholds, or more generally satisfies the \(L^1\)-translation assumptions of Proposition~\ref{prop:bounded-rough-W}, then Condition~\ref{cond:profile-admissibility} holds.  The coefficients are determined by the corresponding values of \(p\) and \(p'\) at the thresholds.
\end{corollary}

\paragraph{Atomic criteria and non-equivalence.}

\begin{proposition}[Finite atomic noise]
\label{prop:atomic-W}
Let \(\nu=\sum_{r\in\mathcal A}p_r\delta_r\) have finite support.  Assume
\(\E_\nu f=0\), \(\E_\nu f^2=1\), and that \(f\) is \(C^2\) in a
neighborhood of every atom, with the neighborhoods chosen so that the shifted
values \(r+s\) remain inside them for all sufficiently small \(|s|\).  Then
Condition~\ref{cond:profile-admissibility} holds with
\[
        a_1=\sum_{r\in\mathcal A}p_rf'(r),
\]
\[
        b_1=2\sum_{r\in\mathcal A}p_rf(r)f'(r),
\]
and
\[
        b_2=\sum_{r\in\mathcal A}p_r\bigl(f'(r)^2+f(r)f''(r)\bigr).
\]
\end{proposition}

\begin{proof}
The shifted profiles are finite sums:
\[
        m_f(s)=\sum_{r\in\mathcal A}p_rf(r+s),
        \qquad
        q_f(s)=\sum_{r\in\mathcal A}p_rf(r+s)^2.
\]
Taylor expansion at each atom gives the stated coefficients.  Since there are only finitely many atoms and \(f\) is locally bounded near each of them, the shifted fourth-moment and Lindeberg conditions are automatic.
\end{proof}

\paragraph{Summary and non-equivalence.}

The criteria above are not equivalent.  The likelihood-ratio criterion controls rough transforms through regularity of the noise law.  The smooth-transform criterion controls arbitrary noise laws through regularity of \(f\).  The bounded rough-transform criterion preserves the low-regularity philosophy most closely, but requires regular density translations.  The atomic criterion shows that discrete noise laws are admissible only when \(f\) has enough pointwise regularity near the atoms.  These routes justify Condition~\ref{cond:profile-admissibility} in broad regimes while Subsection~\ref{app:bare-L4-obstructions} shows that bare \(L^4(\nu)\) alone is insufficient.
\subsection{Obstructions and scope}
\label{app:bare-L4-obstructions}
\label{app:counterexample-W}

This subsection records obstructions to replacing Condition~\ref{cond:profile-admissibility} by the bare assumption \(f\in L^4(\nu)\).  A finite unshifted fourth moment does not by itself give a finite-covariance transformed-spiked LSS theorem.  In the Wigner setting there is also an atomic obstruction: if \(\nu\) is atomic, the shifted transform \(f(\zeta+s)\) depends on values of \(f\) away from the atoms, which are not controlled by the \(L^4(\nu)\)-class of \(f\).

The next example shows that small shifted \(L^4\) distance alone does not control a coherent low-rank mean error.  It does not obstruct the construction in Subsection~\ref{sec:R2-package}, because that construction matches all first four centered moments before comparison.

\begin{proposition}[Metric transfer does not control deterministic response]
\label{prop:E-metric-not-response}
There exist a profile-admissible bounded target transform \(f\), a bounded transform \(h\) with the same unshifted centering and normalization, and a delocalized spike sequence such that
\begin{equation}\label{eq:E-metric-zero}
        \frac{2}{N(N-1)}\sum_{i<j}
        \E|h(\zeta+s_{ij})-f(\zeta+s_{ij})|^4\longrightarrow0,
\end{equation}
while the deterministic mean-matrix difference converges to a nonzero rank-one perturbation and has a nonvanishing analytic LSS response.
\end{proposition}

\begin{proof}
Let \(\nu=\frac12\delta_{-1}+\frac12\delta_1\), and choose a bounded continuous \(f\) satisfying
\[
        f(-1+s)=-1+s,
        \qquad
        f(1+s)=1+s,
        \qquad |s|\le\tfrac14.
\]
Then \(f(\zeta)=\zeta\), \(m_f(s)=s\), \(\E f(\zeta+s)^2=1+s^2\), and \(\Var(f(\zeta+s))=1\).  The centered shifted variable is identically \(\zeta\), so W4 and W5 hold automatically.

Let \(\psi\) be bounded and continuous with
\[
        \psi(-1)=\psi(1)=0,
        \qquad
        \psi(-1+s)=\psi(1+s)=|s|,
        \qquad |s|\le\tfrac14,
\]
and set \(h=f+c\psi\) for fixed \(c\ne0\).  Because \(h=f\) on the support of \(\nu\), the two transforms have identical unshifted laws and zero unshifted \(L^4(\nu)\) distance.

Take \(x_i=N^{-1/2}\).  Then \(s_{ij}=s_N=\sqrt\lambda/\sqrt N>0\) and
\[
        h(\zeta+s_N)-f(\zeta+s_N)=cs_N
\]
deterministically.  The left side of \eqref{eq:E-metric-zero} is therefore \(c^4\lambda^2N^{-2}\).  On the other hand, the zero-diagonal deterministic mean difference is
\[
        \Delta A_N(i,j)=N^{-1/2}cs_N=\frac{c\sqrt\lambda}{N},
        \qquad i\ne j,
\]
so
\[
        \Delta A_N=\frac{c\sqrt\lambda}{N}
        ({\bf 1}{\bf 1}^{\mathsf T}-I).
\]
Its leading eigenvalue tends to \(c\sqrt\lambda\).  If \(W_N\) is any centered Wigner background and \(\varphi(u)=u^2\), independence and centering give
\[
        \E\bigl[\Tr(W_N+\Delta A_N)^2-\Tr W_N^2\bigr]
        =\Tr(\Delta A_N)^2\longrightarrow c^2\lambda\ne0.
\]
Thus the analytic deterministic response does not vanish.
\end{proof}

\paragraph{Atomic and shifted-tail obstructions.}
These examples explain why W2--W5 contain translated information rather than following automatically from bare \(L^4(\nu)\).

Suppose \(\nu=\frac12\delta_{-1}+\frac12\delta_1\).  Then an element of \(L^4(\nu)\) only records the two values \(f(-1)\) and \(f(1)\).  The shifted profile, however, uses
\[
        f(-1+s),
        \qquad
        f(1+s).
\]
Therefore two functions that agree \(\nu\)-almost surely can have different shifted profiles.  Thus Condition~\ref{cond:profile-admissibility} is not a property of the bare \(L^4(\nu)\)-equivalence class unless the transform is represented by a function with additional local regularity near the support of \(\nu\).

\begin{proposition}[Bare \(L^4(\nu)\) is not intrinsic for atomic noise]
Let \(\nu=\frac12\delta_{-1}+\frac12\delta_1\).  There exist functions \(f,g\) with \(f=g\) \(\nu\)-almost surely and with the same centered variance-one normalization under \(\nu\), but such that
\[
        \E f(\zeta+s)\ne \E g(\zeta+s)
\]
for every sufficiently small nonzero \(s\).
\end{proposition}

\begin{proof}
Choose \(f=g\) at \(-1\) and \(1\), but redefine \(g\) in neighborhoods of \(-1\) and \(1\) away from the atoms.  This does not change the \(L^4(\nu)\)-class, but it changes \(g(-1+s)\) and \(g(1+s)\).  Hence the shifted mean changes.
\end{proof}

\paragraph{Shifted fourth-moment obstruction.}

Even for continuous noise, unshifted \(L^4(\nu)\) may fail to control shifted fourth moments.  In particular, for \(\nu=\gamma\) there exists a centered variance-one function \(f\in L^4(\gamma)\) such that
\[
        \E|f(\xi+s)|^4=\infty
\]
for every nonzero \(s\) in a punctured neighborhood of zero.  In the Wigner-profile language, this means that the shifted fourth-moment Lindeberg condition in Condition~\ref{cond:profile-admissibility} fails although the unshifted fourth moment is finite.

\begin{proposition}[Failure of shifted fourth moments]
There exists a probability law \(\nu\) with a smooth positive density and a centered variance-one transform \(f\in L^4(\nu)\) such that
\[
        \E|f(\zeta+s)|^4=\infty
\]
for all sufficiently small nonzero \(s\).
\end{proposition}

\begin{proof}
Take $\nu=\gamma$ and first define
\[
 g(t):=\frac{\exp(t^2/8)}{(1+t^2)^{1/4}}.
\]
Then
\[
 \E g(\xi)^4
 =\frac{1}{\sqrt{2\pi}}\int_{\mathbb R}\frac{1}{1+t^2}\,\dd t<\infty.
\]
For every $s\ne0$, however,
\begin{align*}
 \E g(\xi+s)^4
 &=\int_{\mathbb R}g(t)^4\frac{e^{-(t-s)^2/2}}{\sqrt{2\pi}}\,\dd t\\
 &=\frac{e^{-s^2/2}}{\sqrt{2\pi}}
   \int_{\mathbb R}\frac{e^{st}}{1+t^2}\,\dd t
 =\infty.
\end{align*}
The divergence occurs on the right tail when $s>0$ and on the left tail when
$s<0$.  Since $g\in L^4(\gamma)$, its mean and variance are finite, and the
variance is positive.  Therefore
$f=(g-\E g(\xi))/\sqrt{\Var(g(\xi))}$ is centered, has variance one, belongs to
$L^4(\gamma)$, and retains the same shifted fourth-moment divergence: on the
tails where $g\ge2\E g(\xi)$, one has
$|g-\E g(\xi)|^4\ge g^4/16$.
\end{proof}

\paragraph{Quadratic LSS failure and interpretation.}

The obstruction is visible already at the level of second moments of the
quadratic statistic.  If shifted fourth moments fail along a non-negligible
family of edge shifts, then the centered quadratic statistic contains terms
of the form
\[
        \sum_{i<j}\left( f(\zeta_{ij}+s_{ij})^2-\E f(\zeta+s_{ij})^2\right).
\]
When the corresponding fourth moments are infinite, the variance of this
quadratic statistic is infinite or undefined.  This prevents an
\(L^2\)-strengthened LSS theorem or convergence of the prelimit covariances,
but infinite prelimit variance alone does not logically exclude weak
convergence in distribution.

\begin{proposition}[Quadratic-statistic moment obstruction]
Suppose that, for at least one edge shift \(s_{ij}\), the variable
\(f(\zeta+s_{ij})^2\) has infinite variance.  Then
\(\Tr(M_N^f)^2\) has infinite variance.  Consequently, an analytic LSS
theorem for \(\varphi(t)=t^2\) cannot additionally assert \(L^2\) convergence
or convergence of the prelimit variance to the finite Wigner-profile
covariance.
\end{proposition}

\begin{proof}
For the zero-diagonal model,
\[
        \Tr (M_N^f)^2
        =2N^{-1}\sum_{i<j} f(\zeta_{ij}+s_{ij})^2.
\]
The summands are independent.  If a nonempty subfamily has infinite variance,
then the displayed sum has infinite variance.  Therefore no conclusion that
includes finite prelimit variance, convergence of variances to
\(\mathcal V_{\kappa_4^{f,\nu}}^\Gamma(t^2,t^2)\), or convergence in
\(L^2\) can hold.  This proves the stated moment obstruction.
\end{proof}

\paragraph{Interpretation.}

These examples do not prove that Condition~\ref{cond:profile-admissibility}
is necessary in an exact logical sense.  They show that bare unshifted
\(L^4(\nu)\) does not determine the shifted model and does not guarantee the
moment bounds used by the finite-covariance proof.  A theorem with the present
centering, covariance identification, and replacement method therefore needs
shifted-profile information, shifted fourth-moment control, or concrete
sufficient conditions implying both.  A different weak-convergence theorem
under weaker tails is not excluded.

For Gaussian noise this distinction is sharp at the level of the assumptions
used in the present proof.  Lemma~\ref{lem:Gaussian-L4epsilon-W} shows that
$L^{4+\epsilon}(\gamma)$, and hence polynomial growth, is sufficient for the
uniform condition W5.  The shifted fourth-moment example above shows that
bare $L^4(\gamma)$ need not be sufficient.  This does not make the
$4+\epsilon$ exponent necessary: a transform may satisfy W5 directly without
belonging to any $L^{4+\epsilon}(\gamma)$, and an averaged shifted-tail theorem
would constitute a different, Gaussian-specific refinement.

\end{document}